\documentclass[12pt,a4paper,halfparskip]{scrartcl}

\usepackage[english]{babel}
\usepackage{amsmath}
\usepackage{amssymb}
\usepackage{ulem}
\usepackage{natbib}
\usepackage{comment}
\usepackage{url}
\usepackage{color}
\usepackage{hyperref}
\usepackage[pdftex]{graphicx}
\usepackage{mathdots}
\usepackage{amsthm}
\usepackage{bbm}
\usepackage{enumitem}

\usepackage{framed}

\usepackage[T1]{fontenc}
\usepackage[utf8x]{inputenc}

\DeclareMathOperator{\argmin}{argmin}

\begin{document}


\setlength\parindent{0pt}

\newcommand{\Rho}{P}
\newcommand{\IN}{\mathbb{N}}
\newcommand{\IQ}{\mathbb{Q}}
\newcommand{\IZ}{\mathbb{Z}}
\newcommand{\IR}{\mathbb{R}}
\newcommand{\IC}{\mathbb{C}}
\newcommand{\Ima}{\mbox{Im}}
\newcommand{\dif}{\ \mbox{d}}
\newcommand{\cov}{\mbox{cov}}

\newcommand{\Lp}{\mathcal{L}}
\newcommand{\sI}{\mathcal{I}}
\newcommand{\sA}{\mathcal{A}}
\newcommand{\sB}{\mathcal{B}}
\newcommand{\sP}{\mathcal{P}}
\newcommand{\sE}{\mathcal{E}}
\newcommand{\sF}{\mathcal{F}}
\newcommand{\sL}{\mathcal{L}}
\newcommand{\sG}{\mathcal{G}}
\newcommand{\sH}{\mathcal{H}}
\newcommand{\sT}{\mathcal{T}}
\newcommand{\sV}{\mathcal{V}}

\renewcommand{\Re}{\mbox{Re }}
\renewcommand{\Im}{\mbox{Im }}

\newcommand{\reff}[1]{(\ref{#1})}

\newcommand{\IP}{\mathbb{P}}
\newcommand{\IE}{\mathbb{E}}
\newcommand{\Ii}{\mathbbm{1}}
\newcommand{\supp}{\mbox{supp}}
\newcommand{\Hess}{\mbox{Hess}}
\newcommand{\Var}{\mbox{Var}}
\newcommand{\sX}{\mathcal{X}}
\newcommand{\Kov}{\mbox{Kov}}
\newcommand{\Cov}{\mbox{Cov}}
\newcommand{\tr}{\mbox{tr}}
\newcommand{\gdw}{\Leftrightarrow}
\newcommand{\pto}{\overset{p}{\to}}
\newcommand{\fsto}{\overset{a.s.}{\to}}
\newcommand{\dto}{\overset{d}{\to}}
\newcommand{\lto}{\overset{L^2}{\to}}
\newcommand{\sD}{\mathcal{D}}
\newcommand{\iid}{\overset{\mbox{iid}}{\sim}}
\renewcommand{\l}{\ell}

\newcommand{\lima}{\mbox{l.i.m.}}

\renewcommand{\supp}{\text{supp}}

\newcommand{\err}{\mbox{err}}
\newcommand{\bias}{\mbox{bias}}

\newcommand{\norm}[1]{\left\lVert#1\right\rVert}

\newtheorem{theorem}{Theorem}[section]
\newtheorem{corollary}[theorem]{Corollary}
\newtheorem{definition}[theorem]{Definition}
\newtheorem{proposition}[theorem]{Proposition}
\newtheorem{lemma}[theorem]{Lemma}
\newtheorem{remark}[theorem]{Remark}
\newtheorem{example}[theorem]{Example}
\newtheorem{assumption}[theorem]{Assumption}

\newcommand{\note}[1]
{$^{(!)}$\marginpar[{\hfill\tiny{\sf{#1}}}]{\tiny{\sf{(!) #1}}}}


\title{Towards a general theory for non-linear locally stationary processes}
\author{Rainer Dahlhaus, Stefan Richter, Wei Biao Wu}
\maketitle

\begin{abstract}

   In this paper some general theory is presented for locally stationary processes based on the stationary approximation and the stationary derivative. Laws of large numbers, central limit theorems as well as deterministic and stochastic bias expansions are proved for processes obeying an expansion in terms of the stationary approximation and derivative. In addition it is shown that this applies to some general nonlinear non-stationary Markov-models. In addition the results are applied to derive the asymptotic properties of maximum likelihood estimates of parameter curves in such models.
\end{abstract}

\begin{center}
\end{center}

\section{Introduction}

One of the challenges in statistics for stochastic processes is always to develop a general theory that goes beyond the investigation of specific models. An elegant example are stationary Gaussian time series or  linear time series which are identifiable from the covariance structure or the spectral density of the process. This leads to such powerful tools as the Whittle-likelihood and quite general asymptotic results on statistical inference. This linear theory has been extended to locally stationary processes (cf. \cite{dahlhaus1997}, \cite{dahlhaus2000},\cite{dahlhauspolonik2009}, for an overview see \cite{dahlhaus2012}, Chapter 5) leading again to a general framework in which problems such as bootstrap methods for locally stationary processes (cf. \cite{sergidespaparoditis2008},\cite{kreisspaparoditis2015}), testing problems (cf. \cite{sergidespaparoditis2009},\cite{preussEtAl2013})  long memory models (cf. \cite{palmaolea2010}, \cite{roueffvonsachs2011}) or dynamic non-stationary factor models (cf. \cite{mottaEtAl2011},\cite{eichlerEtAl2011}) can be considered.

In the nonlinear case the situation is more challenging since there is no natural framework similar to the linear Gaussian case. A general theory, however, has been introduced in \cite{wu2005} for Bernoulli shift processes in combination with the functional dependence measure - an important example being Markov processes (and also linear processes). By using this calculus it is possible to transfer a large number of results from the iid case to such processes (for instance M-estimation \cite{wu2007}, empirical process theory \cite{wu2008}, high-dimensional covariance matrix estimation \cite{chen2013} - see also the overview in \cite{wu2011}).

The functional dependence measure for Bernoulli shift processes can also be extended to locally stationary processes (cf. \cite{WuAndZhou2011}) leading to a general framework for non-linear locally stationary processes. Within this framework for example \cite{zhouwu2009} and \cite{WeichiWuAndZhou2017} discuss quantile regression, \cite{Zhou2014a} inference for weighted V-statistics and \cite{Zhou2014b} nonparametric regression for locally stationary processes.

Another general concept for locally stationary processes is the use of stationary approximations and derivative processes introduced in the context of time varying ARCH-processes in \cite{suhasini2006}. The concept has been investigated further in \cite{suhasini2006b} in the context of random coefficient models. \cite{vogt2012} uses the stationary approximation for a definition of local stationarity. The concept has also been used for diffusion processes in \cite{KooLinton2012}.

In this paper the general theory for nonlinear locally stationary processes will be developed further. Our contribution is twofold: First we consider in Sections 2 and 3 processes which admit an expansion in terms of the stationary approximation and the derivative process, and prove several asymptotic results for such processes. We then consider in Section 4 a general Markov-structured locally stationary process and show that it fulfills such an expansion. As a consequence all results of Sections 2 and 3 immediately can be applied for such processes. In addition we use in Section 5 the stationary approximation and the derivative process to derive the asymptotic theory for maximum likelihood estimates.

More precisely we use in Section \ref{DP_section35} the stationary approximation to prove global and local laws of large numbers and a central limit theorem which hold under minimal moment assumptions on the process. The proofs make use of the asymptotic theory for sums of stationary sequences. In Section \ref{DP_section3} we use the differential calculus connected to derivative processes to derive deterministic and stochastic bias expansions of localized sums. In addition we show accurate error estimations of the Wigner-Ville spectrum and the distribution function of the process.

In Section \ref{DP_section2}, we consider a class of Markov processes and prove that they satisfy the expansion from Section \ref{DP_section3} in terms of the stationary approximation and the derivative process. A difficult part of the proof is the existence of a continuous modification of the stationary approximation and the proof that the derivative process can be obtained as the solution of a functional equation. These results are the prerequisite to apply the functional calculus used in sections 2 and 3 and in section 5 thereafter. We also prove that the functional dependence measure of such processes decays exponentially. Locally stationary Markov processes have recently also been investigated with different type of results in \cite{truquet2016}.

In Section \ref{DP_section4}, we use these results to investigate nonparametric maximum likelihood estimation of parameter curves in locally stationary processes. Concluding remarks are given in Section \ref{DP_section5}. Some proofs are postponed to the Supplementary Material \ref{suppA}.

\section{General asymptotic results for locally stationary processes}

\label{DP_section35}

In this and the next section we restrict ourselves to the basic idea of local stationarity, namely that the nonstationary process $X_{t,n}$ can be approximated locally by a stationary process $\tilde X_t(u)$ in some neighborhood of $u$, that is for those $t$ where $|t/n -u|$ is small. An even better approximation can usually be achieved by using the derivative process $\frac {\partial \tilde X_t(u)} {\partial u}$ leading heuristically to
\begin{equation} \label{}
X_{t,n} \approx \tilde X_t \big(\frac {t} {n}\big) \approx \tilde X_t(u_0) + \Big(\frac {t} {n} - u_0\Big) \, \frac {\partial \tilde X_t(u)} {\partial u}_{\big|u=u_0} + \; \mbox{remainder}\label{stat_expansions}
\end{equation}
(or higher order approximations by using higher order derivative processes). An example are the processes defined below in (\ref{DP_rek_1}) and (\ref{DP_rek_2}) which are investigated in detail in Section~\ref{DP_section2}. In this and the next section we explore what kind of results can be obtained just based on this framework where in this section we just use the stochastic approximation $\tilde X_t(u)$ (to derive law of large numbers and central limit theorems) while in the next section we use in addition the derivative process $\frac {\partial \tilde X_t(u)} {\partial u}$ (to derive deterministic and stochastic bias approximations among other results). For a better understanding we make some comments about the situation in advance:\\[6pt]
1) In Assumption~\ref{DP_ass_general}(S2) and (S3) we assume that $\tilde X_t(u)$ is almost surely continuous or continuously differentiable in $u$ respectively. This is a strong assumption since $\tilde X_t(u)$ is initially defined in most cases pointwise in $u$ as in (\ref{DP_rek_2}). The existence of continuous or continuously differentiable versions of $\tilde X_t(u)$ will be proved for (\ref{DP_rek_2}) in Section~\ref{DP_section2}.\\
We mention that almost sure differentiability could be replaced by the weaker assumption of differentiability in $L^q$ (see Remark \ref{remark_differentiability} and Proposition \ref{DP_prop_biasexp}(b)). The stronger assumption of almost sure differentiability leads to some weaker assumptions in applications (see Remark \ref{remark_biasdet}) and has the advantage that several results can be obtained in a more straightforward way and that the presentation is simpler.

%
2) It is obvious that for asymptotic results some form of mixing is needed in addition which we require in Assumption~\ref{DP_ass_general2}. We distinguish between mixing conditions on $X_{t,n}$ and the stationary approximation $\tilde X_t(u)$ since our aim is to show that most of the results can be obtained by only posing mixing assumptions on $\tilde X_t(u)$ which on the other side leads to an additional approximation error in the results and subsequently in come cases to stronger assumptions.\\[6pt]
3) One of the nice features of Assumptions~\ref{DP_ass_general} and \ref{DP_ass_general2} is that, provided they hold for the process $X_{t,n}$, they then automatically also hold for a large class of functionals $g(X_{t,n},...,X_{t-r+1,n})$ (under modified moment assumptions). Thus all results for the process $X_{t,n}$ immediately transfer to $g(X_{t,n},...,X_{t-r+1,n})$. This is stated in Proposition~\ref{DP_lemma_preservation} below.

As usual we are working in the infill asymptotic framework with rescaled time $t/n \in [0,1]$ where $n$ denotes the number of observations. We now assume:

\begin{assumption}[Stationary approximation]\label{DP_ass_general} Let $q>0$ and $\|W\|_q := (\IE |W|^q)^{1/q}$. Let $X_{t,n}$, $t=1,...,n$ be a triangular array of stochastic processes. For each $u\in [0,1]$, let $\tilde X_t(u)$ be a stationary and ergodic process such that the following holds.
	\begin{enumerate}[label=(S\arabic*),ref=(S\arabic*)]
		\item \label{DP_ass_general_s1} $\sup_{u \in [0,1]}\|\tilde X_t(u)\|_q < \infty$. There exists $1 \ge \alpha > 0, C_B > 0$ such that uniformly in $t = 1,...,n$ and $u,v \in [0,1]$,
		\begin{equation}
			\big\|\tilde X_t(u) - \tilde X_t(v)\big\|_q \le C_B |u-v|^{\alpha}, \quad\quad \big\|X_{t,n} - \tilde X_t(\frac{t}{n})\big\|_q \le C_B n^{-\alpha}.\label{DP_ass_general_eq1}
		\end{equation}
		\item \label{DP_ass_general_s2} $u \mapsto \tilde X_t(u)$ is a.s. continuous for all $t \in \IZ$ and $\| \sup_{u\in[0,1]}|\tilde X_t(u)|\ \|_q < \infty$.
		\item \label{DP_ass_general_s3} $\alpha = 1$ and $u \mapsto \tilde X_t(u)$ is a.s. continuously differentiable for all $t \in \IZ$ and\\
       $\| \sup_{u\in[0,1]}|\partial_u \tilde X_t(u)|\ \|_q < \infty$.
	\end{enumerate}
\end{assumption}
\medskip

\ref{DP_ass_general_s1} allows us to replace $X_{t,n}$ by the stationary approximation $\tilde X_t(u)$ with rate $|t/n-u|^{\alpha} + n^{-\alpha}$. In many models and statistical applications, $\alpha = 1$. 
In Section \ref{DP_section2} (cf. Corollary \ref{DP_cor_ass_preservation})
we will show that Assumption \ref{DP_ass_general} is fulfilled for example for processes $X_{t,n}$ defined by the recursion
\begin{equation}
	X_{t,n} = G_{\varepsilon_t}\big(X_{t-1,n}, ..., X_{t-p,n},\frac{t}{n}\vee 0\big), \quad t \le n,
	\label{DP_rek_1}
\end{equation}
where $(\varepsilon_t)_{t\in\IZ}$ are i.i.d. random variables, $G:\IR \times \IR^p \times [0,1] \to \IR$ and $a \vee b := \max\{a,b\}$. Here, $\varepsilon_t$ take the role of i.i.d. innovations which enter the process in the $t$-th step. For $u \in [0,1]$, the \textit{stationary approximation}  $\tilde X_t(u)$, $t\in\IZ$, in this case is given by the recursion
\begin{equation}
	\tilde X_t(u) = G_{\varepsilon_t}\big(\tilde X_{t-1}(u),...,\tilde X_{t-p}(u), u\big), \quad t\in\IZ,\label{DP_rek_2}
\end{equation}
(or an a.s. continuous modification). We first give some examples which are covered by our results. These include in particular several classical parametric time series models where the constant parameters have been replaced by time-dependent parameter curves.

\begin{example}\label{DP_example_model}
	\begin{enumerate}
		\item[(i)] the tvAR($p$) process: Given parameter curves $a_i, \sigma:[0,1] \to \IR$ $(i = 1,...,p$),
		\[
			X_{t,n} = a_1\big(\frac{t}{n}\big)X_{t-1,n} + ... + a_p\big(\frac{t}{n}\big)X_{t-p,n} + \sigma\big(\frac{t}{n}\big)\varepsilon_t.
		\]
		\item[(ii)] the tvARCH($p$) process (cf. \cite{suhasini2006}): Given parameter curves $a_i:[0,1]\to \IR$ ($i = 0,...,p$),
		\[
			X_{t,n} = \big( a_0\big(\frac{t}{n}\big) + a_{1}\big(\frac{t}{n}\big)X_{t-1,n}^2 + ... + a_p\big(\frac{t}{n}\big) X_{t-p,n}^2\big)^{1/2}\varepsilon_t.
		\]
		\item[(iii)] the tvTAR($1$) process (cf. \cite{zhouwu2009}): Given parameter curves $a_1,a_2:[0,1]\to\IR$, define
		\[
			X_{t,n} = a_1\big(\frac{t}{n}\big)X_{t-1,n}^{+} +a_2\big(\frac{t}{n}\big)X_{t-1,n}^{-} + \varepsilon_t,
		\]
		where $x^{+} := \max\{x,0\}$ and $x^{-} := \max\{-x,0\}$.
		\item[(iv)] the time-varying random coefficient model (cf. \cite{suhasini2006b}): With some parameter functions $a_i(\cdot)$, $i = 0,...,p$,
		\[
			X_{t,n} = a_0(\varepsilon_t, \frac{t}{n}) + a_1(\varepsilon_t, \frac{t}{n}) X_{t-1,n} + ... + a_p(\varepsilon_t, \frac{t}{n}) X_{t-p,n}.
		\]
	\end{enumerate}
\end{example}

We now prove laws of large numbers and a central limit theorem. To specify the necessary mixing conditions we use the uniform functional dependence measure (cf. \cite{wu2013}). Let $\varepsilon_t$, $t\in\IZ$ be a sequence of i.i.d. random variables. For $t \ge 0$, let $\sF_t := (\varepsilon_t,\varepsilon_{t-1},...)$ and $\sF_t^{*(t-k)} := (\varepsilon_t,...,\varepsilon_{t-k+1},\varepsilon_{t-k}^{*},\varepsilon_{t-k-1},\varepsilon_{t-k-2},...)$, where $\varepsilon_{t-k}^{*}$ is a random variable which has the same distribution as $\varepsilon_{1}$ and is independent of all $\varepsilon_t$, $t\in\IZ$. For a process $W_t = H_t(\sF_t) \in L^q$ with deterministic $H_t:\IR^{\IN}\to\IR$ define $W_t^{*(t-k)} := H_t(\sF_t^{*(t-k)})$ and the uniform functional dependence measure
\begin{equation}
	\delta^{W}_q(k) := \sup_{t\in\IZ}\|W_t - W_t^{*(t-k)}\|_q.\label{DP_def_dependence}
\end{equation}
If $Y_t$ is stationary, \reff{DP_def_dependence} reduces to the form $\delta^Y_q(k) = \|Y_k - Y_k^{*0}\|_q$.

\begin{assumption}[Dependence measure]\label{DP_ass_general2} For some $q > 0$, assume that
\begin{enumerate}[label=(M\arabic*),ref=(M\arabic*)]
	\item\label{DP_ass_general2_m1} (dependence measure of the stat. approximation) for each $u \in [0,1]$, there exists a measurable function $H(u,\cdot)$ such that $\tilde X_t(u) = H(u,\sF_t)$ and $\delta^{\tilde X}_q(k) := \sup_{u\in [0,1]}\delta_q^{\tilde X(u)}(k)$ fulfills $\Delta_{0,q}^{\tilde X} := \sum_{k=0}^{\infty}\delta_q^{\tilde X}(k) < \infty$.
	\item\label{DP_ass_general2_m2} (dependence measure of the process) for each $t,n\in\IN$, there exists a measurable function $H_{t,n}$ such that $X_{t,n} = H_{t,n}(\sF_t)$ with $\Delta_{0,q}^{X} := \sum_{k=0}^{\infty}\sup_{n\in\IN}\delta_q^{X_{\cdot,n}}(k) < \infty$.
	\item\label{DP_ass_general2_m3} $\partial_u \tilde X_t(u) = \partial_u H(u,\sF_t)$ and $\delta^{\partial \tilde X}_q(k) := \sup_{u \in [0,1]} \delta_q^{\partial_u \tilde X(u)}(k)$ is absolutely summable in the sense that $\Delta_{0,q}^{\partial \tilde X} := \sum_{k=0}^{\infty}\delta_q^{\partial \tilde X}(k) < \infty$.
	\end{enumerate}
\end{assumption}

In Section \ref{DP_section2} we will show that these assumptions are also fulfilled for models which obey \reff{DP_rek_1} and \reff{DP_rek_2}. Note that \ref{DP_ass_general2_m1} is a mixing condition on $\tilde X_t(u)$ (which in most cases is sufficient for asymptotic results) while \ref{DP_ass_general2_m2} is a mixing condition on $X_{t,n}$. In our results, \ref{DP_ass_general2_m1} or \ref{DP_ass_general2_m2} are assumed alternatively. In general our goal is to use mainly assumptions on the stationary approximations, i.e. to use \ref{DP_ass_general2_m1} instead of \ref{DP_ass_general2_m2}. This also allows $X_{t,n}$ to have a structure which is different from $H_{t,n}(\sF_t)$, for instance contamination with some random noise which is decreasing in $n$ and is not produced by the innovations $\varepsilon_i$. Note that posing only \ref{DP_ass_general2_m1} forces us to replace $X_{t,n}$ in the proofs by its stationary approximation which naturally leads to an approximation error $n^{-\alpha}$. In many practical cases, we have $\alpha = 1$. It can be seen in our results that this implies negligibility of this error.

For the model $\tilde X_t(u) = H(u,\cdot)$, $X_{t,n} = H_{t,n}(\sF_t)$ with \ref{DP_ass_general_s1}, then \ref{DP_ass_general2_m2} implies \ref{DP_ass_general2_m1} while the reverse implication is false. This can be seen as follows: For arbitrary $u \in [0,1]$ choose $t_n \in \{1,...,n\}$ with $|u-t_n/n| \le n^{-\alpha}$. Then
\begin{eqnarray*}
\delta_q^{\tilde X(u)}(k) &=& \|\tilde X_t(u) - \tilde X_t(u)^{*(t-k)}\|_q = \|\tilde X_t(t_n/n) - \tilde X_t(t_n/n)^{*(t_n-k)}\|_q + O(n^{-\alpha})\\
&=& \|X_{t_n,n} - X_{t_n,n}^{*(t_n-k)}\|_q + O(n^{-\alpha}) \le \sup_{n\in\IN}\delta_q^{X_{\cdot,n}}(k) + O(n^{-\alpha}),
\end{eqnarray*}
which by $n\to\infty$ implies $\ref{DP_ass_general2_m1}$.
If for instance for each fixed $u \in [0,1]$, $\tilde X_t(u)$ is a simple AR(1) process $\tilde X_t(u) =\sum_{k=0}^{\infty}\alpha(u)^k \varepsilon_{t-k}$ with Lipschitz continuous $\alpha:[0,1] \to (-1,1)$ and $\|\varepsilon_i\|_q < \infty$, then it is easy to see that $X_{t,n} = \tilde X_t(t/n) + (\varepsilon_t + ... + \varepsilon_{t-n})n^{-2}$ satisfies \ref{DP_ass_general_s1} with $q = 1$ and $\alpha = 1$, but $\delta_{q}^{X_{\cdot,n}}(k) = \sup_{t}|\theta(t/n)^k + n^{-2}|$ is not absolutely summable. There are more counterexamples, see the Supplementary Material \ref{suppB}.

\textbf{Invariance property of the assumptions with respect to transformations: } For some fixed $r \in \IN$ define $Z_{t,n} := (X_{t,n},...,X_{t-r+1,n})'$ and $\tilde Z_t(u) := (\tilde X_{t}(u),...,\tilde X_{t-r+1}(u))'$. We prove that $g(Z_{t,n})$ also fulfills Assumptions \ref{DP_ass_general} and \ref{DP_ass_general2} for the following class of functions $g$.


\begin{definition}[The class $\sL_r(M,C)$]\label{DP_hoelder_def} We say that a function $g:\IR^r \to \IR$ is in the class $\sL_r(M,C)$ if $M \ge 0$ and
	\begin{equation}
		\sup_{y\not= y'}\frac{|g(y) - g(y')|}{|y-y'|_1\cdot (1+|y|_1^{M} + |y'|_1^{M})} \le C,\label{DP_hoelder_prop_polynom}
	\end{equation}
	where $|y|_1 := \sum_{i=1}^{r}|y_i|$.
\end{definition}

\begin{proposition}[Invariance property of the stationary approximation]\label{DP_lemma_preservation}
	Assume that $g \in \sL_r(M,C)$ and $q > 0$.
	\begin{enumerate}
		\item[(i)]
	If Assumption \ref{DP_ass_general}\ref{DP_ass_general_s1},\ref{DP_ass_general_s2} and Assumption \ref{DP_ass_general2}\ref{DP_ass_general2_m1},\ref{DP_ass_general2_m2} are fulfilled for the process $X_{t,n}$ with $\tilde q = q\cdot (M+1)$ and $1 \ge  \alpha > 0$, then the same assumptions are fulfilled for the process $g(Z_{t,n})$ with $q$ and $\alpha$.
	\end{enumerate}
	Now let $g$ in addition be continuously differentiable where the partial derivatives $\partial_j g$, $j = 1,...,r$ fulfill $\partial_j g \in \sL_r(M-1,C')$ with some $M \ge 1$.
	\begin{enumerate}
		\item[(ii)] If Assumption \ref{DP_ass_general}\ref{DP_ass_general_s3} and \ref{DP_ass_general2}\ref{DP_ass_general2_m1},\ref{DP_ass_general2_m3} are fulfilled for the process $X_{t,n}$ with $\tilde q = q\cdot (M+1)$, then the same assumptions are fulfilled for the process $g(Z_{t,n})$ with $q$.
	\end{enumerate}
\end{proposition}

The proof is immediate from Hoelder's inequality and therefore omitted. Let us mention that the condition $\partial_j g \in \sL_r(M-1,C')$ is only needed to prove the mixing condition \ref{DP_ass_general2}\ref{DP_ass_general2_m3} since $\|\sup_{u \in [0,1]}|\partial_u g(\tilde Z_0(u))| \|_q < \infty$ can be shown by using $|\partial_j g(z)| \le C(1 + |z|_1^{M})$. 

With slight changes, the statements of Proposition \ref{DP_lemma_preservation} can be extended to Hoelder continuous functions which fulfill
\begin{equation}
		\sup_{y\not= y'}\frac{|g(y) - g(y')|}{|y-y'|_1^{\beta}\cdot (1+|y|_1^{M} + |y'|_1^{M})} \le C,\label{DP_hoelder_prop_polynom_2}
	\end{equation}
with some $1 \ge \beta > 0$.

In view of the above proposition, \underline{all} theorems formulated for $X_{t,n}$ in this and the next section are, under appropriate moment conditions, also valid for transformations $g(Z_{t,n})$ of $X_{t,n}$. An important example  is the covariance operator $g:\IR^{r}\to \IR, g(x_1,...,x_r) = x_1 x_r$ which leads to $g(Z_{t,n}) = X_{t,n}X_{t-r+1,n}$ and fulfills $g \in \sL_{r}(1,1)$.


\textbf{Local and global laws of large numbers: } The smoothness of $X_{t,n}$ in time direction can be used to obtain laws of large numbers by only assuming the existence of the first moment of $X_{t,n}$. The key step of the proof is to split the sum over $X_{t,n}$ into sums over smaller ranges of $t$ where $X_{t,n}$ can be approximated by stationary processes. We will also provide results for localized sums. Usually, we will need the following
\begin{assumption}[Localizing kernel]\label{DP_ass_kernel}
	$K:\IR\to\IR$ is a bounded function, i.e. with some $|K|_{\infty} > 0$ it holds that $\sup_{x\in\IR}|K(x)| \le |K|_{\infty}$, and of bounded variation $B_K$ with compact support $[-\frac{1}{2},\frac{1}{2}]$ satisfying $\int K\dif x = 1$. Let $K_b(x) := \frac{1}{b}K(\frac{x}{b})$.
\end{assumption}
The first part of the following theorem can be seen as a generalization of the ergodic theorem to non-stationary processes, while the second part provides uniform convergence rates if more than the first moment is available.

\begin{theorem}[Law of large numbers]\label{DP_prop_ergodic}
    Let $q = 1$ in (i),(ii) and $q > 1$ in (iii). Suppose that Assumption \ref{DP_ass_general}\ref{DP_ass_general_s1} holds with some $1\ge \alpha > 0$ and that Assumption \ref{DP_ass_kernel} holds. Then we obtain for the process $X_{t,n}$ (or alternatively for the process $g (Z_{t,n})$ if Assumption \ref{DP_ass_general}\ref{DP_ass_general_s1} is fulfilled for the process $X_{t,n}$ with $\tilde q = q\cdot (M+1)$ instead of $q$) the following results:
	\begin{enumerate}
		\item[(i)]
	\[
		\frac{1}{n}\sum_{t=1}^{n}X_{t,n} \to \int_0^{1} \IE \tilde{X}_0(u) \dif u \quad \mbox{in } L^1
	\]
       \item[(ii)]
	For each $u\in (0,1)$
	\[
		\frac{1}{nb}\sum_{t=1}^{n}K\Big(\frac{t/n-u}{b}\Big)\cdot X_{t,n} \to \IE \tilde{X}_0(u) \quad \mbox{in } L^1
	\]
	as $n\to\infty$, $nb \to \infty$ and $b = b_n \to 0$.
	\item[(iii)]Additionally, suppose that Assumption \ref{DP_ass_general2}\ref{DP_ass_general2_m1} holds with $q, \alpha$. Then
	\begin{align}
			& \Big\|\sup_{u\in[0,1]}\big|\frac{1}{nb}\sum_{t=1}^{n}K\Big(\frac{t/n-u}{b}\Big)\cdot \big(X_{t,n} - \IE X_{t,n}\big)\big|\Big\|_q\nonumber\\
			&\quad\quad\le \frac{B_K p}{(q-1)^2}\Delta_{q,0}^{\tilde X}\cdot n^{1/q-1}b^{-1} + 2C_B B_K \cdot n^{-\alpha}b^{-1}.\label{DP_prop_ergodic_eq1}
		\end{align}
		If $q > 2$, then there exist constants $C_1,C_2$ not depending on $n,b$ such that for all $x > 0$:
		\begin{align}
			& \IP\Big(\sup_{u\in[0,1]}\big|\frac{1}{nb}\sum_{t=1}^{n}K\Big(\frac{t/n-u}{b}\Big)\cdot \big(X_{t,n} - \IE X_{t,n}\big)\big| > x\Big)\nonumber\\
			&\quad\quad\le \frac{2C_1 (B_K \Delta_{0,q}^{\tilde X})^q n^{1-q}b^{-q}}{(x/2)^q} + 8G_{1-2/q}\Big(\frac{C_2 n^{1/2}bx}{2B_K \Delta_{0,q}^{\tilde X}}\Big) + \frac{(2 B_K C_B)^q}{(x/2)^q}\cdot (n^{-\alpha}b^{-1})^q, \label{DP_prop_ergodic_eq2}
		\end{align}
		with positive constants $C_1,C_2$ not depending on $n,b$ and $G_{\gamma}(y) := \sum_{j=1}^{\infty}e^{-j^{\gamma} y^2}$ a Gaussian-like tail function. $G_{\gamma}(y)$ can be replaced by $\exp(-c y^2)$ for some $c > 0$.
	\end{enumerate}
\end{theorem}


\begin{remark}\label{DP_remark_ergodic}
	\begin{enumerate}
		\item[(i)] The additional $O(n^{-\alpha}b^{-1})$ or $O((n^{-\alpha}b^{-1})^{\tilde q})$  terms in \reff{DP_prop_ergodic_eq1} and \reff{DP_prop_ergodic_eq2}, respectively, can be omitted under Assumption \ref{DP_ass_general2}\ref{DP_ass_general2_m2}. In this case, one has to replace $\Delta_{0,\tilde q}^{\tilde X}$ by $\Delta_{0,\tilde q}^{X}$.
		\item[(ii)]
		For $q > 1$, $b = o(n^{1-\frac{1}{q}})$ and $n = o(b^{1/\alpha})$, the results of Theorem \ref{DP_prop_ergodic}(ii) and Proposition \ref{DP_prop_biasexp} (which is about bias expansion) can be used to obtain uniform convergence of the mean estimator $\hat \mu_b(u) := \frac{1}{nb}\sum_{t=1}^{n}K\big(\frac{t/n-u}{b}\big)X_{t,n}$ towards $\mu(u) := \IE \tilde X_0(u)$ in the sense that
		\[
			\sup_{u\in[\frac{b}{2},1-\frac{b}{2}]}|\hat \mu_b(u) - \mu(u)| \pto 0.
		\]
	\end{enumerate}
\end{remark}


\textbf{A central limit theorem:} We provide local and global central limit theorems under minimal moment conditions which are useful in particular to find asymptotic distributions of (nonparametric) estimators of locally stationary processes, see Section \ref{DP_section4}. For the proofs we need the dependence condition  from Assumption \ref{DP_ass_general2}\ref{DP_ass_general2_m1}.  Assumption \ref{DP_ass_general}\ref{DP_ass_general_s2}, namely $\|\sup_{u \in [0,1]}|\tilde X_t(u)|\|_q < \infty$, is crucial to show a Lindeberg-type condition under minimal moment conditions.

\begin{theorem}[Central limit theorem - global version] \label{DP_prop_clt} Let Assumption \ref{DP_ass_kernel} hold. Suppose that Assumption \ref{DP_ass_general}\ref{DP_ass_general_s1}, \ref{DP_ass_general_s2} and \ref{DP_ass_general2}\ref{DP_ass_general2_m1} hold with some $q \ge 2$ and $\alpha > \frac{1}{2}$. Define $S_n := \sum_{t=1}^{n}\big(X_{t,n}-\IE X_{t,n}\big)$ (if  $X_{t,n}$ is replaced by $g (Z_{t,n})$ in the assertions below the same assumptions must be fulfilled with $\tilde q = q\cdot (M+1)$ instead of $q$). Then we have the following invariance principle:
\[
	\{S_{\lfloor nu\rfloor}/\sqrt{n},0 \le u \le 1\} \dto \Big\{\int_{0}^{u}\sigma(v) \dif B(v), 0 \le u \le 1\Big\},
\]
where $B(v)$ is a standard-Brownian motion and the long-run variance $\sigma^2(v)$ is given by
\begin{equation}
	\sigma^2(v) = \sum_{k\in\IZ}\Cov(\tilde X_0(v), \tilde X_k(v)).\label{DP_prop_clt_eq1}
\end{equation}
The condition $\alpha > \frac{1}{2}$ can be omitted under Assumption \ref{DP_ass_general2}\ref{DP_ass_general2_m2}.
\end{theorem}

Finally, we present a simple localized version of the central limit theorem for general locally stationary processes.

\begin{theorem}[Central limit theorem - local version]\label{DP_clt_local} Let Assumption \ref{DP_ass_kernel} hold. Suppose that Assumption \ref{DP_ass_general}\ref{DP_ass_general_s1}, \ref{DP_ass_general_s2} and \ref{DP_ass_general2}\ref{DP_ass_general2_m1} hold with some $q \ge 2$ and $1\ge \alpha > 0$ (if  $X_{t,n}$ is replaced by $g (Z_{t,n})$ in the assertion below the same assumptions must be fulfilled with $\tilde q = q\cdot (M+1)$ instead of $q$). Then for all $u\in (0,1)$, provided that $\sqrt{nb}\cdot n^{-\alpha} \to 0$, $b \to 0$ and $nb \to \infty$:
\begin{eqnarray*}
	\frac{1}{\sqrt{nb}}\sum_{t=1}^{n}K\Big(\frac{t/n-u}{b}\Big)\cdot \big\{X_{t,n} - \IE X_{t,n}\big\} \dto N\big(0, \int K(x)^2 \dif x\cdot \sigma^2(u)\big)
\end{eqnarray*}
with $\sigma^2(u)$ defined in \reff{DP_prop_clt_eq1}.
\end{theorem}

\section{Differential calculus for nonstationary processes}
\label{DP_section3}

In this section we prove almost sure and uniform $L^q$ Taylor expansions of $X_{t,n}$ and invent a kind of differential calculus for locally stationary processes. Below we will use this stronger kind of approximation for proving deterministic and stochastic bias expansions. While deterministic bias expansions are used to bound the expectation of expressions which include $X_{t,n}$, stochastic bias expansions can be used to replace the whole localized sum by a localized sum of the stationary process $\tilde X_t(u)$ which is much easier to analyze with known tools. We also prove results on the spectrum and the empirical distribution function.

\subsection{Taylor expansions and expansions of localized sums}

\begin{proposition}[Taylor expansion]\label{DP_cor_taylor_exp}
	Suppose that Assumption \ref{DP_ass_kernel} holds. Suppose that Assumption \ref{DP_ass_general} holds for some $q > 0$. Then we have for all $u\in[0,1]$ and $t = 1,...,n$:
	\begin{equation}
		\tilde X_t\big(\frac{t}{n}\big) = \tilde X_t(u) + \big(\frac{t}{n} - u\big)\cdot \partial_u \tilde X_t(u) + R_{t,n} \quad a.s.,\label{DP_taylorexpansion1}
	\end{equation}
	If $|\frac{t}{n}-u| = o(1)$, then it holds that $R_{t,n} = o_{a.s.}(|\frac{t}{n}-u|)$. Furthermore,
	\begin{eqnarray}
		&& \frac{1}{n}\sum_{t=1}^{n}K_b\big(\frac{t}{n}-u\big)\tilde X_t(\frac{t}{n})\nonumber\\
		&=& \frac{1}{n}\sum_{t=1}^{n}K_b\big(\frac{t}{n}-u\big)\tilde X_t(u) + \frac{1}{n}\sum_{t=1}^{n}K_b\big(\frac{t}{n}-u\big)\big(\frac{t}{n} - u\big) \partial_u \tilde X_t(u) + R_{n} \quad a.s.\label{DP_taylorexpansion2}
	\end{eqnarray}
	holds and $|R_n| \le |K|_{\infty}\cdot b \cdot \sup_{|v-u| \le b}|\partial_u X_t(v) - \partial_u X_t(u)| = o_{a.s.}(b)$.
\end{proposition}

If one is interested in the approximation of moments of localized sums (as it may be the case in bias expansions in nonparametric frameworks), the following theorem is appropriate. It also closes the gap between the locally stationary process $X_{t,n}$ and its approximation $\tilde X_t(t/n)$.


\begin{corollary}[Almost sure and $L^q$ expansion of localized sums]\label{DP_cor_taylor_sums}
	Suppose that Assumption \ref{DP_ass_kernel} holds. Suppose that Assumption \ref{DP_ass_general} holds for some $q \ge 1$. Then for each fixed $u \in (0,1)$,
	\[
		\Big\|\frac{1}{n}\sum_{t=1}^{n}K_b\big(\frac{t}{n}-u\big)\cdot X_{t,n} - \frac{1}{n}\sum_{t=1}^{n}K_b\big(\frac{t}{n}-u\big)\cdot \tilde X_{t}\big(\frac{t}{n}\big)\Big\|_q = O(n^{-1}),
	\]
	and the expansion \reff{DP_taylorexpansion2} is valid with  $\|R_n\|_q = o(b)$.
\end{corollary}

It is also possible to use expansions similar to above if the sum is not localized by a kernel. An example can be found in the proof of Theorem \ref{DP_prop_ergodic}.

\subsection{Bias Expansions}

In nonparametric statistics, bias expansions play an important role to control the mean squared error (MSE) of estimators. Here we give an approach to estimate the deterministic bias term involving locally stationary processes. In recent years (for instance due to model selection via contrast minimization) a more careful analysis of the stochastic part in the calculation of the MSE became important. We make a contribution to this topic via a stochastic bias expansion which allows us to remove the bias from a localized sum $\frac{1}{n}\sum_{t=1}^{n}K_b\big(\frac{t}{n}-u\big)X_{t,n}$ such that only a localized sum over a stationary process $\frac{1}{n}\sum_{t=1}^{n}K_b\big(\frac{t}{n}-u\big)\tilde X_{t}(u)$ remains for which much more theoretical work was done. To obtain bias expansions of smaller order than $O(b)$ we have to assume differentiability of the upcoming expectations which is ensured by assuming Assumption \ref{DP_ass_general}\ref{DP_ass_general_s3}. To emphasize some differences that occur when Assumption \ref{DP_ass_general}\ref{DP_ass_general_s3} is changed to differentiability in $L^q$, we state the deterministic bias expansion for both settings and comment in Remark \ref{remark_biasdet}.

\begin{proposition}[Deterministic bias expansion]\label{DP_prop_biasexp}
	Suppose that Assumption \ref{DP_ass_kernel} holds. Let $q \ge 1$. Suppose that Assumption \ref{DP_ass_general}\ref{DP_ass_general_s1} is fulfilled with some $1 \ge \alpha > 0$, then we have uniformly in $u\in[0,1]$:
	\begin{equation}
		\frac{1}{n}\sum_{t=1}^{n}K_b\big(\frac{t}{n}-u\big)\big\{\IE X_{t,n} - \IE \tilde X_t(t/n)\big\} = O(n^{-\alpha}),\label{DP_cor_biasexp_eq1}
	\end{equation}
	and
	\begin{equation}
		\frac{1}{n}\sum_{t=1}^{n}K_b\big(\frac{t}{n}-u\big)\big\{\IE \tilde X_{t}(t/n) - \IE \tilde X_0(u)\big\} = O(b^{\alpha}) + O(n^{-1}).\label{DP_cor_biasexp_eq2}
	\end{equation}
	Now assume additionally that $K$ is symmetric.
	\begin{enumerate}
		\item[(a)] If Assumption \ref{DP_ass_general}\ref{DP_ass_general_s3} holds, then \reff{DP_cor_biasexp_eq1} is valid with $\alpha = 1$ and we have uniformly in $u \in [\frac{b}{2},1-\frac{b}{2}]$
	\begin{equation}
		\frac{1}{n}\sum_{t=1}^{n}K_b\big(\frac{t}{n}-u\big)\big\{\IE \tilde X_{t}(t/n) - \IE \tilde X_0(u)\big\} = o(b) + O(n^{-1}).\label{DP_cor_biasexp_eq3}
	\end{equation}
		\item[(b)] Suppose that $[0,1] \to L^{q}, u \mapsto \tilde X_0(u)$ is Fr\'{e}chet differentiable with derivative $\tilde D_t(u)$, 
		i.e. for all $u \in [0,1]$,
\begin{eqnarray}
	\lim_{h\to 0}\Big\|\frac{\tilde X_0(u+h) - \tilde X_0(u)}{h} - \tilde D_0(u)\Big\|_q &=& 0.\label{lqdiff_eq1}
\end{eqnarray}
Then the statement of (a) holds.
	\end{enumerate}
\end{proposition}

\medskip

The proof of \reff{DP_cor_biasexp_eq1} and \reff{DP_cor_biasexp_eq2} follows from the Hoelder inequality and the fact that $K$ has bounded variation and thus is bounded. To prove \reff{DP_cor_biasexp_eq3}, note that
\[
	\tilde X_t(t/n)) = \tilde X_t(u) + \big(\frac{t}{n} -u\big)\cdot \partial_u \tilde X_t(u) + \int_{u}^{t/n} \big\{\partial_u \tilde X_t(s) - \partial_u \tilde X_t(u)\big\} \dif s.
\]
As long as $|\frac{t}{n}-u| \le b$, we have
\[
	\big|\IE \int_{u}^{t/n} \big\{\partial_u \tilde X_t(s) - \partial_u \tilde X_t(u)\big\} \dif s\big| \le b\cdot \sup_{|u-s| \le b}\|\partial_u \tilde X_t(s) - \partial_u \tilde X_t(u)\|_1 = o(b),
\]
since $u \mapsto \partial_u \tilde X_t(u)$ is continuous and $\|\sup_{u}|\partial_u \tilde X_t(u)|\|_1 < \infty$. Finally, because $K$ has bounded variation and is symmetric,
\begin{eqnarray*}
	&& \frac{1}{n}\sum_{t=1}^{n}K_b\big(\frac{t}{n}-u\big)\big\{\IE \tilde X_{t}(t/n) - \IE \tilde X_t(u)\big\}\\
	&=& \frac{1}{n}\sum_{t=1}^{n}K_b\big(\frac{t}{n}-u\big)\cdot \big(\frac{t}{n}-u\big)\cdot \IE[\partial_u \tilde X_t(u)] + o(b) = O(n^{-1}) + o(b).
\end{eqnarray*}
The proof of (b) follows by using
\begin{eqnarray*}
	&& \big|\frac{1}{n}\sum_{t=1}^{n}K_b\big(\frac{t}{n}-u\big)\big\{\IE \tilde X_{t}(t/n) - \IE \tilde X_t(u)\big\}\big|\\
	&\le& \frac{1}{n}\sum_{t=1}^{n}K_b\big(\frac{t}{n}-u\big)\big(\frac{t}{n}-u\big) \big\|\frac{\tilde X_t(t/n) - \tilde X_t(u)}{\frac{t}{n}-u} - \tilde D_t(u)\big\|_1\\
	&&\quad\quad + \frac{1}{n}\sum_{t=1}^{n}K_b\big(\frac{t}{n}-u\big)\big(\frac{t}{n}-u\big)\IE \tilde D_t(u) =  o(b) + O(n^{-1}).
\end{eqnarray*}

\begin{remark}
	Note that in the situation of Proposition \ref{DP_prop_biasexp}, derivative processes were used to get $o(b)$ instead of $O(b)$ in \reff{DP_cor_biasexp_eq2}. Even smaller rates can be obtained by using higher order derivative processes together with higher order kernels.
	If we assume that $u \mapsto \tilde X_t(u)$ has a twice continuously differentiable modification and $K$ is symmetric, we obtain a bias decomposition whose structure is well-known:
	\[
		\frac{1}{nb}\sum_{t=1}^{n}K\Big(\frac{t/n-u}{b}\Big)\big\{\IE X_{t,n} - \IE \tilde X_0(u)\big\} = \int x^2 K(x) \dif x \cdot \IE[\partial_u^2 \tilde X_t(u)]\cdot b^2 + o(b^2) + O(n^{-1}).
	\]
\end{remark}

\begin{remark}[Almost sure v.s. $L^q$ differentiability]\label{remark_biasdet} Let us briefly comment on the different conditions in Proposition \ref{DP_prop_biasexp}(a), (b). In (a) we ask for a.s. differentiability of $u \mapsto \tilde X_0(u)$ while  (b) asks for differentiability in $L^q$ which is weaker. If we want to apply Proposition \ref{DP_prop_biasexp} to $g(Z_{t,n})$ with some function $g \in \sL_r(M,C)$ which is continuously differentiable, then there occur some differences due to the different natures of the conditions:
\begin{itemize}
	\item If Assumption \ref{DP_ass_general}\ref{DP_ass_general_s3} is fulfilled for $q' = q(M+1)$, then Assumption \ref{DP_ass_general}\ref{DP_ass_general_s3} is fulfilled for the process $g(Z_{t,n})$ with $q$ (cf. Proposition \ref{DP_lemma_preservation}(ii) and the comment afterwards) and we obtain \reff{DP_cor_biasexp_eq3} by Proposition \ref{DP_prop_biasexp}(a).
	\item If \reff{lqdiff_eq1} holds, we have to assume additionally that all the derivatives $\partial_j g$ ($j = 1,...,r$) are Hoelder continuous with polynomially growing Hoelder constant, i.e. with some $\gamma > 0$,
	\[
		\sup_{y\not= y'}\frac{|\partial_j g(y) - \partial_j g(y')|}{|y-y'|_1^{\gamma}\cdot (1+|y|_1^{M-1} + |y'|_1^{M-1})} < \infty,
	\]
	to obtain \reff{lqdiff_eq1} for $\tilde X_t^{\circ}(u) = g(\tilde Z_t(u))$ with derivative $D_t^{\circ}(u) = \partial_z g(\tilde Z_t(u))\cdot D_t(u)$.
\end{itemize}
Note that we have to ask $g$ to be slightly more smooth when using differentiability in $L^q$.
\end{remark}

We now prove the stochastic bias expansion. It turns out that we have to bound moments of sums of the upcoming derivative processes $\partial_u \tilde X_t(u)$ which means that we have to pose dependence conditions on $\partial_u \tilde X_t(u)$. This is done via Assumption \ref{DP_ass_general2}\ref{DP_ass_general2_m3}. Using the projection operator $P_{j}\cdot  := \IE[\cdot |\sF_{j}] - \IE[\cdot|\sF_{j-1}]$, we can bound moments of sums of $\partial_u \tilde X_t(u)$ by moments of martingales which can then be bounded with results from \cite{rio2009}. It can be shown (similar to \cite{wu2005}, Theorem 1(i) and (ii)) that for some shift process $W_t = H_t(\sF_t)$ with measurable $H_t$ it holds for $q \ge 1$:
\begin{equation}
	\|P_{t-k}W_t\|_q \le \delta^{W}_q(k).\label{dependence_measure_relation}
\end{equation}

\begin{proposition}[Stochastic bias expansion]\label{DP_prop_biasexp_stoch}
	Suppose that Assumption \ref{DP_ass_general} and \ref{DP_ass_general2}\ref{DP_ass_general2_m3} are fulfilled for some $q \ge 2$. Assume that $K$ is symmetric. Then we have
	\begin{equation}
		\sup_{u \in [\frac{b}{2},1-\frac{b}{2}]}\Big\|\frac{1}{n}\sum_{t=1}^{n}K_b\big(\frac{t}{n}-u\big) \big\{ X_{t,n} - \tilde X_t(u)\big\}\Big\|_q = o(b) + O(n^{-1}).\label{DP_cor_biasexp_eq5}
	\end{equation}
\end{proposition}
\begin{proof}[Proof of Proposition \ref{DP_prop_biasexp_stoch}:] To prove \reff{DP_cor_biasexp_eq5}, we can show similarly as in \reff{DP_cor_biasexp_eq1} that
\[
	\frac{1}{n}\sum_{t=1}^{n}|K_b\big(\frac{t}{n}-u\big)|\cdot \| X_{t,n} - \tilde X_t(t/n)\|_q = O(n^{-1}).
\]
To deal with $\frac{1}{n}\big\|\sum_{t=1}^{n}K_b\big(\frac{t}{n}-u\big) (\tilde X_{t}(t/n) - \tilde X_t(u))\big\|_q$, we use the expansion \reff{DP_taylorexpansion2} together with the result $\|R_n\|_q = o(b)$ from Corollary \ref{DP_cor_taylor_sums}. It therefore remains to analyze
\[
	\frac{1}{n}\big\|\sum_{t=1}^{n}K_b\big(\frac{t}{n}-u\big) \big(\frac{t}{n}-u\big)\cdot \big\{\partial_u \tilde X_t(u) - \IE \partial_u \tilde X_t(u)\big\}\big\|_q + \frac{1}{n}\sum_{t=1}^{n}K_b\big(\frac{t}{n}-u\big) \big(\frac{t}{n}-u\big)\cdot \IE \partial_u \tilde X_t(u).
\]
While the second term is $O(n^{-1})$ by stationarity, the first term is bounded by
\begin{eqnarray*}
	&&\frac{1}{n}\sum_{k=0}^{\infty}\big\|\sum_{t=1}^{n}K_b\big(\frac{t}{n}-u\big) \big(\frac{t}{n}-u\big) P_{t-k}\partial_u \tilde X_t(u)\big\|_q\\
	&\le& \frac{q^{1/2}}{nb}\sum_{k=0}^{\infty}\Big(\sum_{t=1}^{n}\big(K_b\big(\frac{t}{n}-u\big) \cdot \big(\frac{t}{n}-u\big)\big)^2 \|P_{t-k} \partial_u \tilde X_t(u)\|_q^2\Big)^{1/2}\\
	&\le& \frac{q^{1/2}|K|_{\infty}b^{1/2}}{n}\sum_{k=0}^{\infty}\delta_q^{\partial_u \tilde X(u)}(k) = \frac{q^{1/2}|K|_{\infty} b^{1/2}}{n}\Delta_{0,q}^{\tilde X}
\end{eqnarray*}
by Theorem 2.1 in \cite{rio2009} and \reff{dependence_measure_relation}.
\end{proof}

The main advantage of a stochastic bias expansion is that we can reduce a sum over locally stationary processes to a sum over stationary processes by keeping the terms stochastic. This allows for instance to apply large deviation results for stochastic processes which usually have a simpler and more closed form.

\begin{remark}[An application of the stochastic bias expansion: Deviation inequalities]
	Suppose that the assumptions of Proposition \ref{DP_prop_biasexp_stoch} are fulfilled with $q = 2$. Assume that $\IE \tilde X_t(u) = 0$. Then we have for $\gamma > 0$,
	\begin{eqnarray*}
		\IP\Big(\Big|\frac{1}{n}\sum_{t=1}^{n}K_b(u-t/n)X_{t,n}\Big| > \gamma\Big) &\le& \IP\Big(\Big|\frac{1}{n}\sum_{t=1}^{n}K_b(u-t/n) \tilde X_t(u)\Big| > \frac{\gamma}{2}\Big)\\
		&&\quad + \frac{1}{(\gamma/2)^2}\Big\|\frac{1}{n}\sum_{t=1}^{n}K_b(u-t/n) (X_{t,n} - \tilde X_t(u))\Big\|_2^2.
	\end{eqnarray*}
	By Proposition \ref{DP_prop_biasexp_stoch}, the second term on the right hand side is $o(b) + O(n^{-1})$, i.e. has bias order. For the first term, one can use deviation results for stationary processes.  
\end{remark}


\subsection{Differentiability of functionals}

The existence of derivative processes allows an expansion of the corresponding mean $\IE g(Z_{t,n})$ into the mean of the corresponding stationary version $\IE g(\tilde Z_t(u))$. This can be applied to various functionals such as expectations, covariances, the Wigner-Ville spectrum and the distribution function. The following result is an immediate Corollary from Lemma \ref{DP_lemma_preservation} applied to some $g \in \sL_r(M,C)$. Recall $Z_{t,n} = (X_{t,n},...,X_{t-r+1,n})'$ and $\tilde Z_t(u) = (\tilde X_t(u),...,\tilde X_{t-r+1}(u))'$.

\begin{proposition}\label{DP_general_mean_diff}
	Assume that $g \in \sL_r(M,C)$. Suppose that Assumption \ref{DP_ass_general}\ref{DP_ass_general_s1} is fulfilled for some $1 \ge \alpha > 0$ and $q = M+1$. Then we have uniformly for $t = 1,...,n$:
	\begin{equation}
		\IE g(Z_{t,n}) = \IE g\big(\tilde Z_t\big(\frac{t}{n}\big)\big) + O(n^{-\alpha}) = \IE g(\tilde Z_t(u)) + O\big(n^{-\alpha} + \big|\frac{t}{n}-u\big|^{\alpha}\big).\label{DP_general_mean_diff_eq1}
	\end{equation}
	If additionally Assumption \ref{DP_ass_general}\ref{DP_ass_general_s3} is fulfilled and $g$ is continuously differentiable with partial derivatives $\partial_j g \in \sL_{r}(M-1,C')$, $j = 1,...,r$, then $u \mapsto \IE g(\tilde Z_t(u))$ is continuously differentiable with derivative
	\begin{equation}
		\partial_u  \IE g(\tilde Z_t(u))  = \sum_{j=1}^{r}\IE[\partial_j g(\tilde X_{t}(u),...,\tilde X_{t-r+1}(u))\cdot \partial_u \tilde X_{t-j+1}(u)].\label{DP_general_mean_diff_eq2}
	\end{equation}
\end{proposition}

The result of Proposition \ref{DP_general_mean_diff} enables us to get expansions of the mean, the covariance and the distribution function of $X_{t,n}$. Suppose in the following that Assumption \ref{DP_ass_general} holds for some $q \ge M+1$.

\begin{corollary}[Mean expansion, $M = 0$]\label{DP_cor_mean_diff}
	Choosing $g:\IR \to \IR, g(y) = y$ yields
	\[
		\IE X_{t,n} = \IE \tilde X_t(t/n) + O(n^{-1}),
	\]
	where $\mu(u) := \IE \tilde X_0(u)$ is continuously differentiable with derivative $\partial_u\mu(u) = \IE \partial_u \tilde X_0(u)$.
\end{corollary}

\begin{corollary}[Covariance expansion, $M = 1$]\label{DP_cor_cov_diff} Fix $r > 0$. 
	Define the covariances $\gamma(u,r) := \Cov(\tilde X_t(u),\tilde X_{t-r}(u))$. Choosing $g:\IR^{r+1} \to \IR, g(y) = y_1y_{r+1}$, we obtain uniformly for $t = 1,...,n$:
	\begin{equation}
		\gamma_{t,n}(r) := \Cov(X_{t,n},X_{t-r,n}) = \gamma(\frac{t}{n},r) + O(n^{-1})\label{DP_cor_cov_diff_eq1}
	\end{equation}
	and $\gamma(u,r)$ is continuously differentiable with derivative
	\[
		\partial_u \gamma(u,r) = \Cov(\partial_u \tilde X_0(u), \tilde X_r(u)) + \Cov(\tilde X_0(u), \partial_u \tilde X_r(u)).
	\]
\end{corollary}

Similar expansions can be derived for higher-order cumulants and also for the Wigner-Ville spectrum (cf. \cite{martin1985}).

As a last application of Proposition \ref{DP_general_mean_diff}, we present an expansion of the distribution function of $X_{t,n}$ which may also be used to approximate quantiles of locally stationary processes.

\begin{example}[Expansion of the distribution function]\label{DP_cor_dist_diff}
	Suppose that Assumption \ref{DP_ass_general} holds with $q = 1$. Assume that the i.i.d. random variables $\varepsilon_t$, $t\in\IZ$ have a Lipschitz continuous and continuously differentiable distribution function $F_{\varepsilon}$ with Lipschitz constant $L_{\varepsilon}$ and derivative $f_{\varepsilon}$. Let the processes $X_{t,n}$ and $\tilde X_t(u)$ obey the recursion equations \reff{DP_rek_1} and \reff{DP_rek_2}.
	
	Assume that $(\varepsilon,y,u) \mapsto G_{\varepsilon}(y,u)$ is continuously differentiable and that the derivative $\partial_{\varepsilon}G_{\varepsilon}(y,u) \ge \delta_G > 0$ is uniformly bounded from below by some positive constant $\delta_G > 0$.
	 By the inverse function theorem we know that there exists a continuously differentiable inverse $x \mapsto H(x,y,u)$ of $\varepsilon \mapsto G_{\varepsilon}(y,u)$. Finally, assume that for all $x\in\IR$, the expressions
	\[
		C(x) := \sup_{u\in[0,1]}\sup_{y\not=y'}\frac{|H(x,y,u)-H(x,y',u)|}{|y-y'|_1}
	\]
	are finite.
	
	Put $Y_{t-1,n} = (X_{t-1,n},...,X_{t-p,n})'$, $\tilde Y_{t-1}(u) = (\tilde X_{t-1}(u),...,\tilde X_{t-p}(u))'$. In this situation it holds that the distribution function of $X_{t,n}$,
	\[
		F_{X_{t,n}}(x) = \IE\big[\IP(G_{\varepsilon_t}(Y_{t-1,n},t/n) \le x|\sF_{t-1})\big] = \IE\big[F_{\varepsilon}(H(x,Y_{t-1,n},t/n))\big]
	\]
	can be approximated by the distribution function $F_{\tilde X_t(u)}(x) :=  \IP(\tilde X_t(u) \le x)$ by
	\begin{align*}
		& |F_{X_{t,n}}(x) - F_{\tilde X_t(t/n)}(x)|\\
		&\le L_{\varepsilon}\|H(x,Y_{t-1,n},t/n) - H(x,\tilde Y_{t-1}(t/n),t/n)\|_1\\
		&\le L_{\varepsilon} C(x) \sum_{j=1}^{p}\|X_{t-j-1,n} - \tilde X_{t-j-1}(t/n)\|_1 \le p C_B L_{\varepsilon}\cdot C(x) \cdot n^{-1}
	\end{align*}
	Furthermore $u \mapsto F_{\tilde X_t(u)}(x)$ is differentiable with derivative
	\begin{align*}
		&\partial_u F_{\tilde X_t(u)}(x)\\
		&= \IE\big[ f_{\varepsilon}(H(x,\tilde Y_{t-1}(u),u))\cdot \big(\langle \partial_2 H(x,\tilde Y_{t-1}(u),u), \partial_u \tilde Y_{t-1}(u)\rangle + \partial_3 H(x,\tilde Y_{t-1}(u),u)\big)\big].
	\end{align*}
\end{example}

\section{Nonlinear locally stationary processes}
\label{DP_section2}


In this section we show in a sequence of theorems that the Markov processes given by \reff{DP_rek_1} and \reff{DP_rek_2} fulfill Assumption \ref{DP_ass_general} and the mixing conditions of Assumption \ref{DP_ass_general2}. 
Furthermore we prove that the derivative process can be obtained as the solution of a functional equation. The existence of these processes and their properties have previously been derived for tvAR models (cf. \cite{dahlhaus2012}), tvARCH models (cf. \cite{suhasini2006}) and random coefficient models (cf. \cite{suhasini2006b}). The situation in the present case is however different since the process \reff{DP_rek_1} is only defined by a recursion and the explicit solution is usually not available for the calculations. To prove the results, we state the following elementary assumptions on the recursion function $G_{\varepsilon}(y,u)$. Let $\partial_1 G$, $\partial_2 G$ denote the derivatives of $G$ w.r.t. $y$ and $u$, respectively.

\begin{assumption}\label{DP_ass1}
	In the model \reff{DP_rek_1}, \reff{DP_rek_2} we assume with $Y_{t-1,n} = (X_{t-1,n},...,X_{t-p,n})'$ and $\tilde Y_{t-1}(u) = (\tilde X_{t-1}(u),...,\tilde X_{t-p}(u))'$ that there exists $q > 0$, $\chi = (\chi_1,...,\chi_p) \in \IR^p_{\ge 0}$ with $|\chi|_1 = \sum_{i=1}^{p}\chi_i < 1$ and $y_0 \in \IR^p$ such that with $q' := \min\{q,1\}$:
	\begin{enumerate}[label=(L\arabic*),ref=(L\arabic*)]
		\item\label{DP_ass1_l1} $\sup_{u\in[0,1]}\|G_{\varepsilon_0}(y_0,u)\|_q < \infty$, and (with $|z|_{\chi,q'} := \big(\sum_{i=1}^{p}|z_i|^{q'}\cdot \chi_i\big)^{1/q'}$ the weighted $q'$-norm)
		\begin{equation}
	\sup_{u\in [0,1]}\sup_{y\not=y'}\frac{\|G_{\varepsilon_0}(y,u) - G_{\varepsilon_0}(y',u)\|_q}{|y-y'|_{\chi,q'}} \le 1.\label{DP_cond_1}
\end{equation}
		\item\label{DP_ass1_l2} $(y,u) \mapsto G_{\varepsilon}(y,u)$ is continuous for all $\varepsilon$, $\| \sup_{u\in[0,1]}|G_{\varepsilon_0}(y_0,u)|\ \|_q < \infty$, and
		\begin{equation}
			\Big\| \sup_{u\in [0,1]}\sup_{y\not=y'}\frac{|G_{\varepsilon_0}(y,u) - G_{\varepsilon_0}(y',u)|}{|y-y'|_{\chi,q'}}\Big\|_{q} \le 1.\label{DP_continuous_cond}
		\end{equation}
		\item\label{DP_ass1_l3} $(y,u) \mapsto G_{\varepsilon}(y,u)$ is continuously differentiable for all $\varepsilon$, $\| \sup_{u\in[0,1]}|\partial_2 G_{\varepsilon_0}(y_0,u)|\ \|_q < \infty$, and
		\begin{equation}
	C_i := \Big\|\sup_{u\in[0,1]}\sup_{y\not=y'}\frac{|\partial_i G_{\varepsilon_0}(y,u) - \partial_i G_{\varepsilon_0}(y',u)|_1}{|y-y'|_{1,q'}}\Big\|_{q} < \infty, \quad i = 1,2.\label{DP_cond_9}
\end{equation}
Furthermore, assume that either (a) \reff{DP_continuous_cond} holds for $q/2$ instead of $q$ or (b) $y \mapsto \partial_1 G_{\varepsilon}(y,u)$ is constant for all $\varepsilon,u$.
		\item\label{DP_ass1_l4} For some $0 < \alpha \le 1$, it holds that
		\begin{equation}
		 C:= \sup_{u\in [0,1]}\|C(\tilde Y_t(u))\|_q < \infty, \quad\mbox{ where }\quad C(y) := \sup_{u\not=u'}\frac{\|G_{\varepsilon_0}(y,u) - G_{\varepsilon_0}(y,u')\|_q}{|u-u'|^{\alpha}}. \label{DP_cond_4}
	\end{equation}
	\end{enumerate}
\end{assumption}

Let us briefly discuss the conditions in Assumption \ref{DP_ass1}.
\begin{remark}
\begin{itemize}
	\item[(i)] Note that \ref{DP_ass1_l1}-\ref{DP_ass1_l3} impose increasingly strong smoothness assumptions on the recursion function $G_{\varepsilon}(y,u)$. While \ref{DP_ass1_l1}-\ref{DP_ass1_l3} are directly verifiable, \ref{DP_ass1_l4} includes conditions on the stationary approximation $\tilde X_t(u)$. Note that the upcoming theorems also state properties of $\tilde X_t(u)$. Their results can be used to verify \ref{DP_ass1_l4}.
	\item[(ii)] The condition \ref{DP_ass1_l2} means that the mapping $y \mapsto G_{\varepsilon}(y,u)$ can be viewed as a contraction in the space of continuous functions $C[0,1]$ which in turn implies the a.s. continuity of the limit. \ref{DP_ass1_l3} is necessary to ensure that $y \mapsto G_{\varepsilon}(y,u)$ is a contraction in $C^1[0,1]$. 
	\item[(iii)] Condition \ref{DP_ass1_l3}(a) or (b) is necessary due to the product in \reff{DP_rek_3}. Note that by Hoelder's inequality, \ref{DP_ass1_l3}(a) follows from \reff{DP_continuous_cond} if $q \ge 2$. The inequality $|z|_{q} \le |z|_{q'}$ for $0 < q' \le q$, $z\in\IR^p$ implies that \ref{DP_ass1_l3}(a) is fulfilled if \reff{DP_continuous_cond} holds with $\sum_{i=1}^{p}\chi_i^{1/2} < 1$. 
	\item[(iv)] For $p > 1$, the conditions stated in Assumption \ref{DP_ass1} may lead to non-optimal restrictions on $G$ which is due to the general formulation. One way to circumvent this is by posing conditions on the $m$-th iteration of $G$ instead of $G$ itself. Some models like tvAR($p$) or tvARCH($p$) also allow a reformulation to a $p$-dimensional recursion with only one lag. Since our aim is to cover a wide range of models with simple conditions, we will not discuss these approaches in detail.
\end{itemize}
\end{remark}

\begin{remark}[Almost sure calculus v.s. $L^q$ calculus]\label{remark_differentiability}
As mentioned in the beginning of this paper, many of the statistical applications in Section \ref{DP_section3} can be proved by only assuming differentiability of $u \mapsto \tilde X_t(u)$ in $L^q$, an example was given in Proposition \ref{DP_prop_biasexp}(b). As pointed out by a referee, to obtain $L^q$ differentiability, Assumptions \ref{DP_ass1}\ref{DP_ass1_l2}, \ref{DP_ass1_l3} can be weakened. Technically, proving differentiability of $u \mapsto \tilde X_t(u)$ then corresponds to analyzing smoothness properties of a fixed point of the iteration \reff{DP_rek_2}. As analytic results in \cite{fixedpoint2017} suggest, the main difference is a decrease in smoothness assumptions that have to be posed on $G$, namely $G$ is no longer needed to be continuously differentiable but only differentiable a.e. and the suprema in \reff{DP_continuous_cond} and \reff{DP_cond_9} can be taken outside. This theory would also include $tvTAR$ processes (cf. Example \ref{DP_example_model}).

There are some drawbacks when using only $L^q$ calculus and no a.s. statements. As already mentioned in Remark \ref{remark_biasdet}, one has to pose slightly more smoothness conditions on $g$ if one wants to apply the theory to processes $g(Z_{t,n})$. Moreover, it seems that a Lindeberg-type condition which is used in the proof of the global CLT Theorem \ref{DP_prop_clt} can only be shown under Assumption \ref{DP_ass_general}\ref{DP_ass_general_s2} when only second moments are available. To ensure \ref{DP_ass_general}\ref{DP_ass_general_s2}, we have to ask for \ref{DP_ass1}\ref{DP_ass1_l2} in view of Theorem \ref{DP_thm_cont}.
\end{remark}


\textbf{Existence and uniqueness of $X_{t,n}$ and $\tilde X_t(u)$}. We now establish existence and uniqueness under mild contraction conditions.

\begin{proposition}\label{DP_rek_1_sol}(i) Existence of a stationary approximation: Suppose that Assumption \ref{DP_ass1}\ref{DP_ass1_l1} holds. Then for all $u \in [0,1]$, the recursion \reff{DP_rek_2} has an a.s. unique $\sF_t$-measurable, stationary and ergodic solution $\tilde X_t(u) = H(u,\sF_t)$ and  we have  with some $C >0$ and $0 < \rho < 1$:
\[
	\sup_{u\in[0,1]}\delta^{\tilde X(u)}_q(k) \le C \rho^k, \quad\quad \sup_{u\in[0,1]}\|\tilde X_0(u)\|_q < \infty.
\]
(ii) Existence of the nonstationary process: Under the above conditions, there exists an a.s. unique $\sF_t$-measurable solution of \reff{DP_rek_1} with $X_{t,n} = H_{t,n}(\sF_t)$, where $H_{t,n}$ are measurable functions. Furthermore, $\sup_{n\in\IN}\sup_{t=1,...,n}\|X_{t,n}\|_q < \infty$ and with some $C > 0$ and $0 < \rho < 1$:
	\[
		\sup_{n\in\IN}\delta^{X_{\cdot,n}}_q(k) \le C \rho^k.
	\]
\end{proposition}

The proof of (i) for fixed $u\in[0,1]$ is similar to the proof in \cite{shaowu2007}, Theorem 5.1. Since we state the results uniformly in $u\in[0,1]$, we will give the proof in the appendix for completeness. Since the definition of $X_{t,n}$ and $\tilde X_t(0)$ coincide for $t \le 0$, existence and uniqueness of $X_{t,n}$ follow from the existence and uniqueness of $\tilde X_t(0)$. Therefore, the existence statement in (ii) is an immediate corollary of (i).\\


\textbf{A uniform $L^q$ approximation:} We now prove that $X_{t,n}$ can be approximated by the stationary process $\tilde X_t(u)$ uniformly in a $L^q$-sense.

\begin{lemma}\label{DP_lemma_1}
	Suppose that Assumption \ref{DP_ass1}\ref{DP_ass1_l1},\ref{DP_ass1_l4} hold. Then
	\begin{equation}
		\sup_{u \not= u'}\frac{\| \tilde X_t(u) - \tilde X_{t}(u')\|_q}{|u - u'|^{\alpha}} \le \frac{C}{(1-|\chi|_1)^{1/q'}}.\label{DP_x_stat_hoelder}
	\end{equation}
	Furthermore, we have:
	\begin{equation}
		\sup_{t=1,...,n}\|X_{t,n} - \tilde X_{t}(t/n)\|_q \le Cp^{\alpha} \left(\frac{|\chi|_1}{(1-|\chi|_1)^2}\right)^{1/q'} \cdot n^{-\alpha}. \label{DP_x_lokstat_lipschitz}
	\end{equation}
\end{lemma}

Note that the approximation error in \reff{DP_x_lokstat_lipschitz} cannot be avoided - cf. \cite{dahlhaus2012}, (49), for the tvAR(1) case (with a different error due to different assumptions).\\



\textbf{Existence of continuous modifications and derivative processes:} Proposition \ref{DP_rek_1_sol} gives the almost sure uniqueness of $\tilde X_t(u)$ for each $u \in [0,1]$, but not continuity of $u \mapsto \tilde X_t(u)$ since this involves uncountably many points $u\in[0,1]$. In order to guarantee the existence of a continuous or even differentiable modification $\hat X_t(u)$ of $\tilde X_t(u)$ we have to impose stronger conditions on the recursion function $G$ in \reff{DP_rek_1} ($\hat X_t(u)$ is a modification of $\tilde X_t(u)$ if for all $u\in [0,1]$, $\hat X_t(u) = \tilde X_t(u)$ a.s.). A natural way would be to apply extensions of the Kolmogorov-Chentzov theorem, but they usually contain tradeoffs in their conditions between moment assumptions and smoothness of the process which usually leads to either strong moment or smoothness assumptions which may not be useful in practice. Furthermore it does not use the specific structure of the process which is known and we could not give a bound for moments of $\sup_{u\in[0,1]}|\hat X_t(u)|$. We therefore use a different approach.

\begin{theorem}[Existence of a continuous modification]\label{DP_thm_cont} 
Suppose that Assumption \ref{DP_ass1}\ref{DP_ass1_l2} holds. Then for each $t\in\IZ$, there exists a continuous modification $(\hat X_t(u))_{u \in [0,1]}$ of $(\tilde X_t(u))_{u \in [0,1]}$ from Proposition \ref{DP_rek_1_sol} with $\sup_{u \in [0,1]}|\hat X_t(u)| \in L^{q}$.
\end{theorem}

\begin{remark}
	In the case $G_{\varepsilon}(y,u) = \tilde G_{\varepsilon}(y,\theta_0(u))$ with some parameter curve $\theta_0:[0,1] \to \Theta$ (cf. Section \ref{DP_section4}), the supremum taken over $u\in[0,1]$ in \reff{DP_continuous_cond} restricts the parameter space $\Theta$. If additionally \ref{DP_ass1}\ref{DP_ass1_l3} is fulfilled, Theorem \ref{DP_thm_cont} also holds under the weaker condition $\sup_{u\in[0,1]}\Big\|\sup_{y\not=y'}\frac{|G_{\varepsilon_0}(y,u) - G_{\varepsilon_0}(y',u)|}{|y-y'|_{\chi,q'}}\Big\|_q \le 1$ which leads to larger admissible parameter spaces $\Theta$. For details see Proposition \ref{DP_prop_continuous} in the appendix.
\end{remark}


In the following we assume that $(y,u) \mapsto G_{\varepsilon}(y,u)$ is differentiable in both components. For the moment, assume that there exists a modification $(\hat X_t(u))_{u\in[0,1]}$ of the process $(\tilde X_t(u))_{u \in [0,1]}$ with differentiable paths and denote the derivative by $\partial_u \hat X_{t}(u)$. 
Then the following recursion equation for $D_t(u) = \partial_u \hat X_t(u)$, obtained by differentiating \reff{DP_rek_2} \underline{should hold} a.s.:
\begin{eqnarray}
	D_t(u) = \langle \partial_{1} G_{\varepsilon_t}(\tilde Y_{t-1}(u), u) , (D_{t-1}(u),...,D_{t-p}(u))'\rangle + \partial_2 G_{\varepsilon_t}(\tilde Y_{t-1}(u),u),\label{DP_rek_3}
\end{eqnarray}
This is shown in the next theorem. The first part is devoted to the existence of a solution $D_t(u)$ of the recursion \reff{DP_rek_3} given the existence of the process $\tilde X_t(u)$ from Theorem \ref{DP_rek_1_sol}; in the second part we prove that $\tilde X_t(u)$ has a differentiable modification with respect to $u$ and that the derivative coincides with $D_t(u)$. Both $\tilde X_t(u)$ and $D_t(u)$ are uniquely determined by \reff{DP_rek_2} and \reff{DP_rek_3}.

\medskip
\begin{theorem}[Existence of derivative processes]\label{DP_thm_diff}
	Suppose that Assumptions \ref{DP_ass1}\ref{DP_ass1_l2}, \ref{DP_ass1_l3} hold. Then the following statements are true.
	\begin{enumerate}
		\item[(i)] Existence of the first derivative process: 
		For all $u\in[0,1]$, the recursion \reff{DP_rek_3} has a unique stationary and ergodic solution $D_t(u) = \tilde H(u,\sF_t)$ with some measurable $H$ and it holds that
	\[
		\delta^{D(u)}_q(k) \le C \rho^k, \quad\quad \sup_{u\in[0,1]}\|D_t(u)\|_q < \infty
	\]
	with some $C > 0$, $0 < \rho < 1$.
	\item[(ii)] Differentiability:
	\begin{itemize}
		\item[(a)] There exists a continuously differentiable modification $(\hat X_t(u))_{u\in [0,1]}$ of the process $(\tilde X_t(u))_{u\in[0,1]}$ from Proposition \ref{DP_rek_1_sol} such that for all $u\in[0,1]$ it holds that $\partial_u \hat X_t(u) = D_t(u)$ a.s.
		\item[(b)] $\sup_{u\in[0,1]}|\partial_u \hat X_t(u)| \in L^{q}$.
	\end{itemize}
\end{enumerate}
\end{theorem}

Finally, let us summarize the results from this section in the following Corollary.

\begin{corollary}\label{DP_cor_ass_preservation}
	Let Assumption \ref{DP_ass1} be fulfilled. Then modifications of the a.s. unique solutions of \reff{DP_rek_1} and \reff{DP_rek_2} fulfill Assumption \ref{DP_ass_general} and \ref{DP_ass_general2}.
\end{corollary}

For some models it is possible to obtain explicit expressions for the corresponding derivative processes.

\begin{example}[Explicit representations for derivative processes]\noindent
\begin{enumerate}
	\item[(i)] The tvAR($p$) process $X_{t,n} = \sum_{j=1}^{p}a_{j}\big(\frac{t}{n}\big) X_{t-j,n} + \varepsilon_t$ has the corresponding stationary approximation $\tilde X_t(u) = \sum_{j=1}^{p}a_j(u) \tilde X_{t-j}(u) + \varepsilon_t$ which has an explicit representation $\tilde X_t(u) = \sum_{j=0}^{\infty}\psi_j(u)\cdot \varepsilon_{t-j}$ with differentiable $\psi_j:[0,1]\to \IR$ ($j=0,1,2,...)$. It is easy to see that $\partial_u \tilde X_t(u) = \sum_{j=0}^{\infty}\partial_u \psi_j(u) \cdot \varepsilon_{t-j}$ is the a.s. uniquely determined derivative process.
	\item[(ii)] Similarly to (i), it is easy to see that general linear processes $\tilde X_t(u) = \sum_{j=0}^{\infty}\psi_j(u)\cdot \varepsilon_{t-j}$ with differentiable $\psi_j:[0,1]\to \IR$ ($j = 0,1,2,...$) have derivative process $\partial_u \tilde X_t(u) = \sum_{j=1}^{\infty}\partial_u \psi_j(u)\cdot \varepsilon_{t-j}$ under appropriate summability conditions.
	\item[(iii)] For tvARCH($p$) processes, explicit expressions for the derivative processes were obtained in \cite{suhasini2006}.
\end{enumerate}
\end{example}

In the following we will write $\tilde X_t(u)$ even if we mean the differentiable modification to keep notation simple. Since all our results only involve countably many observations, this will not cause any problems.

\textbf{Higher order derivative processes:} Under additional assumptions, one can show uniform $L^q$ Hoelder properties of the first derivative process:

\begin{proposition}[Hoelder property of the first derivative process]\label{DP_hoelder_derivative}
	Suppose that Assumption \ref{DP_ass1}\ref{DP_ass1_l2},\ref{DP_ass1_l3} hold. Additionally assume that for some $1 \ge \alpha_2 > 0$ and $i = 1,2$ it holds component-wise:
	\begin{equation}
		D_i := \sup_u \|D_i(\tilde Y_t(u))\|_q < \infty, \quad D_i(y) := \sup_{u\not=u'}\frac{\|\partial_i G_{\varepsilon_0}(y,u) - \partial_i G_{\varepsilon_0}(y,u')\|_{q}}{|u-u'|^{\alpha_2}}\label{DP_cond_20}
	\end{equation}
	Then
	\[
		\sup_{u\not= u'}\frac{\|\partial_u \tilde X_t(u) - \partial_u \tilde X_{t}(u')\|_{q/2}}{|u-u'|^{\alpha_2}} \le C.
	\]
	with some constant $C > 0$.
\end{proposition}

If $\tilde X_t(u)$ has a twice continuously differentiable modification and $(y,u) \mapsto G_{\varepsilon}(y,u)$ is twice continuously differentiable, then the following recursion equation for $\partial_u^2 \tilde X_t(u)$ should hold:
\begin{eqnarray}
	\partial_u^2 \tilde X_t(u) &=& \langle \partial_1 G_{\varepsilon_t}(\tilde Y_{t-1}(u), u), \partial_u^2 \tilde Y_{t-1}(u)\rangle + \langle \partial_1^2 G_{\varepsilon_t}(\tilde Y_{t-1}(u),u) \partial_u \tilde Y_{t-1}(u), \partial_u \tilde Y_{t-1}(u)\rangle\nonumber\\
	&&\quad + 2\langle \partial_1 \partial_2 G_{\varepsilon_t}(\tilde Y_{t-1}(u),u), \partial_u \tilde Y_{t-1}(u)\rangle + \partial_2^2 G_{\varepsilon_t}(\tilde Y_{t-1}(u),u).\label{DP_rek_5}
\end{eqnarray}
Using the same techniques as in Theorem \ref{DP_thm_diff}, one can find similar conditions as in Assumption \ref{DP_ass1} such that a second (or even higher) order derivative process $\partial_u^2 \tilde X_t(u)$ exists. Let us point out an interesting anomaly in the case of second order derivatives that is also existent for higher order derivatives: Due to the additional products in \reff{DP_rek_5} it turns out that, in general, one has to assume $2q$-th moments of $\tilde X_t(u)$ to guarantee the existence of the $q$-th moment of $\partial_u^2 \tilde X_t(u)$. The formalization of this is beyond the scope of this paper, but in Proposition \ref{DP_hoelder_derivative} one already can see the imbalance of moments in the assumption and the obtained result.\\

\textbf{A simulation study: } To quantify the quality of the approximations given in Lemma \ref{DP_lemma_1} and Proposition \ref{DP_cor_taylor_exp}, we consider the tvARCH(1) model
\[
	X_{t,n} := \Big(a_0 + a_1\big(\frac{t}{n}\big)X_{t-1,n}^2\Big)^{1/2}\varepsilon_t
\]
with $a_0 := 0.2$, $a_1(u) = 0.95u^2$ and $\varepsilon_0 \sim N(0,1)$. Note that if $t/n$ tends to 1, the values of $X_{t,n}$ are more dependent to each other than for smaller values of $t/n$. We generated realizations of $X_{t,n}$, $\tilde X_t(\frac{t}{n})$ with $n = 500$ (see Figure \ref{DP_figure1_simu}(a),(b) for a realization of $X_{t,n}$ and $X_{t,n} - \tilde X_t(\frac{t}{n})$). In Figure \ref{DP_figure1_simu}(c) we have the plotted empirical 5\%- and 95\%-quantile curves of the difference $X_{t,n} - \tilde X_t(\frac{t}{n})$ for $N = 1000$ replications. It can be seen that with stronger dependence, the quality of the approximation $X_{t,n} \approx \tilde X_t(\frac{t}{n})$ gets worse as it is suggested by the bound in Lemma \ref{DP_lemma_1}.\\
Secondly we consider the approximation quality of $\tilde X_t(t/n)$ by $\tilde X_t(u)$ and $\tilde X_t(u) + (\frac{t}{n}-u)\partial_u \tilde X_t(u)$, respectively. Since these approximations are only working locally (for $|t/n-u| \ll 1$), we compare them by dividing the whole time line $t = 1,...,n$ into subsets $(u_i - b,u_i+b]$, where $b = 25$ and $u_i = (2i-1)b$ for $i = 1,...,10$. In Figure \ref{DP_figure1_simu}(d) empirical 5\%- and 95\%-quantile curves obtained from $N = 1000$ replications for the differences $\tilde X_t(\frac{t}{n})-\tilde X_t(u_i)$ and $\tilde X_t(\frac{t}{n})-\tilde X_t(u_i) - (\frac{t}{n} - u_i)\partial_u \tilde X_t(u_i)$ (where $t\in(u_i - b,u_i+b]$) are depicted, respectively. We emphasize that the improvement of the (pointwise) approximation $\tilde X_t(\frac{t}{n})$ by taking into account the derivative process is remarkable. However, both approximations again get worse if the dependence of $X_{t,n}$ to earlier values increases.

\begin{figure}[h!]
	\centering
	\begin{tabular}{cc}
		(a) \includegraphics[width=5cm]{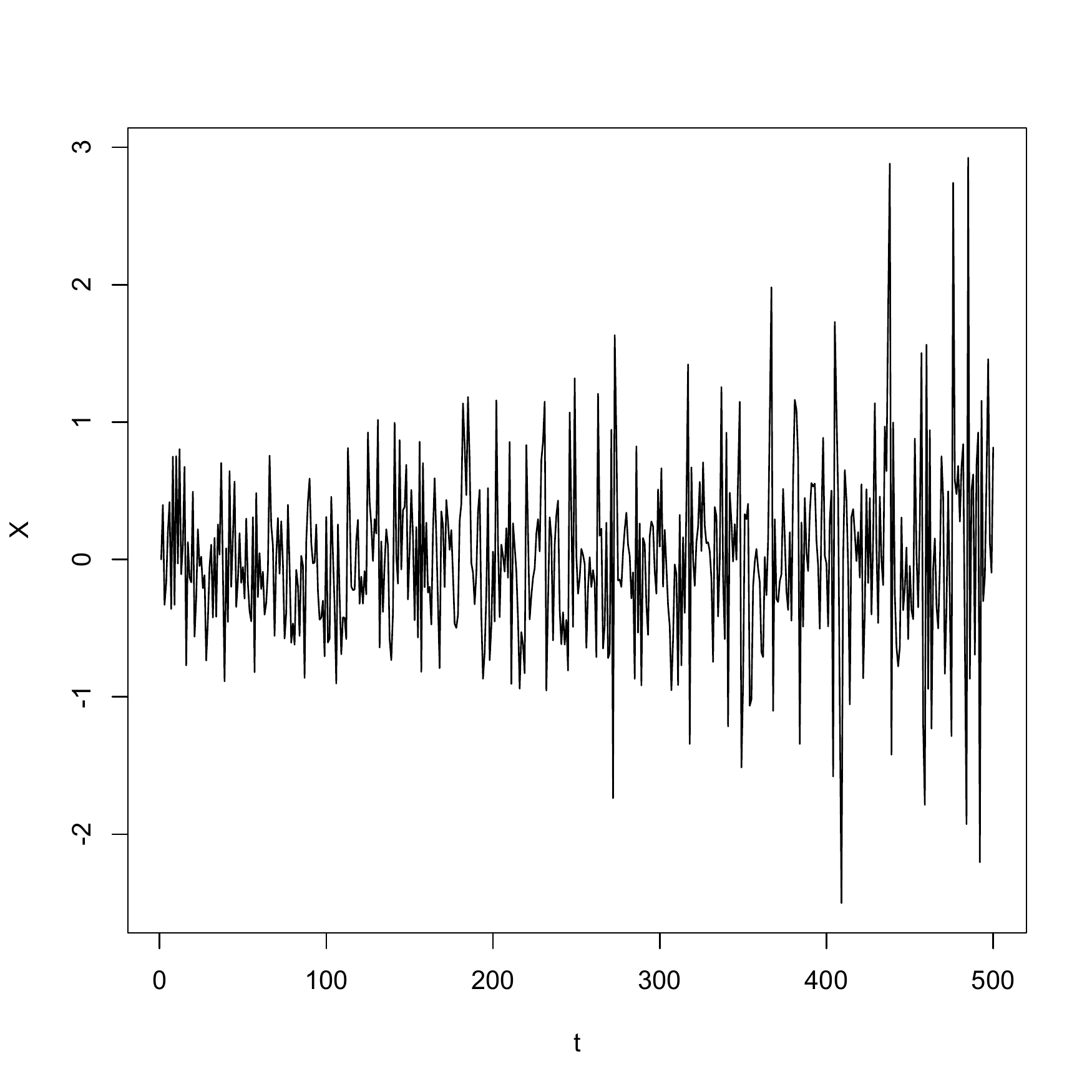} & (b) \includegraphics[width=5cm]{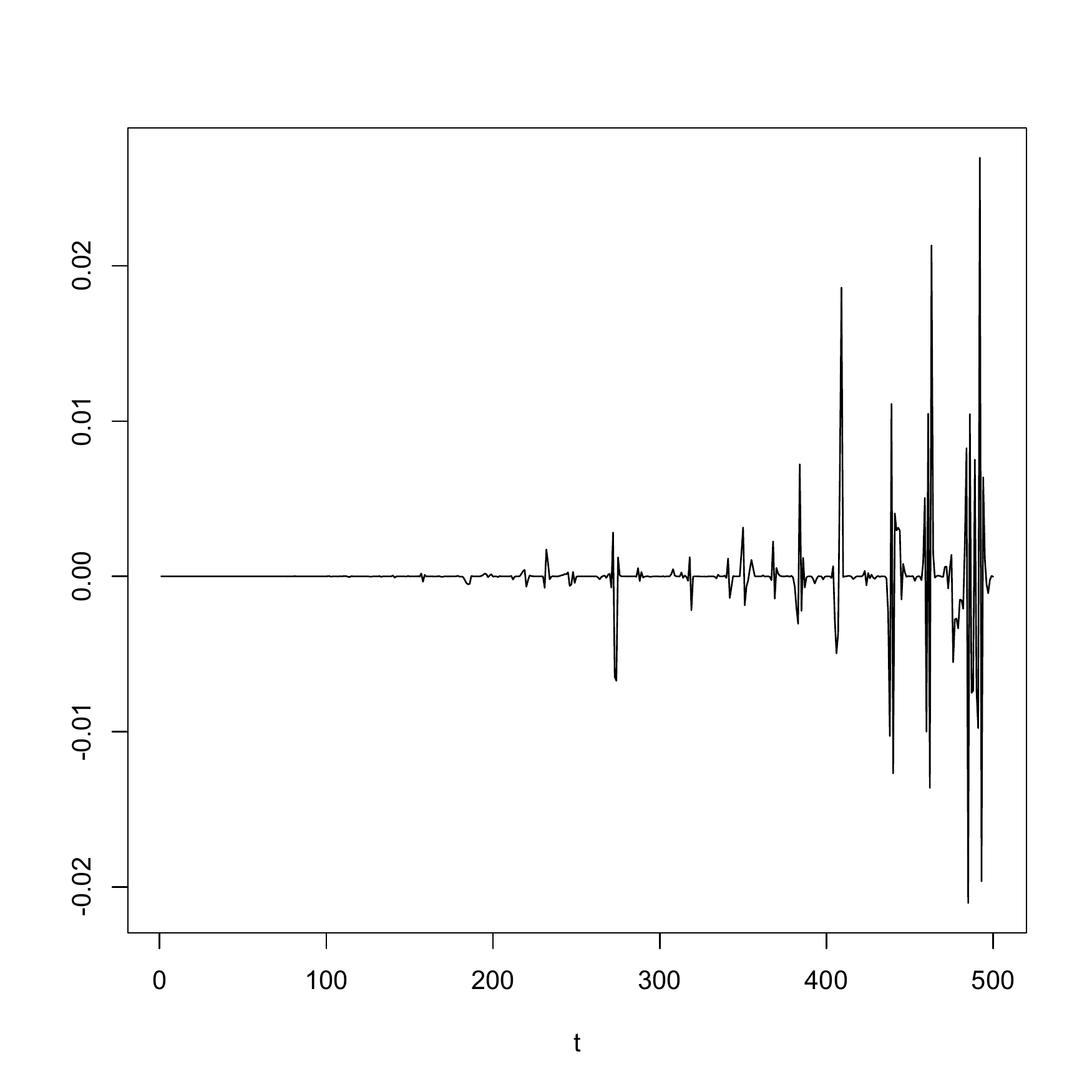}\\
		(c)  \includegraphics[width=5cm]{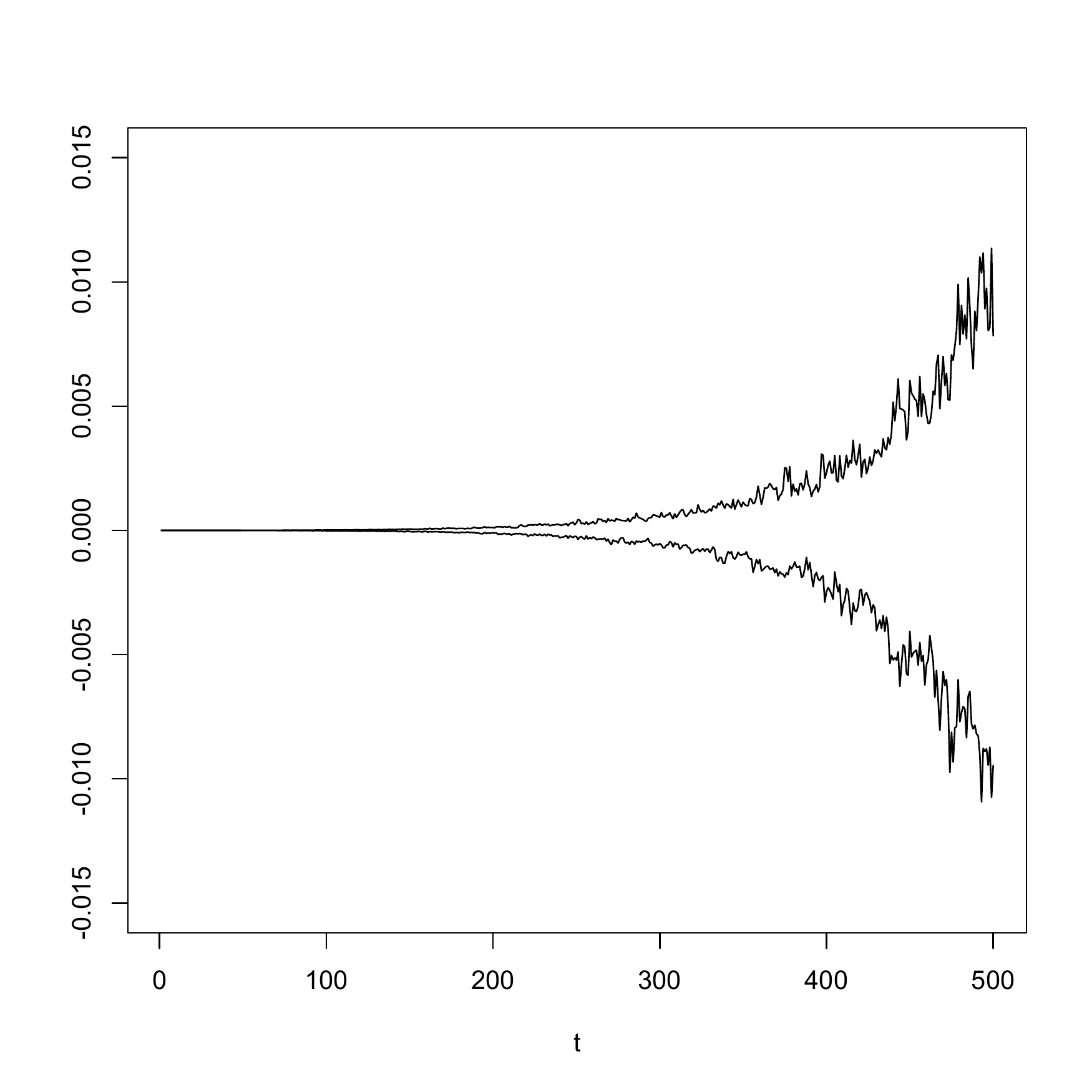} & (d) \includegraphics[width=5cm]{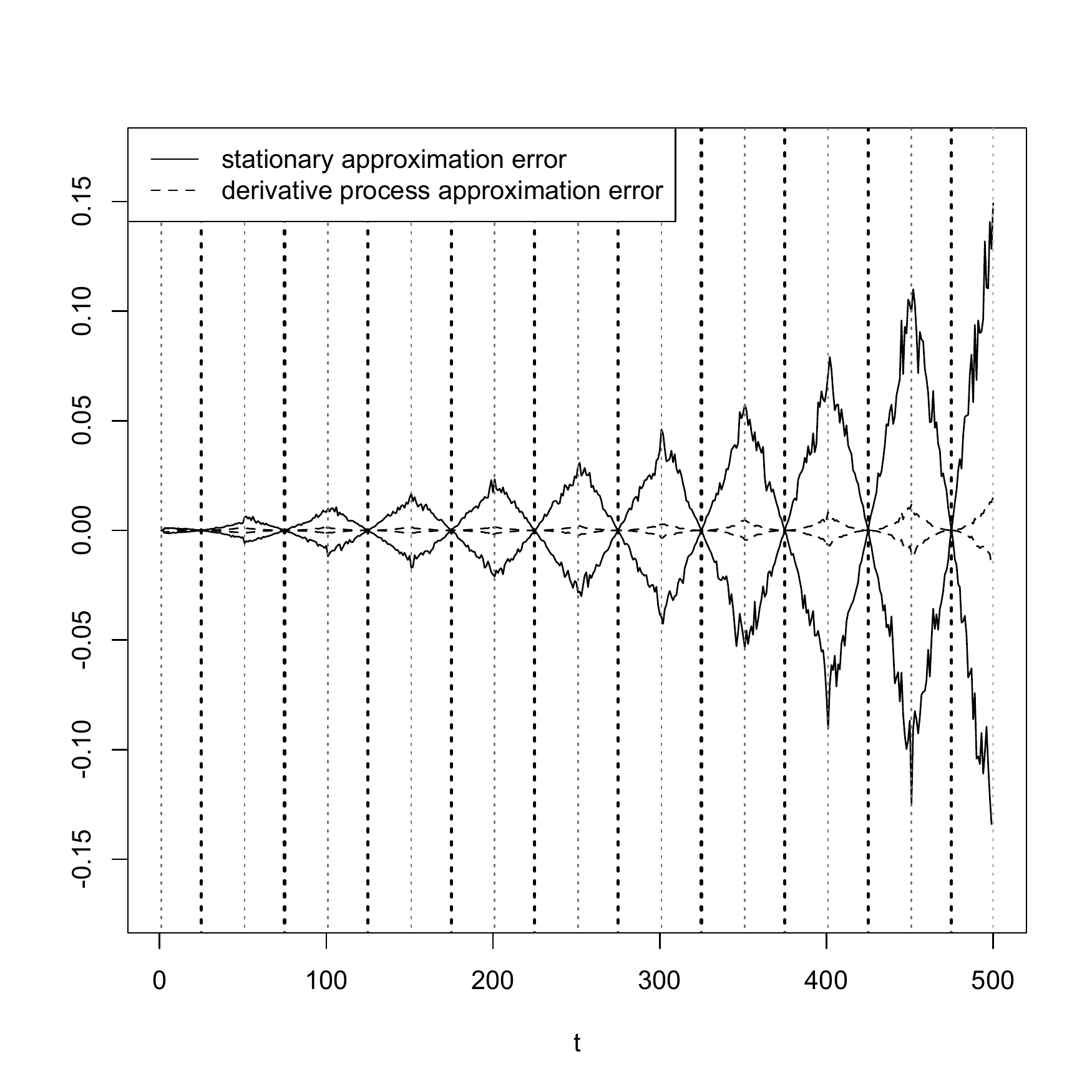}
	\end{tabular}
	\caption{Top: (a) Realization of one $X_{t,n}$, $t = 1,...,n$. (b) Difference $X_{t,n} - \tilde X_t(\frac{t}{n})$ for one realization.
	Bottom: (c) empirical 5\%- and 95\%-quantile curves of $X_{t,n}-\tilde X_t(\frac{t}{n})$ for $N = 1000$ replications. (d) Solid and Dashed: empirical 5\%- and 95\%-quantile curves of $\tilde X_{t}(\frac{t}{n})-\tilde X_t(u_i)$ and $\tilde X_{t}(\frac{t}{n}) - \tilde X_t(u_i) - (\frac{t}{n}-u_i)\partial_u \tilde X_t(u_i)$ for $t \in (u_i - b, u_i + b]$ (grey thin vertical dotted lines) and $N = 1000$ replications, respectively. Here, $b = 25$ and $u_i = (2i-1)b$ (black thick vertical dotted lines), $i = 1,...,10$.}
	\label{DP_figure1_simu}
\end{figure}


\section{Application to Maximum Likelihood estimation}
\label{DP_section4}

In this section we investigate the asymptotic properties of maximum likelihood estimates for parameter curves of locally stationary models which can be written in the form \reff{DP_rek_1}. The results are in particular derived by using the asymptotic results and the differential calculus of Section 2 and 3. More precisely we investigate the recursively defined model
\begin{equation}
	X_{t,n} = \tilde G_{\varepsilon_t}\big(X_{t-1,n},...,X_{t-p,n}, \theta_0\big(\frac{t}{n}\big)\big), \quad t = 1,...,n.\label{DP_rek_morespecific}
\end{equation}
where now the function $G_{\varepsilon}(y,u)$ from \reff{DP_rek_1} has been replaced by $\tilde G_{\varepsilon}(y,\theta_0(u))$ with the unknown parameter curve $\theta_0: [0,1] \to \Theta \subset \IR^d$ which is to be estimated. Our goal is to obtain estimators for $\theta_0(\cdot)$ based on $X_{t,n}$, $t = 1,...,n$ with a quasi maximum likelihood approach.

Suppose for the moment that $\varepsilon \mapsto G_{\varepsilon}(y,\theta)$ is continuously differentiable for all $\varepsilon,y,u$ and that the derivative $\partial_{\varepsilon}\tilde G_{\varepsilon}(y,\theta) \ge \delta_{G} > 0$ is bounded uniformly from below with some constant $\delta_G > 0$. This ensures that the new innovation $\varepsilon_t$ has an impact on the value of $X_{t,n}$ which is not too small. Under these conditions, there exists a continuously differentiable inverse $x\mapsto H(x,y,\theta)$ of $\varepsilon \mapsto G_{\varepsilon}(y,\theta)$ (see also Example \ref{DP_cor_dist_diff}).\\
Suppose that $\varepsilon_0$ has a continuous density $f_{\varepsilon}$. The negative conditional log likelihood of  $X_{t,n} = x$ given $(X_{t-1,n},...,X_{t-p,n}) = y$ and $\theta_0(\cdot) \equiv \theta$ is then
\begin{equation}
	\l(x,y,\theta) = -\log f_{\varepsilon}(H(x,y,\theta)) - \log \partial_x H(x,y,\theta).\label{DP_likelihood_standard_choice}
\end{equation}
In the following derivations, we do not make use of the specific structure of $\l$. This means especially that we allow for model misspecifications due to a false density $f_{\varepsilon}$. Many authors prefer the case of a Gaussian density $f_{\varepsilon}(x) = (2\pi)^{-1/2}\exp(-x^2/2)$ because then a minimizer $\theta$ of $\l$ can be interpreted as a minimum (quadratic) distance estimator (see \cite{dahlhausgiraitis1998} in the tvAR case, \cite{suhasini2006} in the tvARCH case).

Based on this we define $\l_{t,n}(\theta) := \l(X_{t,n}, Y_{t-1,n},\theta)$. Let $b\in(0,1)$ be a bandwidth $K$ a kernel function as considered in Assumption \ref{DP_ass_kernel}. We define the local negative log conditional likelihood
\[
	L_{n,b}(u,\theta) := \frac{1}{n}\sum_{t=p+1}^{n}K_b\Big(\frac{t}{n}-u\Big)\cdot \l_{t,n}(\theta).
\]
For $u \in [0,1]$, the estimator of $\theta_0(u)$ is defined via
\begin{equation}
	\hat \theta_b(u) := \arg \min_{\theta \in \Theta}L_{n,b}(u,\theta).\label{DP_mle_def}
\end{equation}

\textbf{Asymptotic results: } We will now discuss conditions such that $\hat \theta_b(\cdot)$ is consistent and asymptotically normal. A convenient way to formulate these results is to make a structural assumption on $\l$: We suppose that $\l$ is Lipschitz continuous in its components with at most polynomially increasing Lipschitz constant. To make this more precise, we introduce the class $\tilde \sL_{p+1}(M,C)$ of functions using the definition of $\sL_r(M,C)$ from Definition \ref{DP_hoelder_def}.

\begin{definition}[The class $\tilde \sL_{p+1}(M,C)$] We say that a function $g:\IR^{p+1} \times \Theta \to \IR$ is in the class $\tilde \sL_{p+1}(M,C)$ with $C = (C_{z},C_{\theta})$ and constants $C_{z},C_\theta \ge 0$ and $M \ge 0$ if
	for all $z\in \IR^{p+1},\theta\in\Theta$ it holds that $g(\cdot,\theta) \in \sL_{p+1}(M,C_z)$ and $g(z,\cdot) \in \sL_{d}\big(0,C_{\theta}(1+|z|_1^{M+1})\big)$.
\end{definition}

As in Section \ref{DP_section35}, a generalization to Hoelder-type conditions \reff{DP_hoelder_prop_polynom_2} in the first component of $g$ is possible.
It turns out in Theorem \ref{DP_mle_consistency} that the (pointwise) consistency of $\hat \theta_b$ can be obtained by posing conditions on the likelihood of the corresponding stationary process which is defined via $L(u,\theta) := \IE[\tilde \l_t(u,\theta)]$ with $\tilde \l_{t}(u,\theta) := \l(\tilde X_{t}(u), \tilde Y_{t-1}(u),\theta)$. Especially if $\l$ is taken to be of the form \reff{DP_likelihood_standard_choice} with $f_{\varepsilon}$ the standard Gaussian density, the properties of $L(u,\theta)$ are usually well-known from the maximum likelihood theory of the stationary process $X_t(\theta)$ and therefore are easy to verify (see also Example \ref{DP_mle_standard_example}).

To prove consistency, we have to inspect $L_{n,b}(u,\theta)$ which consists of summands of the form $\ell(X_{t,n},Y_{t-1,n},\theta)$. Since $\ell \in \tilde \sL_{p+1}(M,C)$, these terms behave like polynomials of degree $M+1$ in $X_{t,n}$. We mainly need the law of large numbers Proposition \ref{DP_prop_ergodic}(ii) and the statements about deterministic bias expansions Proposition \ref{DP_prop_biasexp}. The conditions therein require Assumption \ref{DP_ass_general}\ref{DP_ass_general_s1} with $q = M+1$. Translated to the Markov process setting in this section, we have to assume \ref{DP_ass1}\ref{DP_ass1_l1}, \ref{DP_ass1_l4} with $q = M+1$ by the results from Section \ref{DP_section2}.

\begin{theorem}[Pointwise and uniform consistency of $\hat \theta_b$]\label{DP_mle_consistency} Let Assumption \ref{DP_ass_kernel} hold. Assume that $\l \in \tilde \sL_{p+1}(M,C)$ for some $M \ge 0$. Suppose that Assumption \ref{DP_ass1}\ref{DP_ass1_l1}, \ref{DP_ass1_l4} holds with some $1 \ge \alpha > 0$ and $q = M+1$.\\
Furthermore suppose that for all $u\in[0,1]$, $\theta_0(u)\in \mbox{int}(\Theta)$ is the unique minimizer of $L(u,\theta)$ over $\theta \in \Theta$, where $\Theta \subset \IR^d$ is  a compact set. Then:
\begin{enumerate}
	\item[(i)] For all $u \in (0,1)$ with $b \to 0$ and $bn \to \infty$:
\[
	\hat \theta_b(u) \pto \theta_0(u).
\]
	\item[(ii)] If additionally $q > M+1$ and $b = o(n^{1-\frac{M+1}{q}})$ and $\theta_0(\cdot)$ is continuous, we have
\[
	\sup_{u\in[\frac{b}{2},1-\frac{b}{2}]}|\hat \theta_b(u) - \theta_0(u)| \pto 0.
\]
\end{enumerate}
\end{theorem}

\begin{remark}
	Note that in nearly all cases, the conditions of Assumption \ref{DP_ass1}\ref{DP_ass1_l4} assumed in Theorem \ref{DP_mle_consistency} implicitly impose a Hoelder continuity condition on $\theta_0(\cdot)$; see also Example \ref{DP_mle_standard_example}.
\end{remark}

\begin{proof}[Proof of Theorem \ref{DP_mle_consistency}] (i) For fixed $u\in[0,1]$ and $\theta \in \Theta$, we have $\l(\cdot,\cdot,\theta) \in \sL_{p+1}(M,C_z)$. Application of Theorem \ref{DP_prop_ergodic}(ii) (see also Remark \ref{DP_remark_ergodic}(ii)) leads to
\[
	L_{n,b}(u,\theta) = \frac{1}{n}\sum_{t=1}^{n}K_b\Big(\frac{t}{n}-u\Big) \cdot \l(X_{t,n},Y_{t-1,n},\theta) \pto \IE \l(\tilde X_{t}(u), \tilde Y_{t-1}(u),\theta) = L(u,\theta).
\]
The function $\theta \mapsto L(u,\theta)$ is continuous since
\begin{eqnarray*}
	|L(u,\theta) - L(u,\theta')| &\le& \| \l(\tilde X_t(u),\tilde Y_{t-1}(u), \theta) - \l(\tilde X_t(u), \tilde Y_{t-1}(u),\theta')\|_1\\
	&\le& C_{\theta}\cdot |\theta - \theta'|_1\cdot \big(1 + \Big(\sum_{j=0}^{p}\|\tilde X_t(u)\|_{M+1}\Big)^{M+1}\big).
\end{eqnarray*}
It remains to show stochastic equicontinuity of $L_{n,h}(u,\theta)$: Define $h:\IR^{p+1}\to\IR$, $h(z) = C_{\theta}(1+|z|_1^{M+1})$. Fix $\eta > 0$. We have
\begin{eqnarray*}
	|L_{n,b}(u,\theta) - L_{n,b}(u,\theta')| &\le& |\theta-\theta'|_1 \cdot \frac{1}{n}\sum_{t=1}^{n}\Big|K_b\Big(\frac{t}{n}-u\Big)\Big|\cdot h(X_{t,n},Y_{t-1,n}).
\end{eqnarray*}
Obviously, $h \in \sL_{p+1}(M,C)$ with some constant $C > 0$. Application of Proposition \ref{DP_prop_ergodic}(ii) to $K/\int K \dif x$ and $h$ (see also Remark \ref{DP_remark_ergodic}(ii)) yields for all $u\in (0,1)$:
\begin{equation}
	\frac{1}{n}\sum_{t=1}^{n}\Big|K_b\Big(\frac{t}{n}-u\Big)\Big|\cdot h(X_{t,n},Y_{t-1,n}) \pto \int |K| \dif x \cdot \IE h(\tilde X_t(u), \tilde Y_{t-1}(u)) =: c(u).\label{DP_mle_consistency_eq1}
\end{equation}
Choosing $\delta = \frac{\eta}{2c(u)}$ yields
\begin{eqnarray*}
	&& \IP\Big(\sup_{|\theta - \theta'|_1 \le \delta}|L_{n,b}(u,\theta) - L_{n,b}(u,\theta')| > \eta\Big)\\
	&\le& \IP\Big(\Big|\frac{1}{n}\sum_{t=1}^{n}\Big|K_b\Big(\frac{t}{n}-u\Big)\Big|\cdot h(X_{t,n},Y_{t-1,n}) - c(u)\Big| > c(u)\Big) \to 0\quad (n\to\infty).
\end{eqnarray*}
This gives $\sup_{\theta \in \Theta}|L_{n,b}(u,\theta) - L(u,\theta)| \pto 0$. By standard arguments (cf. \cite{vandervaart1998}, Theorem 5.7), the proof is complete.

To prove (ii), we apply Theorem \ref{DP_prop_ergodic}(iii) on $\l(X_{t,n},Y_{t-1,n},\theta)$ with $\tilde q = \frac{q}{M+1} > 1$ (see also Remark \ref{DP_remark_ergodic}(ii)) to obtain for each $\theta \in \Theta$ that
\[
	\sup_{u\in[0,1]}\big|L_{n,b}(u,\theta) - \IE L_{n,b}(u,\theta)\big| = O_p(n^{\frac{M+1}{q}-1}b^{-1}).
\]
By Proposition \ref{DP_prop_biasexp} and the bounded variation of $K$, we have $\sup_{u\in[\frac{b}{2},1-\frac{b}{2}]}|\IE L_{n,b}(u,\theta) - L(u,\theta)| = O(b^{\alpha}) + O((nb)^{-1})$, which yields
\[
	\sup_{u\in[\frac{b}{2},1-\frac{b}{2}]}\big|L_{n,b}(u,\theta) - L(u,\theta)\big| \pto 0.
\]
Similarly we can strengthen \reff{DP_mle_consistency_eq1} to
\[
	 \sup_{u\in[\frac{b}{2},1-\frac{b}{2}]}\Big|\frac{1}{n}\sum_{t=1}^{n}\Big|K_b\Big(\frac{t}{n}-u\Big)\Big|\cdot h(X_{t,n},Y_{t-1,n}) - c(u)\Big| \pto 0.
\]
Now define $c := \inf_u c(u) > 0$ (by continuity of $c(\cdot)$). Choosing $\delta = \frac{\eta}{2c}$ yields
\begin{eqnarray*}
	&& \IP\Big(\sup_{u\in[\frac{b}{2},1-\frac{b}{2}]}\sup_{|\theta - \theta'|_1 \le \delta}|L_{n,b}(u,\theta) - L_{n,b}(u,\theta')| > \eta\Big)\\
	&\le& \IP\Big(\sup_{u\in[\frac{b}{2},1-\frac{b}{2}]}\Big|\frac{1}{n}\sum_{t=1}^{n}\Big|K_b\Big(\frac{t}{n}-u\Big)\Big|\cdot h(X_{t,n},Y_{t-1,n}) - c(u)\Big| > c\Big) \to 0\quad (n\to\infty).
\end{eqnarray*}
So we have seen that $\sup_{u\in[\frac{b}{2},1-\frac{b}{2}]}\sup_{\theta \in \Theta}|L_{n,b}(u,\theta) - L(u,\theta)| \pto 0$. Standard arguments give the result (see also the appendix).
\end{proof}

We now provide a central limit theorem for $\hat \theta_b$ including a bias decomposition. Let $\nabla$ denote the derivative with respect to $\theta$. To use a standard Taylor expansion from M-estimation theory, we need the existence of $\nabla \ell \in \tilde \sL_{p}(M',C')$ and $\nabla^2 \ell \in \tilde \sL_p(M'',C'')$. To apply the local central limit theorem \ref{DP_clt_local} to $\nabla L_{n,b}(u,\theta)$, we additionally need Assumption \ref{DP_ass_general_s2} with $q = 2(M'+1)$ which is fulfilled if Assumption \ref{DP_ass1}\ref{DP_ass1_l2} is valid for $q = 2(M'+1)$.

\begin{theorem}[A central limit theorem for $\hat \theta_b$]\label{DP_mle_clt} Additionally to Theorem \ref{DP_mle_consistency}(i), suppose that $\ell$ is twice continuously differentiable w.r.t. $\theta$ and
\begin{itemize}
	\item $\nabla \l \in \tilde \sL_{p+1}(M',C')$ for some $M' \ge 0$, 
	$\nabla^2 \l \in \tilde \sL_{p+1}(M'',C'')$ for some $M'' \ge 0$, 
	\item Assumption \ref{DP_ass1}\ref{DP_ass1_l1}, \ref{DP_ass1_l4} is fulfilled with $q = \max\{2(M'+1), M'' + 1\}$ and some $1 \ge \alpha' > 0$, Assumption \ref{DP_ass1}\ref{DP_ass1_l2} is fulfilled with $q = 2(M'+1)$.
\end{itemize}
Assume that the model is correct in the weak sense that $\IE[\nabla \tilde \l(u,\theta_0(u))|\sF_{t-1}] = 0$, i.e. $\nabla \tilde \l_t(u,\theta_0(u))$ is a martingale difference sequence with respect to $(\sF_t)$. Let $b\to 0$, $nb \to \infty$ and $bn^{1-2\alpha} = o(1)$.\\
(i) Then we have for $nb^{1+2\alpha'} = o(1)$:
\begin{equation}
	\sqrt{nb}\big(\hat \theta_b(u) - \theta_0(u)\big) \dto N\big(0, \int K(x)^2 \dif x\cdot V(u)^{-1}I(u)V(u)^{-1} \big),\label{DP_mle_clt_eq1}
\end{equation}
where $I(u) := \IE[\nabla \tilde \l_t(u,\theta_0(u)) \nabla \tilde \l_t(u,\theta_0(u))']$ and $V(u) := \nabla^2 L(u,\theta_0(u))$ is assumed to be positive definite.\\
(ii) If additionally $\nabla \l$ is continuously differentiable and Assumption \ref{DP_ass1}\ref{DP_ass1_l3} is fulfilled for $q = M'+1$, then we have for  $nb^3 = O(1)$:
\begin{eqnarray*}
	&&\sqrt{nb}\big(\hat \theta_b(u) - \theta_0(u) - b \cdot V^{-1}(u) \IE \partial_u \nabla \l(\tilde X_t(u), \tilde Y_{t-1}(u),\theta_0(u))\cdot \int K(x) x \dif x\big)\\
	&\dto& N\big(0, \int K(x)^2 \dif x\cdot V(u)^{-1}I(u)V(u)^{-1} \big),
\end{eqnarray*}
so the result \reff{DP_mle_clt_eq1} remains true if $K$ is symmetric.
\end{theorem}
\begin{proof}[Proof of Theorem \ref{DP_mle_clt}:] 
The conditions on $\nabla^2 \l $ imply that $u \mapsto \nabla^2 L(u,\theta) = \IE[\nabla^2 \tilde \l_t(u,\theta)]$ is continuous. Note that by Theorem \ref{DP_clt_local}, we have
\begin{eqnarray*}
	&& \sqrt{nb}\nabla L_{n,b}(u,\theta_0(u))\\
	&=& \frac{1}{\sqrt{nb}}\sum_{t=p+1}^{n}K\Big(\frac{t/n-u}{b}\Big)\Big(\nabla \l(X_{t,n},Y_{t-1,n},\theta_0(u)) - \IE \nabla \l(X_{t,n},Y_{t-1,n},\theta_0(u))\Big)\\
	&\dto& N\big(0,\int K(x)^2 \dif x \cdot \sigma^2(u)\big),
\end{eqnarray*}
where $\sigma^2(u) = \big\|\sum_{l=0}^{\infty}P_0 \nabla \tilde \l_t(u,\theta_0(u))\big\|_2^2 = I(u)$ by the martingale difference property. Furthermore, 
\begin{eqnarray*}
	&& \frac{1}{\sqrt{nb}}\sum_{t=1}^{n}K\Big(\frac{t/n-u}{b}\Big)\IE \nabla \l(X_{t,n},Y_{t-1,n},\theta_0(u))\\
	&=& \frac{1}{\sqrt{nb}}\sum_{t=1}^{n}K\Big(\frac{t/n-u}{b}\Big)\big(\IE \nabla \l(X_{t,n},Y_{t-1,n},\theta_0(u)) - \IE\nabla \l(\tilde X_{t}(u),\tilde Y_{t-1}(u),\theta)\big)\\
	&=& \frac{1}{\sqrt{nb}}\sum_{t=1}^{n}K\Big(\frac{t/n-u}{b}\Big)\big(\IE \nabla \l(\tilde X_{t}(t/n),\tilde Y_{t-1}(t/n),\theta_0(u)) - \IE\nabla \l(\tilde X_{t}(u),\tilde Y_{t-1}(u),\theta)\big)\\
	&&\quad\quad + O(\sqrt{n^{1-2\alpha}b}).
\end{eqnarray*}
Proposition \ref{DP_prop_biasexp} gives that the first term is $O(\sqrt{nb^{1+2\alpha'}})$ in the case of (i). In case of (ii), the first term has the form 
\[
	\frac{1}{\sqrt{nb}}\sum_{t=1}^{n}K\Big(\frac{t/n-u}{b}\Big)\cdot \big(\frac{t}{n} - u\big)\cdot \IE\big[\partial_u \nabla \l(\tilde X_t(u), \tilde Y_{t-1}(u),\theta)\big]\big|_{\theta = \theta_0(u)} + O((nb)^{-1/2}) + o(\sqrt{nb^3})
\]
and is $o(\sqrt{nb^3}) + O((nb)^{-1/2})$ if $K$ is symmetric. Since $\nabla^2\l$ fulfills the same assumptions as $\l$ in Theorem \ref{DP_mle_consistency}, we can mimic its proof and obtain
\[
	\sup_{\theta \in \Theta}|\nabla^2 L_{n,b}(u,\theta) - \nabla^2 L(u,\theta)| \pto 0.
\]
By continuity of $\theta \mapsto \nabla^2 L(u,\theta)$, we obtain for each sequence $\tilde \theta_n \pto \theta_0(u)$ that
\[
	|\nabla^2 L_{n,b}(u,\tilde \theta_n) -V(u)| \le |\nabla^2 L_{n,b}(u,\tilde \theta_n) - \nabla^2 L(u,\tilde \theta_n)|  + |\nabla^2 L(u,\tilde \theta_n) - \nabla^2 L(u,\theta_0(u))| \pto 0.
\]
Standard arguments now give the result.
\end{proof}

The results of Theorem \ref{DP_mle_clt}(ii) show that under the existence of derivative processes, one can choose the MSE-optimal rate $b \sim n^{-1/3}$ for the bandwidth, keeping $\hat \theta_b(u)$ still asymptotically unbiased. This result can be used in several applications, for instance for bootstrapping $X_{t,n}$ via the recursion \reff{DP_rek_morespecific} with estimated errors $\hat \varepsilon_t = H(X_{t,n},Y_{t-1,n},\hat \theta_b(t/n))$, $t = p+1,...,n$.

An important special case is the case of Gaussian conditional likelihoods combined with nonlinear autoregressive models. Specific examples for these are given in Example \ref{DP_example_model}.

\medskip
\begin{example}[Nonlinear autoregressive models]\label{DP_mle_standard_example}
	In this example we discuss the model $\tilde G_{\varepsilon}(y,\theta) = \mu(y,\theta) + \sigma(y,\theta)\varepsilon$, where $\mu, \sigma:\IR^p \times \Theta \to \IR$ satisfy
	\begin{equation}
		\sup_{\theta}\sup_{y\not=y'}\frac{|\mu(y,\theta) - \mu(y',\theta)|}{|y-y'|_{\chi,1}} + \sup_{\theta}\sup_{y\not=y'}\frac{|\sigma(y,\theta) - \sigma(y',\theta)|}{|y-y'|_{\chi,1}}\|\varepsilon_0\|_2 \le 1\label{DP_mle_standard_example_eq1}
	\end{equation}
	with some $\chi \in \IR^{p}_{\ge 0}$ with $|\chi|_1 < 1$. Assume that $\IE \varepsilon_0 = 0$ and $\IE \varepsilon_0^2 = 1$ and that  $\theta_0$ is Hoelder-continuous with exponent $\alpha$. Then Assumption \ref{DP_ass1}\ref{DP_ass1_l2} is fulfilled with $q = 2$.
	
	If we choose $f_{\varepsilon}$ to be the standard Gaussian density, we obtain from \reff{DP_likelihood_standard_choice}:
	\begin{equation}
		\l(x,y,\theta) = \frac{1}{2}\Big(\frac{x - \mu(y,\theta)\big)}{\sigma(y,\theta)}\Big)^2 - \frac{1}{2}\log \sigma^2(y,\theta) + \mbox{const.}\label{DP_mle_standard_example_eq4}
	\end{equation}
	Furthermore assume that
	\begin{equation}
		\sup_{y}\sup_{\theta\not=\theta'}\frac{|\mu(y,\theta) - \mu(y,\theta')|}{|\theta-\theta'|_1\cdot(1+|y|_1)} < \infty, \quad\quad \sup_{y}\sup_{\theta\not=\theta'}\frac{|\sigma(y,\theta) - \sigma(y,\theta')|}{|\theta-\theta'|_1\cdot(1+|y|_1)} < \infty.\label{DP_mle_standard_example_eq2}
	\end{equation}
	 Let $\sigma(\cdot) \ge \delta_{\sigma}$ be uniformly bounded from below with some $\delta_{\sigma} > 0$. Then $\l \in \tilde\sL_{p+1}(1,C)$ with some $C > 0$, and Assumption \ref{DP_ass1}\ref{DP_ass1_l1},\ref{DP_ass1_l4} is fulfilled with $q = 2$ and $\alpha$ from above.
	
	Fix $u\in[0,1]$. Suppose that
	\[
		\mu(\tilde Y_{t-1}(u),\theta) = \mu(\tilde Y_{t-1}(u),\theta_0(u)) \quad\mbox{ and }\quad \sigma(\tilde Y_{t-1}(u),\theta) = \sigma(\tilde Y_{t-1}(u),\theta_0(u)) \quad a.s.
	\]
	implies $\theta = \theta_0(u)$. Then $\theta \mapsto L(u,\theta)$ has a unique minimum in $\theta = \theta_0(u)$ since $\log(x) \le x-1$ if and only if $x = 1$ and $x^2 \ge 0$ if and only if $x = 0$ and, omitting the argument $\tilde Y_{t-1}(u)$,
	\begin{eqnarray*}
		2\big(L(u,\theta) - L(u,\theta_0(u))\big) = \IE \Big(\frac{\mu(\theta)-\mu(\theta_0(u))}{\sigma(\theta)}\Big)^2 + \IE\Big[\log\frac{\sigma(\theta)^2}{\sigma(\theta_0(u))^2} - 1 + \frac{\sigma(\theta_0(u))^2}{\sigma(\theta)^2}\Big] \ge 0.
	\end{eqnarray*}
	If additionally $\Theta$ is compact and $\theta_0(u) \in \mbox{int}(\Theta)$, the assumptions of Theorem \ref{DP_mle_consistency} are fulfilled and we obtain for $\hat \theta_b$ defined by \reff{DP_mle_def}:
	 \[
	 	\hat \theta_b(u) \pto \theta_0(u).
	 \]
	 We now will show asymptotic normality of $\hat \theta_b$. To keep the presentation simple, we will assume $\sigma(\cdot,\cdot) \equiv 1$, $\IE \varepsilon_0^4 < \infty$ and replace $\IE \varepsilon_0^2 = 1$ by $\IE \varepsilon_0^2 = \sigma_0^2 > 0$. Note that Assumption \ref{DP_ass1}\ref{DP_ass1_l2} is fulfilled with $q = 4$. Then, omitting the arguments $(y,\theta)$ of $\mu$, we have
	 \[
	 	\nabla \l(x,y,\theta) = -(x-\mu) \nabla \mu, \quad\quad \nabla^2 \l(x,y,\theta) = \nabla \mu\cdot \nabla \mu' - (x-\mu)\nabla^2 \mu.
	 \]
	 This shows $\IE[\nabla \l(\tilde X_t(u),\tilde Y_{t-1}(u),\theta_0(u))|\sF_{t-1}] = 0$ and $I(u) = \IE[\nabla \l \cdot \nabla \l'] = \sigma_0^2 \IE[\nabla \mu\cdot \nabla \mu'] = \sigma_0^2 V(u)$ with $V(u) := \nabla^2 L(u,\theta_0(u))$. If additionally
	 \begin{equation}
	 	\sup_{\theta}\sup_{y\not=y'}\frac{|\nabla \mu(y,\theta) - \nabla \mu(y',\theta)|_1}{|y-y'|_1} < \infty, \quad\quad \sup_{y}\sup_{\theta\not=\theta'}\frac{|\nabla \mu(y,\theta) - \nabla \mu(y,\theta')|_1}{|\theta-\theta'|_1(1+|y|_1)} < \infty\label{DP_mle_standard_example_eq3}
	 \end{equation}
	 and similar assumptions are fulfilled for $\nabla^2 \mu$, then we have $\nabla \l, \nabla^2 \l \in \tilde \sL_{p+1}(1,C')$ with some $C' > 0$. This shows that all conditions of the first part of Theorem \ref{DP_mle_clt} are fulfilled and we obtain for $b\to 0$, $nb\to\infty$ and $nb^3 = o(1)$:
	 \begin{equation}
	 	\sqrt{nb}\big(\hat \theta_b(u) - \theta_0(u)\big)  \dto N\big(0, \sigma_0^2 \cdot V(u)^{-1}\big).\label{DP_mle_standard_example_eq5}
	 \end{equation}
	 If additionally, $\mu, \nabla \mu$ and $\theta_0$ are continuously differentiable and
	 \begin{equation}
	 	\sup_{\theta}\sup_{y\not=y'}\frac{|\partial_i \mu(y,\theta) - \partial_i \mu(y',\theta)|_1}{|y-y'|_{1}} < \infty, \quad (i = 1,2),\label{DP_mle_standard_example_eq6}
	 \end{equation}
	 then $\nabla \l$ is continuously differentiable and Assumption \ref{DP_ass1}\ref{DP_ass1_l3} is fulfilled with $q = 2$. If $K$ is symmetric, all conditions of the second  part of Theorem \ref{DP_mle_clt} are fulfilled and we obtain \reff{DP_mle_standard_example_eq5} even if $nb^3 = O(1)$.
\end{example}

We close this section by using the results of Example \ref{DP_mle_standard_example} in a more specific example of the tvExpAR(1) process which is a locally stationary version of the ExpAR(1) process discussed in \cite{jones1978}. Up to now, there is no asymptotic theory available for parameter estimators in this model; we show that our theory immediately provides consistency and asymptotic normality of the corresponding maximum likelihood estimator.

\begin{example}[Maximum likelihood estimation in the tvExpAR(1) process] Assume that there exists $\theta_0:[0,1] \to \Theta$ (where the image of $\theta_0$ is in the interior of $\Theta$) with $\Theta := \{\theta \in \IR: 0 \le \theta \le \rho\}$ and some fixed $\rho > 0$, $0 < |a_0| < 1$ such that
\[
	X_{t,n} = a_0 \exp\Big(-\theta_0\big(\frac{t}{n}\big) X_{t-1,n}^2\Big) X_{t-1,n} + \varepsilon_t, \quad t = 1,...,n.
\]
Assume that $\IE \varepsilon_0 = 1$, $\IE \varepsilon_0^2 = \sigma_0^2 > 0$ and $\IE \varepsilon_0^4 < \infty$. It is easily seen that this model fulfills the smoothness assumptions \reff{DP_mle_standard_example_eq1}, \reff{DP_mle_standard_example_eq2},  \reff{DP_mle_standard_example_eq3} and \reff{DP_mle_standard_example_eq6} with $\mu(y,\theta) := a_0 \exp(-\theta y^2) y$ and $\sigma(\cdot,\cdot) \equiv 1$. Let $\tilde X_t(u)$ denote the corresponding stationary approximation of $X_{t,n}$. Identifiability of $\theta$ is obtained due to
\[
	\IE[(\mu(\tilde X_t(u),\theta) - \mu(\tilde X_t(u),\theta'))^2] \ge a_0^2 \IE[\exp(-2 \rho \tilde X_0(u)^2) \tilde X_0(u)^6] \cdot |\theta - \theta'|^2,
\]
since $\IE[\exp(-2\rho \tilde X_t(u)^2) \tilde X_t(u)^6] = 0$ would imply $\tilde X_t(u) = 0$ a.s. which is a contradiction to $\IE[\tilde X_t(u)^2] \ge \sigma_0^2$ which follows from the recursion of $\tilde X_t(u)$. Let $\hat \theta_b(u)$ be defined by \reff{DP_mle_def} based on the likelihood \reff{DP_mle_standard_example_eq4} and let Assumption \ref{DP_ass_kernel} hold. We obtain for $b \to 0$, $bn \to \infty$:
\[
	\hat \theta_b(u) \pto \theta_0(u),
\]
and for $nb^3 = O(1)$:
\[
	\sqrt{nb}\big(\hat \theta_b(u) - \theta_0(u)\big) \dto N(0, \sigma_0^2 V(u)^{-1}),
\]
where $V(u) = a_0^2 \IE[\exp(-2 \theta_0(u) \tilde X_0(u)^2)\tilde X_0(u)^6]$.

\end{example}

\section{Concluding Remarks}
\label{DP_section5}
	In this paper, we have made some steps towards a general asymptotic theory for nonlinear locally stationary processes. A key role in our derivations is played by the local stationary approximation, the derivative process, the corresponding Taylor-expansion and the resulting differential calculus.

Just based on this local approximation we were able to prove laws of large numbers, a central limit theorem, and stochastic and deterministic bias approximations - results which have not been proved so far for general locally stationary processes. For example for the global strong law of large numbers we need only the existence of the first order moment of the process. It should be noted that for these results we concluded from local assumptions to global results such as the strong law of large numbers and the central limit theorem. A simulation displayed in Figure \ref{DP_figure1_simu} shows that the pointwise approximation of $X_{t,n}$ by $\tilde X_t(u)$ and $\partial_u \tilde X_t(u)$ works quite well.

We also showed that these results can be applied to a general nonlinear time series model with a nonstationary Markov structure which includes several nonlinear models. As another application we derived the asymptotic properties of the maximum likelihood estimator for such processes. The result is proved by applying the differential calculus of the derivative process.

\section*{Acknowledgements}

We are very grateful to the associate editor and a referee whose comments lead to a considerable improvement of the paper. In particular the observation that in many cases differentiability in $L^q$ is sufficient (see Remark \ref{remark_differentiability}) was pointed out by the referee.
We also gratefully acknowledge support by Deutsche Forschungsgemeinschaft
through the Research Training Group RTG 1653.

\newpage

\section{Supplement A}
\label{suppA}
This supplement contains the remaining proofs for Section \ref{DP_section35}, \ref{DP_section2} and \ref{DP_section4}.


\subsection{Proofs of Section \ref{DP_section35}}
Let us first cite a Lemma from \cite{suhasini2006} (Lemma A.1 and A.2) which can be easily generalized to convergence in $L^1$:
\begin{lemma}\label{DP_lemma_ergodic} Assume that $(Y_t)$ is a stationary and ergodic process with $\IE|Y_1| < \infty$. Assume that $u\in(0,1)$. Let $b = b_n \to 0$ such that $nb_n \to \infty$. Then the following convergence holds in $L^1$:
\[
	\frac{1}{nb}\sum_{t=1}^{n}K\Big(\frac{t/n-u}{b}\Big)Y_t \to \IE Y_1.
\]	
\end{lemma}
\begin{proof}[\normalfont \textbf{Proof of Theorem \ref{DP_prop_ergodic}}]
	(i) Without loss of generality, let us assume $\alpha \le 1$. For $J\in\IN$ and $j = 1,...,2^{J}$ define intervals of indices $I_{j,J,n} := \{t: t/n\in (\frac{j-1}{2^J},\frac{j}{2^J}]\}$ such that $\bigcup_{j=1}^{2^J}I_{j,J,n} = \{1,...,n\}$. For fixed $J \in \IN$, we have
	\begin{eqnarray*}
		&& \Big\|\frac{1}{n}\sum_{t=1}^{n}X_{t,n} - \frac{1}{2^J}\sum_{j=1}^{2^J}\frac{1}{|I_{j,J,n}|}\sum_{t \in I_{j,J,n}}X_{t,n}\Big\|_1 \le \Big\|\sum_{j=1}^{2^J}\Big(\frac{|I_{j,J,n}|}{n} - \frac{1}{2^J}\Big)\cdot \frac{1}{|I_{j,J,n}|}\sum_{t \in I_{j,J,n}}X_{t,n}\Big\|_1\\
		&\le& \sum_{j=1}^{2^J}\Big|\frac{|I_{j,J,n}|}{n} - \frac{1}{2^J}\Big| \cdot \sup_{t=1,...,n}\|X_{t,n}\|_1 \le  \frac{2^J}{n}\cdot \sup_{t=1,...,n}\|X_{t,n}\|_1
	\end{eqnarray*}
	and
	\begin{eqnarray*}
		&& \Big\|\frac{1}{2^J}\sum_{j=1}^{2^J}\frac{1}{|I_{j,J,n}|}\sum_{t \in I_{j,J,n}}X_{t,n} - \frac{1}{2^J}\sum_{j=1}^{2^J}\frac{1}{|I_{j,J,n}|}\sum_{t \in I_{j,J,n}}\tilde X_{t}\Big(\frac{j}{2^J}\Big)\Big\|_1\\
		&\le& \sup_{t=1,...,n}\|X_{t,n} - \tilde X_t(t/n)\|_{1} + \sup_{|u-v|\le 2^{-J}}\big\| \tilde X_t(u) - \tilde X_t(v)\|_1
	\end{eqnarray*}
	Note that for fixed $J$, by the ergodic theorem for stationary sequences we have for $n\to\infty$:
	\[
		E(J,n) := \frac{1}{2^J}\sum_{j=1}^{2^J}\frac{1}{|I_{j,J,n}|}\sum_{t\in I_{j,J,n}}\tilde X_{t}\Big(\frac{j}{2^J}\Big) \to \frac{1}{2^J}\sum_{j=1}^{2^J}\IE \tilde X_0\Big(\frac{j}{2^J}\Big) =: E(J) \quad \mbox{a.s. and in }L^1,
	\]
	By the continuity of $[0,1]\to \IR, u \mapsto \IE \tilde W_0(u)$, we have
	\[
		E(J) = \frac{1}{2^J}\sum_{j=1}^{2^J}\IE \tilde X_0\Big(\frac{j}{2^J}\Big) \to \int_{0}^{1}\IE \tilde X_0(u) \dif u =: E \quad (J\to\infty).
	\]
	Finally,
	\begin{eqnarray*}
		&& \Big\| \frac{1}{n}\sum_{t=1}^{n}X_{t,n} - E\Big\|_1\\
		&\le& \frac{2^J}{n}\cdot \sup_{t=1,...,n}\|X_{t,n}\|_1 + \sup_{t=1,...,n}\|X_{t,n} - \tilde X_t(t/n)\|_1 + \sup_{|u-v| \le 2^{-J}}\|\tilde X_t(u) - \tilde X_t(v)\|_1\\
		&&\quad\quad + \|E(J,n) - E(J)\|_1 + |E(J) - E|.
	\end{eqnarray*}
	Thus for all $J\in\IN$:
	\[
		\limsup_{n\to\infty}\Big\| \frac{1}{n}\sum_{t=1}^{n}X_{t,n} - E\Big\|_1 \le \sup_{|u-v| \le 2^{-J}}\|\tilde X_t(u) - \tilde X_t(v)\|_1 + |E(J) - E|.
	\]
	The limit $J\to\infty$ gives the result.
	
	(ii) To prove the local weak law of large numbers, first note that
	\begin{eqnarray*}
		&& \Big\|\frac{1}{nb}\sum_{t=1}^{n}K\Big(\frac{t/n-u}{b}\Big)\cdot \big(X_{t,n} - \tilde X_t(u)\big)\Big\|_1\\
		&\le& |K|_{\infty}\Big(\sup_{t=1,...,n}\|X_{t,n} - \tilde X_t(t/n)\|_1 + \sup_{|u-v|\le b/2}\|\tilde X_t(u) - \tilde X_t(v)\|_1\Big) \to 0.
	\end{eqnarray*}
	This shows that it is enough to consider the convergence of the sum with the corresponding stationary sequence. From Lemma \ref{DP_lemma_ergodic}, we have that $\frac{1}{nb}\sum_{t=1}^{n}K\big(\frac{t/n-u}{b}\big)\cdot \tilde X_t(u) \to \IE \tilde X_t(u)$ holds in $L^1$, which finishes the proof.
	
	(iii) Define $S_{n}(u) := \sum_{t=1}^{n}K\big(\frac{t/n-u}{b}\big)\big(X_{t,n} - \IE X_{t,n}\big)$ and $\tilde S_{n}(u) := \sum_{t=1}^{n}K\big(\frac{t/n-u}{b}\big)\big(\tilde X_{t}(t/n)- \IE \tilde X_{t}(t/n)\big)$ and $\tilde S_{k,n} := \sum_{t=1}^{k}\tilde X_{t}(t/n)$. By partial summation, we have
	\[
		\tilde S_{n}(u) = \sum_{t=1}^{n-1}\Big[K\Big(\frac{t/n-u}{b}\Big) - K\Big(\frac{(t+1)/n-u}{b}\Big)\Big]\cdot \tilde S_{t,n} + K\Big(\frac{1-u}{b}\Big) \tilde S_{n,n}.
	\]
	Since $K$ is of bounded variation $B_K$, we have $\sum_{t=1}^{n-1}\big|K\big(\frac{t/n-u}{b}\big) - K\big(\frac{(t+1)/n-u}{b}\big)\big| \le B_K$ and thus
	\begin{equation}
		|\tilde S_{n}(u)| \le B_K \cdot \sup_{t=1,...,n}|\tilde S_{t,n}|.\label{DP_weaklaw_basic_inequality}
	\end{equation}
	The same calculation yields
	\[
		|S_n(u) - \tilde S_n(u)| \le B_K \cdot \sum_{k=1}^{n}\big(|X_{t,n} - \tilde X_t(t/n)| + |\IE X_{t,n} - \IE \tilde X_t(t/n)|\big).
	\]
	First assume $1 < q \le 2$. By using the decomposition $\tilde X_{t}(t/n)- \IE \tilde X_{t}(t/n) = \sum_{l=0}^{\infty}P_{t-l}\tilde X_{t}(t/n)$ and applying Doob's $L^q$ maximal inequality (cf. Theorem 5.4.3 in \cite{durret}), Burkholder's inequality (cf. \cite{burkholder1988}) and the elementary inequality $(|a_1| + |a_2|)^{q/2} \le |a_1|^{q/2} + |a_2|^{q/2}$, we obtain
	\begin{eqnarray*}
		\big\|\sup_{t=1,...,n}|\tilde S_{t,n}|\big\|_q &\le& \sum_{l=0}^{\infty}\big\|\sup_{t=1,...,n}\big|\sum_{s=1}^{t}P_{s-l}\tilde X_{s}(s/n)\big|\big\|_q\\
		&\le& \sum_{l=0}^{\infty}\frac{q}{q-1}\big\|\sum_{s=1}^{n}P_{s-l}\tilde X_{s}(s/n)\big\|_q \le \sum_{l=0}^{\infty}\frac{q}{(q-1)^2}\Big(\IE\Big(\sum_{s=1}^{n}(P_{s-l}\tilde X_{s}(s/n))^2\Big)^{q/2}\Big)^{1/q}\\
		&\le& \frac{q}{(q-1)^2}\sum_{l=0}^{\infty}\Big(\sum_{s=1}^{n}\|P_{s-l}\tilde X_{s}(s/n)\|_q^q\Big)^{1/q}\\
		&\le& \frac{q}{(q-1)^2}\cdot n^{1/q}\cdot \sum_{l=0}^{\infty}\sup_{u \in [0,1]}\delta_q^{\tilde X(u)}(l).
	\end{eqnarray*}
	which shows that
	\begin{eqnarray*}
		\big\|\sup_{u\in [0,1]}|(nb)^{-1} S_{n}(u)|\big\|_q &\le& \big\|\sup_{u\in [0,1]}|(nb)^{-1} \tilde S_{n}(u)|\big\|_q + \big\|\sup_{u\in [0,1]}|(nb)^{-1} (S_n(u) - \tilde S_{n}(u))|\big\|_q\\
		&\le& \frac{B_K q}{(q-1)^2}\Delta_{0,q}^{\tilde X}\cdot n^{1/q-1}b^{-1} + 2B_K  C_B \cdot n^{-\alpha} b^{-1}.
	\end{eqnarray*}
	If $q > 2$, we use a Nagaev-type inequality from \cite{wu2013}, Theorem 2(ii) which also holds in our situation as the authors point out in their Section 4. Applying this theorem  to $\tilde S_{t,n}$ and $-\tilde S_{t,n}$, we have for all $x > 0$:
	\[
		\IP\big(\sup_{t=1,...,n}|\tilde S_{t,n}| > x/2\big) \le \frac{2C_1 (\Delta_{0,q}^{\tilde X})^q n}{(x/2)^q} + 8G_{1-2/q}\Big(\frac{C_2 x}{2\sqrt{n}\Delta_{0,q}^{\tilde X}}\Big)
	\]
	with positive constants $C_1,C_2$ not depending on $n$. Using \reff{DP_weaklaw_basic_inequality}, we obtain
	\begin{eqnarray*}
		\IP\Big(\sup_{u\in[0,1]}|(nb)^{-1}S_n(u)| > x\Big) &\le& \IP\Big(\sup_{u\in[0,1]}|(nb)^{-1}\tilde S_n(u)| > x/2\Big)\\
		&&\quad\quad\quad\quad\quad + \IP\Big(\sup_{u\in[0,1]}|(nb)^{-1}(S_n(u) - \tilde S_n(u))| > x/2\Big)\\
		&\le& \IP\Big(\sup_{t=1,...,n}|\tilde S_{t,n}| > \frac{nbx}{2B_K}\Big) + \IP\Big(\sup_{u\in[0,1]}|S_n(u) - \tilde S_n(u)| > nbx/2\Big)\\
		&\le&  \frac{2C_1 (B_K \Delta_{0,q}^{\tilde X})^q n (nb)^{-q}}{(x/2)^q} + 8G_{1-2/q}\Big(\frac{C_2 nb x}{2\sqrt{n}B_K \Delta_{0,q}^{\tilde X}}\Big)\\
		&&\quad\quad\quad\quad\quad + \frac{(2 B_K C_B)^q}{(x/2)^q}\cdot (n^{-\alpha}b^{-1})^q
	\end{eqnarray*}
	In case that knowledge of the dependence measure $\delta_q^{X_{\cdot,n}}$ of the locally stationary process is available, the approximation of $X_{t,n}$ by $\tilde X_t(t/n)$ is not necessary and therefore the discussion of the terms $|S_n(u) - \tilde S_n(u)|$ can be omitted.
\end{proof}

\begin{proof}[\normalfont \textbf{Proof of Proposition \ref{DP_prop_clt}}] For the proof, we use Theorem 5.46 in \cite{witting1995}. Put $S_{k,n} = \sum_{t=1}^{k} X_{t,n}$ and $\tilde S_{k,n} := \sum_{t=1}^{k}\tilde X_t(t/n)$. Note that by $\alpha > \frac{1}{2}$,
	\[
		\big\|\sup_{u\in[0,1]}\big|S_{\lfloor n u \rfloor,n}/\sqrt{n} - \tilde S_{\lfloor n u \rfloor,n}/\sqrt{n}\big| \big\|_2 \le n^{-1/2}\sum_{t=1}^{n}\big\| X_{t,n} - \tilde X_t(t/n)\big\|_2 \le C_B n^{1/2}\cdot n^{-\alpha} \to 0
	\]
	Put $\tilde S_{k,n,L} := \sum_{l=0}^{L-1}\sum_{t=1}^{k}P_{t-l}\tilde X_t\big(\frac{t}{n}\big)$. Use the abbreviation $\lima$ for $\limsup_{L\to\infty}\limsup_{n\to\infty}$. Because $P_{t-l} \tilde X_{t}(t/n) - \IE \tilde X_{t}(t/n) \to 0$ a.s. and in $L^1$ for $l\to\infty$, we have by Doob's maximal inequality:
	\begin{eqnarray*}
		&& \lima\Big\|\sup_{u\in[0,1]}|\tilde S_{\lfloor nu\rfloor,n}/\sqrt{n} - \tilde S_{\lfloor nu\rfloor,n,L}/\sqrt{n}|\Big\|_2\\
		&\le& \lima \sum_{l=L}^{\infty}\frac{1}{\sqrt{n}}\Big\| \sup_{T=1,...,n}\Big|\sum_{t=1}^{T}P_{t-l}\tilde X_{t}(t/n)\Big|\Big\|_2 \le  \lima \sum_{l=L}^{\infty}\frac{2}{\sqrt{n}}\Big\| \sum_{t=1}^{n}P_{t-l}\tilde X_{t}(t/n)\Big\|_2\\
		&\le& \lima \sum_{l=L}^{\infty}\frac{2}{\sqrt{n}}\Big(\sum_{t=1}^{n}\|P_{t-l}\tilde X_{t}(t/n)\|_2^2\Big)^{1/2} \le \lima \ 2\sum_{l=L}^{\infty}\delta^{\tilde X}_2(l) = 0.
	\end{eqnarray*}
	This shows that $S_{\lfloor n u \rfloor,n}/\sqrt{n}$ can be approximated by $\tilde S_{\lfloor n u \rfloor,n,L}/\sqrt{n}$. In the case that the dependence measure $\delta_2^{X_{\cdot,n}}(l)$ of $X_{t,n}$ is defined and summable in $l$, we can omit the condition $\alpha > \frac{1}{2}$ by using the following argument. Define $S_{k,n,L} := \sum_{l=0}^{L-1}\sum_{t=1}^{k}P_{t-l}X_{t,n}$. Similarly as above, it can be shown that
	\[
		\lima\Big\|\sup_{u\in[0,1]}| S_{\lfloor nu\rfloor,n}/\sqrt{n} -  S_{\lfloor nu\rfloor,n,L}/\sqrt{n}|\Big\|_2 = 0.
	\]
	Furthermore, it holds that
	\[
		\|P_{t-l}(X_{t,n} - \tilde X_t(t/n))\|_2 \le \min\Big\{\delta^{\tilde X(t/n)}_2(l) + \delta^{X_{\cdot,n}}_2(l),\sup_{t=1,...,n}\|X_{t,n} - \tilde X_t(t/n)\|_2\Big\} =: \min\{\delta_n(l), c_n\}.
	\]
	By similar arguments as in the calculation above, we obtain
	\begin{eqnarray*}
		&& \lima \Big\|\sup_{u\in[0,1]}|S_{\lfloor nu\rfloor,n,L}/\sqrt{n} - \tilde S_{\lfloor nu\rfloor,n,L}/\sqrt{n}|\Big\|_2\\
		&\le& \lima\  2\sum_{l=0}^{L-1}\min\{\delta_n(l), c_n\} \le \lima \Big( \sum_{0 \le l \le c_n^{-1/2}}c_n + \sum_{l > c_n^{-1/2}}\delta_n(l)\Big)\\
		&\le& \lima\Big(c_n^{1/2} + \sum_{l > c_n^{-1/2}}\sup_{n\in\IN}\delta_n(l)\Big) = 0,
	\end{eqnarray*}
	which in turn also shows that $S_{\lfloor n u \rfloor,n}/\sqrt{n}$ can be approximated by $\tilde S_{\lfloor n u \rfloor,n,L}/\sqrt{n}$.
	
	Now fix $L \in \IN$. Define the index-shifted variant of $\tilde S_{k,n,L}$ by $\hat S_{k,n,L} := \sum_{t=1}^{k} \Big(\sum_{l=0}^{L-1}P_{t}\tilde X_{t+l}\big(\frac{t+l}{n}\big)\Big)$, where $\tilde X_t(u) := \tilde X_t(1)$ for $u > 1$. For $T = 1,...,n$, we have
	\begin{eqnarray*}
		&& |\tilde S_{T,n,L} - \hat S_{T,n,L}| \le \sum_{l=0}^{L-1}\sum_{t=1}^{l}\big|P_{t-l}\tilde X_t\big(\frac{t}{n}\big)\big| + \sum_{l=0}^{L-1}\sum_{t=T-l+1}^{T}\big|P_{t}\tilde X_{t+l}\big(\frac{t+l}{n}\big)\big|.\\
	\end{eqnarray*}
	Finally, define the martingale differences $M_{t,l} := P_{t}\tilde X_{t+l}\big(\frac{t+l}{n}\big)$ and the stationary martingale differences $M_l(u) := P_{0}\tilde X_l\big(u)$. Note that $M_{t,l}$ has the same distribution as $M_l(\frac{t+l}{n})$. We have
	\begin{eqnarray}
		\IP\Big(\frac{1}{\sqrt{n}}\sup_{T=1,...,n}|M_{t,l}| \ge \varepsilon\Big) &\le& n \cdot \sup_{t=1,...,n}\IP(|M_{t,l}| \ge \varepsilon \sqrt{n}) \le \sup_{t=1,...,n}\IE[|M_{t,l}|^2 \Ii_{\{|M_{t,l}| \ge \varepsilon \sqrt{n}\}}]\nonumber\\
		&=& \sup_{u\in[0,1]}\IE[M_l(u)^2 \Ii_{\{|M_l(u)| \ge \varepsilon \sqrt{n}\}}]\nonumber\\
		&\le& \IE\Big[\big(\sup_u |M_0(u)|\big)^2 \cdot \Ii_{\{\sup_u |M_0(u)| \ge \varepsilon \sqrt{n}\}}\Big] \to 0,\label{DQ_proof_clt_eq5}
	\end{eqnarray}
	which shows $\frac{1}{\sqrt{n}}\sup_{u\in[0,1]}|\tilde S_{\lfloor nu\rfloor,n,L} - \hat S_{\lfloor nu\rfloor,n,L}| \pto 0$.\\
	We now investigate the weak convergence of $\hat S_{\lfloor nu\rfloor,L}/\sqrt{n}$ with a martingale central limit theorem from \cite{billingsley}, Theorem 18.2. Note that $\sum_{l=0}^{L-1}M_{t,l}/\sqrt{n}$ is a martingale difference sequence with respect to $\sF_{t}$. By elementary inequalities it can be seen that for each $T = 1,...,n$ and each $\varepsilon > 0$,
	\[
		\sum_{t=1}^{T}\IE\Big[\Big(\sum_{l=0}^{L-1}M_{t,l} / \sqrt{n}\Big)^2\Ii_{\{|\sum_{l=0}^{L-1}M_{t,l}| \ge \varepsilon \sqrt{n}\}}\Big]
	\]
	is bounded by finitely many (dependent on $L$) terms of the form
	\[
		\frac{1}{n}\sum_{t=1}^{T}\IE[M_{t,l}^2 \Ii_{\{|M_{t,l'}| \ge \varepsilon\sqrt{n}\}}],
	\]
	where $l,l'\in \{0,...,L-1\}$. By using similar techniques as in \reff{DQ_proof_clt_eq5}, it can be shown that these converge to 0.
	
	It remains to investigate the behavior of
	\[
		\sum_{t=1}^{T}\IE\Big[\Big(\sum_{l=0}^{L-1}M_{t,l}/\sqrt{n}\Big)^2\Big|\sF_{t-1}\Big] = \sum_{l,l'=0}^{L-1}\frac{1}{n}\sum_{t=1}^{T}\IE[M_{t,l}M_{t,l'}|\sF_{t-1}]
	\]
	for $T = \lfloor sn\rfloor$, $s \in (0,1]$ and $l,l' \in \{0,...,L-1\}$. Define $I_{k,K,T} := \{t:\frac{t}{T}\in (\frac{k-1}{2^K},\frac{k}{2^K}]\}$, then  we have for $K\in\IN$:
	\begin{eqnarray*}
		&& \Big\|\frac{1}{T}\sum_{t=1}^{T}\IE[M_{t,l}M_{t,l'}|\sF_{t-1}] - \frac{1}{2^K}\sum_{k=1}^{2^K}\frac{1}{|I_{k,K,T}|}\sum_{t \in I_{k,K,T}}\IE[M_{t,l}M_{t,l'}|\sF_{t-1}]\Big\|_1\\
		&\le& \frac{2^K}{T}\cdot \sup_{t=1,...,n}\sup_{l=0,...,L-1}\|M_{t,l} M_{t,l'}\|_1,
	\end{eqnarray*}
	which is bounded by $\frac{2^K}{T}\sup_{u}\|\tilde X_0(u)\|_2^2$. Furthermore, since $\frac{t}{T} \in I_{k,K,T}$ implies $|\frac{t+l}{n} - \frac{k}{2^K}s| \le 2^{-K} + \frac{L}{n}$, we obtain
	\begin{eqnarray*}
		&& \Big\| \frac{1}{2^K}\sum_{k=1}^{2^K}\frac{1}{|I_{k,K,T}|}\sum_{t \in I_{k,K,n}}\Big(\IE[M_{t,l}M_{t,l'}|\sF_{t-1}] - \IE[M_{t,l}(\frac{k}{2^K}s)M_{t,l'}(\frac{k}{2^K}s)|\sF_{t-1}]\Big)\Big\|_1\\
		&\le& 2\Big(\sup_{|u-v| \le 2^{-K}}\|\tilde X_0(u) - \tilde X_0(v)\|_2 + \sup_{|u-v| \le L n^{-1}}\|\tilde X_0(u) - \tilde X_0(v)\|_2\Big)\cdot \sup_u \| \tilde X_0(u)\|_2,
	\end{eqnarray*}
	where $M_{t,l}(u) := P_{t}\tilde X_{t+l}(u)$. Since $\IE[M_{t,l}(u)M_{t,l'}(u)|\sF_{t-1}]$ is ergodic, we have
	\[
		\frac{1}{|I_{k,K,T|}}\sum_{t \in I_{k,K,T}}\IE[M_{t,l}(\frac{k}{2^K}s)M_{t,l'}(\frac{k}{2^K}s)|\sF_{t-1}] \pto \IE[M_{0,l}(\frac{k}{2^K}s)M_{0,l'}(\frac{k}{2^K}s)].
	\]
	In total, performing first $n\to\infty$ and afterwards $K\to\infty$, we obtain
	\[
		\sum_{l,l'=0}^{L-1}\frac{1}{n}\sum_{t=1}^{\lfloor ns\rfloor}\IE[M_{t,l}M_{t,l'} |\sF_{t-1}] \to \sum_{l,l'=0}^{L-1}s\cdot \int_{0}^{1}\IE[M_{0,l}(xs)M_{0,l'}(xs)]\dif x = \int_{0}^{s}\Big\|\sum_{l=0}^{L-1}P_{0}\tilde X_l(y)\Big\|_2^2 \dif y.
	\]
	So we have seen that $\{S_{\lfloor nu\rfloor}/\sqrt{n},0 \le u \le 1\} \dto \{\int_{0}^{u}\Big\|\sum_{l=0}^{L-1}P_{0}\tilde X_l(v)\Big\|_2 \dif B(v), 0 \le u \le 1\}$. By the dominated convergence theorem, $\int_{0}^{u}\Big\|\sum_{l=0}^{L-1}P_{0}\tilde X_l(v)\Big\|_2^2 \dif v \to \int_{0}^{u}\sigma^2(v) \dif v$ for $L\to\infty$, which completes the proof.
\end{proof}

\begin{proof}[\normalfont \textbf{Proof of Theorem \ref{DP_clt_local}}]
	Define $W_{n,b} := \frac{1}{\sqrt{nb}}\sum_{t=1}^{n}K\Big(\frac{t/n-u}{b}\Big)\cdot (X_{t,n} - \IE X_{t,n})$. Note that
	\begin{eqnarray*}
		&& \Big\|W_{n,b} - \frac{1}{\sqrt{nb}}\sum_{t=1}^{n}K\Big(\frac{t/n-u}{b}\Big)\cdot \big(\tilde X_t(t/n) - \IE \tilde X_t(t/n)\big)\Big\|_1\\
		&\le&  2|K|_{\infty}\sqrt{nb}\sup_{t=1,...,n}\|X_{t,n} - \tilde X_t(t/n)\|_1.
	\end{eqnarray*}
	Since $\|X_{t,n} - \tilde X_t(t/n)\|_1\le C_B n^{-\alpha}$ by assumption, the term above is of order $\sqrt{nb}n^{-\alpha}$. Since $\sum_{k=0}^{\infty}\sup_u\delta^{\tilde X(u)}_2(k) < \infty$, $|K|_{\infty} < \infty$ and $( K_b(t/n-u)P_{t-l}\tilde X_t(t/n))_t$ is a martingale difference sequence with respect to $(\sF_{t-l})$, we can use the same technique as in the proof of Theorem \ref{DP_prop_clt} to show that
	\[
		 \limsup_{L\to\infty}\limsup_{n\to\infty}\Big\|\frac{1}{\sqrt{nb}}\sum_{t=1}^{n}K\Big(\frac{t/n-u}{b}\Big)\cdot \Big[(\tilde X_t(t/n)-\IE \tilde X_t(t/n)) - \sum_{l=0}^{L-1}P_{t-l}\tilde X_{t}(t/n)\Big]\Big\|_2 = 0.
	\]
	Now fix $L \in\IN$. Since $K$ is Lipschitz continuous and $\sup_t\|\tilde X_t((t+l)/n) - \tilde X_t(t/n)\|_1 \le C_B l^{\alpha} n^{-\alpha}$, it is enough to consider the weak convergence of $\sum_{t=1}^{n}W_t(t/n)$, where we define $W_{t}(v) := \sum_{l=0}^{L-1}K\Big(\frac{t/n-u}{b}\Big)P_{t} \tilde X_{t+l}(v)/\sqrt{nb}$. Note that $W_t(t/n)$ is a martingale difference sequence w.r.t. $\sF_t$.
	It holds that
	\begin{eqnarray*}
		&& \sum_{t=1}^{n}\|W_t^2(t/n) - W_t^2(u)\|_1\\
		&\le& \sum_{l,l'=0}^{L-1}\frac{1}{nb}\sum_{t=1}^{n}K\Big(\frac{t/n-u}{b}\Big)^2 \|P_{t}\tilde X_{t+l}(t/n)P_{t}\tilde X_{t+l'}(t/n) - P_t \tilde X_{t+l}(u)P_t \tilde X_{t+l'}(u)\|_1\\
		&\le& 2\sum_{l,l'=0}^{L-1}\frac{1}{nb}\sum_{t=1}^{n}K\Big(\frac{t/n-u}{b}\Big)^2 \|\tilde X_0(t/n) - \tilde X_0(u)\|_2\cdot \sup_u \|\tilde X_0(u)\|_2\\
		&\le& 2L^2 |K|_{\infty}^2 C_B \sup_{u \in [0,1]}\|\tilde X_0(u)\|_2 \cdot b^{\alpha} = o(1).
	\end{eqnarray*}
	By Lemma \ref{DP_lemma_ergodic},
	\begin{eqnarray*}
		\sum_{t=1}^{n}\IE[W_{t}^2(u)|\sF_{t-1}] &=& \sum_{l,l'=0}^{L-1}\frac{1}{nb}\sum_{t=1}^{n}K\Big(\frac{t/n-u}{b}\Big)^2\IE[P_{t}\tilde X_{t+l}(u)P_{t} \tilde X_{t+l'}(u)|\sF_{t-1}]\\
		&\pto& \int K^2(x) \dif x \cdot \Big\|\sum_{l=0}^{L-1}P_{0}\tilde X_{l}(u)\Big\|_2^2.
	\end{eqnarray*}
	Fix $\varepsilon > 0$. The sum $\sum_{t=1}^{n}\IE[W_{t}^2(t/n)\Ii_{\{|W_{t}(t/n)| \ge \varepsilon\}}]$ is bounded by finitely many (dependent on $L$) terms of the form
	\begin{eqnarray*}
		&& \frac{1}{nb}\sum_{t=1}^{n}K\Big(\frac{t/n-u}{b}\Big)^2 \IE[(P_t \tilde X_{t+l}(t/n))^2\Ii_{\{|K|_{\infty}|P_t \tilde X_{t+l'}(t/n)| \ge \varepsilon \sqrt{nb}\}}]\\
		&\le& |K|^{2}_{\infty} \sup_{u\in[0,1]}\IE[(P_0 \tilde X_l(u))^2 \Ii_{\{|P_0 \tilde X_{l'}(u)| \ge \varepsilon\sqrt{nb}/|K|_{\infty}\}}]\\
		&\le& |K|^{2}_{\infty}\IE[(\sup_u |P_0 \tilde X_l(u)|)^2 \Ii_{\{\sup_u |P_0 \tilde X_{l'}(u)| \ge \varepsilon\sqrt{nb}/|K|_{\infty}\}}]
	\end{eqnarray*}
	which converges to 0 since $\|\sup_u |P_0 \tilde X_l(u)|\|_2 < \infty$ by assumption. So we can apply Theorem 18.1. from \cite{billingsley} to obtain
	\[
		\sum_{t=1}^{n}W_{t}(t/n) \dto N\Big(0, \int K^2(x) \dif x \cdot \Big\|\sum_{l=0}^{L-1}P_{0}\tilde X_{l}(u)\Big\|_2^2\Big)
	\]
	and thus by Theorem 5.46 in \cite{witting1995},
	\[
		W_{n,b} \dto N\big(0,\int K^2(x) \dif x \cdot \big\|\sum_{l=0}^{\infty}P_{0}\tilde X_{l}(u)\big\|_2^2\big).
	\]
\end{proof}

\subsection{Proofs of Section \ref{DP_section2}}
Here, we prove the results from Section \ref{DP_section2}.  The following lemma from \cite{duflo1997}, Lemma 6.2.10 therein  will be used frequently to verify the geometric decay of the difference of recursively defined processes:

\begin{lemma}\label{DP_standard_rec_argument} Assume that $p > 0$ is a positive natural number, $\chi \in \IR^p_{\ge 0}$ with $|\chi|_1 < 1$ and that there are sequences of real-valued nonnegative numbers $(z_s)_{s > -p}$, $(\mu_s)_{s > 0}$ which fulfill for all $s=1,2,...$:
\begin{equation}
	z_s \le \sum_{i=1}^{p}\chi_i z_{s-i} + \mu_s. \label{DP_standard_rec_argument_eq1}
\end{equation}
Then there exist constants $\lambda_0\in(0,1)$, $C_{\lambda} > 0$ only depending on $\chi,p$ such that for all $s = 1,2,...$:
\[
	z_s \le C_{\lambda}\Big( \lambda_0^{s} \cdot |(z_{0},...,z_{-p+1})|_1 + \sum_{i=0}^{s-1}\lambda_0^{i}\mu_{s-i}\Big).
\]
\end{lemma}
Sometimes we will apply the lemma for $s = 0,1,2,...$ instead of $s = 1,2,3,...$ .

For the following proofs, recall the abbreviations $Y_{t-1,n} = (X_{t-1,n},...,X_{t-p,n})$ and $\tilde Y_{t-1}(u) = (\tilde X_{t-1}(u),...,\tilde X_{t-p}(u))$. For $y \in \IR^p$, we will use the abbreviation $G_{\varepsilon,u}(y) := G_{\varepsilon}(y,u)$. Define the random map $R_{\varepsilon,u}(y) := (G_{\varepsilon,u}(y),y_1,...,y_{p-1})$. Let $X_{n,u}(y)$ be the first element of the vector $H_{n,u}(y) := R_{\varepsilon_0,u} \circ R_{\varepsilon_{-1},u} \circ ... \circ R_{\varepsilon_{-n},u}(y)$, where $n= 0,1,2,...$ . For consistency of the following argumentations, define $X_{n,u}(y) := y_{-n}$ for $n = -1,...,-p$. Note that $H_{n,u}(y)_j = X_{n-j+1,u}(y)$ (in distribution) for $j = 1,...,p$.
Let $J_{n,u}(y)$ be defined similarly to $H_{n,u}(y)$ but based on $\varepsilon_{-1},...,\varepsilon_{-n-1}$ instead of $\varepsilon_0,...,\varepsilon_{-n}$. Note that $X_{n,u}(y) = G_{\varepsilon_0,u}(J_{n-1,u}(y))$ and that $J_{n-1,u}(y) = H_{n-1,u}(y) = (X_{n-1,u}(y),...,X_{n-p,u}(y))'$ holds in distribution.

\begin{proof}[\normalfont \textbf{Proof of Proposition \ref{DP_rek_1_sol}}]
	(i) Note that $(|a| + |b|)^{q'} \le |a|^{q'} + |b|^{q'}$ since $0 < q' \le 1$. By \reff{DP_cond_1}, we obtain
	\begin{eqnarray*}
		&& \|X_{n,u}(y) - X_{n,u}(y')\|_q^{q'}\\
		&\le& \| G_{\varepsilon_0,u}(J_{n-1,u}(y)) - G_{\varepsilon_0,u}(J_{n-1,u}(y'))\|_q^{q'}\\
		&\le& \IE\big[\IE\big[|G_{\varepsilon_0,u}(J_{n-1,u}(y)) - G_{\varepsilon_0,u}(J_{n-1,u}(y'))|^{q} \big|\sF_{-1}\big]\big]^{q'/q}\\
		&\le& \IE\big[|J_{n-1,u}(y) - J_{n-1,u}(y')|_{\chi,q'}^{q}]^{q'/q}\\
		&\le& \IE\Big[\Big(\sum_{j=1}^{p}\chi_j |X_{n-j,u}(y) - X_{n-j,u}(y')|^{q'}\Big)^{q/q'}\Big]^{q'/q}\\
		&=& \big\|\sum_{j=1}^{p}\chi_j |X_{n-j,u}(y) - X_{n-j,u}(y')|^{q'}\big\|_{q/q'}\\
		&\le&  \sum_{j=1}^{p}\chi_j \big\| |X_{n-j,u}(y) - X_{n-j,u}(y')|^{q'}\big\|_{q/q'}
		= \sum_{j=1}^{p}\chi_j \big\|X_{n-j,u}(y) - X_{n-j,u}(y')\|_q^{q'}.
	\end{eqnarray*}
	By Lemma \ref{DP_standard_rec_argument}, we have with some $C_{\lambda} > 0, \lambda_0\in(0,1)$ independent of $u\in[0,1]$ that for all $n\in\IN$:
	\begin{equation}
		\|X_{n,u}(y) - X_{n,u}(y')\|_q^{q'} \le C_{\lambda} \lambda_0^{n+1}\cdot |y-y'|_1^{q'}.\label{DP_rek_1_zwischenresultat}
	\end{equation}
	Applying \reff{DP_rek_1_zwischenresultat} to $y = y_0$ and $y' = R_{\varepsilon_{-n-1},u}(y_0)$, we obtain
	\begin{eqnarray*}
		\Big\|\sum_{n=0}^{\infty}|X_{n,u}(y_0) - X_{n+1,u}(y_0)|\Big\|_q^{q'} &\le& \sum_{n=0}^{\infty}\|X_{n,u}(y_0) - X_{n+1,u}(y_0)\|_q^{q'}\\
		&\le& C_{\lambda}\sum_{n=0}^{\infty}\lambda_0^{n+1} \cdot \| |y_0 - R_{\varepsilon_{-n-1},u}(y_0)|_1\|_q^{q'}< \infty.
	\end{eqnarray*}
	By Markov's inequality and Borel-Cantelli's lemma, this shows that $(X_{n,u}(y_0))_{n\in\IN}$ is a Cauchy sequence a.s. and thus has an almost sure limit $\tilde X_0(u)$ (say). Furthermore, we have
	\[
		\| X_{n,u}(y_0)\|_q^{q'} \le |y_0|_1^{q'} + \sum_{k=0}^{n-1}\|X_{k+1,u}(y_0) - X_{k,u}(y_0)\|_q^{q'} \le  |y_0|_1^{q'} + \frac{C_\lambda \lambda_0}{1-\lambda_0}\| |y_0 - R_{\varepsilon_{-n-1},u}(y_0)|_1\|_q^{q'}.
	\]
	By Fatou's lemma,
	\[
		\sup_{u\in[0,1]}\|\tilde X_0(u)\|_q^{q'} \le \sup_{u\in[0,1]}\liminf_{n\to\infty}\|X_{n,u}(y_0)\|_q^{q'} < \infty,
	\]
	since $\sup_{u\in[0,1]}\| G_{\varepsilon_0}(y_0, u)\|_q < \infty$ by assumption.\\
	Since $\tilde X_0(u)$ is $\sF_0$-measurable, we can write $\tilde X_0(u) = H(u,\sF_0)$ for some measurable function $H$. By \reff{DP_rek_1_zwischenresultat}, $X_{n,u}(y)$ converges almost surely to the same limit $\tilde X_0(u)$ for arbitrary $y\in\IR^p$. This shows a.s. uniqueness among all $\sF_0$-measurable processes and we can express $\tilde X_t(u) = H(u,\sF_t)$ a.s. Put $\tilde X_t^{*0}(u) = H(u,\sF_t^{*0})$ for $t\in\IZ$. Because $\tilde X_t(u)$ obeys \reff{DP_rek_2}, we have for $X_t^{*0}(u) = H(u,\sF_t^{*0})$ by \reff{DP_cond_1}:
	\[
		\| \tilde X_t(u) - \tilde X_t^{*0}(u)\|_q^{q'} \le \sum_{j=1}^{p}\chi_j \|\tilde X_{t-j}(u) - \tilde X_{t-j}^{*0}(u)\|_q^{q'}
	\]
	By Lemma \ref{DP_standard_rec_argument}, we conclude $\big(\delta_q^{\tilde X(u)}(k)\big)^{q'} =  \| \tilde X_t(u) - \tilde X_t^{*0}(u)\|_q^{q'} \le 2pC_{\lambda}\lambda_0^{t} \|\tilde X_0(u)\|_q^{q'}$.

	(ii) Because $X_{0,n} = \tilde X_0(0)$ by means of \reff{DP_rek_1}, the existence and the a.s. uniqueness statement is obvious from Proposition \ref{DP_rek_1_sol}(i). From \reff{DP_cond_1} and the triangle inequality, we obtain
	\begin{eqnarray*}
		\|X_{t,n}\|_q^{q'} &\le& \sum_{j=1}^{p}\chi_j \|X_{t-j,n} - y_{0j}\|_q^{q'} + \big\| G_{\varepsilon_0}\big(y_0,\frac{t}{n}\big)\big\|_q^{q'}\\
		&\le& \sum_{j=1}^{p}\chi_j \| X_{t-j,n}\|_q^{q'} + |y_0|_1^{q'} + \sup_{u\in [0,1]} \|G_{\varepsilon_0}(y_0,u)\|_q^{q'}.
	\end{eqnarray*}
	Since $\|X_{s,n}\|_q^{q'} = \|\tilde X_0(0)\|_q^{q'}$ for $s \le 0$, Lemma \ref{DP_standard_rec_argument} implies
	\[
		\|X_{t,n}\|_q^{q'} \le C_{\lambda}p\lambda_0^{t}\|\tilde X_0(0)\|_q^{q'} + (1-\lambda_0)^{-1}\big(|y_0|_1^{q'} + \sup_{u\in [0,1]} \|G_{\varepsilon_0}(y_0,u)\|_q^{q'}\big)
	\]
	for all $t =1,...,n$, which gives $\sup_{n\in\IN}\sup_{t=1,...,n}\|X_{t,n}\|_q^{q'} < \infty$. Note that for arbitrary $t \ge 0$, $k \ge 0$, we have by \reff{DP_cond_1}:
	\[
		\|X_{t,n} - X_{t,n}^{*(t-k)}\|_q^{q'} \le \sum_{j=1}^{p}\chi_j \|X_{t-j,n} - X_{t-j,n}^{*(t-k)}\|_q^{q'}.
	\]
	Note that $z_s := \|X_{s+(t-k),n} - X_{s+(t-k),n}^{*(t-k)}\|_q^{q'} = 0$ for $s < 0$ and furthermore,  $z_0 \le 2 \sup_{n\in\IN}\sup_{t=1,...,n}\|X_{t,n}\|_q^{q'}$. Lemma \ref{DP_standard_rec_argument} implies
	\[
		\big(\delta_q^{X_{\cdot,n}}(k)\big)^{q'} = \|X_{t,n} - X_{t,n}^{*(t-k)}\|_q^{q'} = z_k \le 2C_{\lambda}\lambda_0^{k} \sup_{n\in\IN}\sup_{t=1,...,n}\|X_{t,n}\|_q^{q'}.
	\]
\end{proof}

\begin{proof}[\normalfont \textbf{Proof of Lemma \ref{DP_lemma_1}}]
	The first inequality \reff{DP_x_stat_hoelder} is a consequence of
	\begin{eqnarray*}
		&& \|\tilde X_t(u) - \tilde X_t(u')\|_q^{q'}\\
		&\le& \|G_{\varepsilon_t}(\tilde Y_{t-1}(u), u) - G_{\varepsilon_t}(\tilde Y_{t-1}(u),u')\|_q^{q'} + \|G_{\varepsilon_t}(\tilde Y_{t-1}(u), u') - G_{\varepsilon_t}(\tilde Y_{t-1}(u'),u')\|_q^{q'}\\
		&\le& \|C(\tilde Y_{t-1}(u))\|_q^{q'}|u-u'|^{\alpha q'} + \sum_{j=1}^{k}\chi_j \|\tilde X_{t-j}(u) - \tilde X_{t-j}(u')\|_q^{q'}\\
		&\le& C^{q'}|u - u'|^{\alpha q'} + |\chi|_1 \cdot  \|\tilde X_t(u) - \tilde X_t(u')\|_q^{q'}.
	\end{eqnarray*}
	For the second inequality, note that we have for all $s = 1,...,n$:
	\begin{eqnarray*}
		&& \left\|X_{s,n} - \tilde X_{s}\left(\frac{s}{n}\right)\right\|_q^{q'} = \left\|G_{\varepsilon_t}\left(Y_{s-1,n}, \frac{s}{n}\right) - G_{\varepsilon_t}\left(\tilde Y_{s-1}\left(\frac{s}{n}\right),\frac{s}{n}\right)\right\|_q^{q'}\\
		&\le& \sum_{i=1}^{p}\chi_i \cdot \left\|X_{s-i,n} - \tilde X_{s-i}\left(\frac{s}{n}\right)\right\|_q^{q'}\\
		&\le& \sum_{i=1}^{p}\chi_i \cdot \left\|X_{s-i,n} - \tilde X_{s-i}\left(\frac{s-i}{n}\vee 0\right)\right\|_q^{q'} + \sum_{i=1}^{p}\chi_i \cdot \left\|\tilde X_{s-i}\left(\frac{s-i}{n} \vee 0\right) - \tilde X_{s-i}\left(\frac{s}{n}\right)\right\|_q^{q'}\\
		&\le&  \sum_{i=1}^{p}\chi_i \cdot \left\|X_{s-i,n} - \tilde X_{s-i}\left(\frac{s-i}{n} \vee 0\right)\right\|_q^{q'} + C^{q'} p^{\alpha q'}\frac{|\chi|_1}{1-|\chi|_1}\cdot n^{-\alpha q'}.
	\end{eqnarray*}
	Define $z_s := \|X_{s,n} - \tilde X_s(\frac{s}{n}\vee 0)\|_q^{q'}$. Note that $z_s = 0$ for $s \le 0$ and define $\mu := C^{q'} p^{\alpha q'}\frac{|\chi|_1}{1-|\chi|_1}\cdot n^{-\alpha q'}$. In this special case we can calculate the constants from Lemma \ref{DP_standard_rec_argument} directly, since $z_{s-i_1-...-i_s} = 0$ for $i_1,...,i_s \in \{1,...,p\}$:
	\[
		z_s \le \sum_{i_1=1}^{p}\chi_{i_1}z_{s-i_1} + \mu \le \sum_{i_1,i_2 = 1}^{p}\chi_{i_1}\chi_{i_2}z_{s-i_1-i_2}  + \mu(1+|\chi|_1) \le ... \le \mu(1 + |\chi|_1 + ... + |\chi|_1^{s-1}).
	\]
	which yields $z_s \le \frac{\mu}{1-|\chi|_1}$ and thus
	\[
		\sup_{s=1,...,n}\left\|X_{s,n} - \tilde X_{s}\left(\frac{s}{n}\right)\right\|_q^{q'} \le C^{q'} p^{\alpha q'}\frac{|\chi|_1}{(1-|\chi|_1)^2} n^{-\alpha q'}.
	\]
\end{proof}

\begin{proof}[\normalfont \textbf{Proof of Theorem \ref{DP_thm_cont}}]
	With out loss of generality, we prove the statement for $t = 0$. Because of the continuity of $G$, the process $(X_{n,u}(y_0))_{u\in[0,1]}$ is continuous and thus a random element of the normed space $(C[0,1],|\cdot|_{\infty})$ where $|\cdot|_{\infty}$ denotes the supremum norm on $[0,1]$.
	With condition \reff{DP_continuous_cond} we obtain for two functions $u \mapsto y(u),\tilde y(u)$:
	\[
		\Big\|\sup_{u\in[0,1]}|X_{n,u}(y) - X_{n,u}(\tilde y)|\Big\|_{q}^{q'} \le \sum_{j=1}^{p}\chi_j\cdot \Big\|\sup_{u \in [0,1]}|X_{n-j,u}(y) - X_{n-j,u}(\tilde y)|\Big\|_{q}^{q'}.
	\]
	Lemma \ref{DP_standard_rec_argument} implies that there exist $C_{\lambda} > 0$, $0 \le \lambda < 1$ such that
	\begin{equation}
		\big\|\sup_{u\in[0,1]}|X_{n,u}(y) - X_{n,u}(\tilde y)|\big\|_{q}^{q'} \le C_{\lambda} \lambda_0^{n+1} \sup_{u\in[0,1]}|y-\tilde y|^{q'}_1.\label{DP_tightnessarg2}
	\end{equation}
	Taking $y(u) = y_0$, $\tilde y(u) = R_{\varepsilon_{-n-1},u}(y_0)$, we conclude
	\begin{equation}
		\Big\|\sup_{u\in[0,1]}|X_{n+1,u}(y_0) - X_{n,u}(y_0)|\Big\|_{q}^{q'} \le C_{\lambda} \lambda_0^{n+1} \big\|\sup_{u\in[0,1]}|y_0 - R_{\varepsilon_0}(y_0,u)|_1\big\|_{q}^{q'}.\label{DP_tightnessarg1}
	\end{equation}
	Markov's inequality and Borel-Cantelli's lemma implies that the sequence $(X_{n,u}(y_0))_{u\in[0,1]}$, $n\in\IN$ of elements of $C[0,1]$ is a Cauchy sequence in $(C[0,1],|\cdot|_{\infty})$ almost surely. Since this space is complete, there exists a continuous limit $\hat X_0 = (\hat X_0(u))_{u\in[0,1]}$. It was already shown in the proof of Proposition \ref{DP_rek_1_sol} that $X_{n,u}(y_0) \to \tilde X_0(u)$ a.s. for fixed $u\in[0,1]$. This implies that $\hat X_0$ is a continuous modification of $(\tilde X_0(u))_{u\in[0,1]}$.
	By \reff{DP_tightnessarg1}, we have
	\begin{eqnarray}
		&& \Big\|\sup_{u\in[0,1]}|X_{n,u}(y_0)|\Big\|_{q}^{q'} \le \sum_{k=0}^{n-1}\Big\|\sup_{u\in[0,1]}|X_{k,u}(y_0) - X_{k+1,u}(y_0)|\Big\|_{q}^{q'} + |y_0|_1^{q'}\nonumber\\
		&\le& \frac{C_{\lambda}\lambda_0}{1-\lambda_0}\big\|\sup_{u\in[0,1]}|y_0 - R_{\varepsilon_0}(y_0,u)|_1\big\|_{q}^{q'} + |y_0|_1^{q'} =: D^{q'}.\label{DP_tightnessarg_eq1}
	\end{eqnarray}
	Because for $M\in\IN$, $M \wedge \sup_{u\in[0,1]}|\cdot|$ is a bounded and continuous functional, we obtain $ \big\|M \wedge \sup_{u\in[0,1]}|\hat X_0(u)|\big\|_{q} \le D$. The monotone convergence theorem implies $ \sup_{u\in[0,1]}|\hat X_t(u)| \in L^{q}$.
\end{proof}

\begin{proposition}\label{DP_prop_continuous}
	In the situation of Theorem \ref{DP_thm_cont}, instead of \reff{DP_continuous_cond} assume that $x \mapsto G_{\varepsilon}(x,u)$ is differentiable for all $\varepsilon,u$ and that for all $u_0 \in [0,1]$,
	\[
		\limsup_{\delta \to 0}\big\|\sup_{|u-u_0| \le \delta}\sup_{x}|\partial_1 G_{\varepsilon_0}(x,u) - \partial_{1}G_{\varepsilon_0}(x,u_0)|_1\big\|_{q} = 0
	\]
	and
	\[
		\sup_{u \in [0,1]}\Big\|\sup_{y\not=y'}\frac{|G_{\varepsilon_0}(y,u) - G_{\varepsilon_0}(y',u)|}{|y-y'|_{\chi,q'}}\Big\|_{q} \le 1.
	\]
	Then the results of Theorem \ref{DP_thm_cont} are still valid.
\end{proposition}
\begin{proof}[\normalfont \textbf{Proof of Proposition \ref{DP_prop_continuous}}]
	For fixed $u_0 \in [0,1]$, the fundamental theorem of calculus gives
	\begin{eqnarray*}
		&& G_{\varepsilon_0}(y,u) - G_{\varepsilon_0}(y',u)\\
		&=& \int_0^{1} \langle \partial_1 G_{\varepsilon_0}(y' + s\cdot (y-y'),u) - \partial_1 G_{\varepsilon_0}(y' + s\cdot (y-y'),u_0), y-y'\rangle \dif s\\
		&&\quad + \big(G_{\varepsilon_0}(y,u_0) - G_{\varepsilon_0}(y',u_0)\big).
	\end{eqnarray*}
	The first term is bounded in absolute value by $\sup_{x}|\partial_1 G_{\varepsilon_0}(x,u) - \partial_1 G_{\varepsilon_0}(x,u_0)|_1 \cdot |y-y'|_{\infty}$.
	Since $|\chi|_1 < 1$, we can assume w.l.o.g. that $\chi_j > 0$ for all $j = 1,...,p$ (if for instance $\chi_1 = 0$, one can define $\chi' := \chi + (1-|\chi|_1/2,0,...,0)$ which still fulfills $|\chi'|_1 < 1$). Now choose $\beta > 1$ such that $\beta|\chi|_1 < 1$, and define $\chi' := \delta \chi$ for some $\delta > 0$. We have $|y-y'|_{\infty} \le \frac{1}{\min(\chi')}|y-y'|_{\chi',q'}$. For $\delta$ small enough, we have
	\begin{eqnarray*}
		&& \Big\|\sup_{|u-u_0| \le \delta}\sup_{x\not=y}\frac{|G_{\varepsilon_0}(y,u) - G_{\varepsilon_0}(y',u)|}{|y-y'|_{\chi',q'}}\Big\|_{q}^{q'}\\
		&\le& \frac{1}{\min(\chi')^{q'}}\Big\|\sup_{|u-u_0| \le \delta}\sup_x|\partial_1 G_{\varepsilon_0}(x,u) - \partial_1 G_{\varepsilon_0}(x,u_0)|_1\Big\|_{q}^{q'}\\
		&&\quad\quad\quad\quad\quad\quad\quad\quad\quad\quad + \frac{1}{\beta^{q'}}\sup_{|u-u_0| \le \delta}\Big\|\sup_{x\not=y}\frac{|G_{\varepsilon_0}(y,u) - G_{\varepsilon_0}(y',u)|}{|y-y'|_{\chi,q'}}\Big\|_{q}^{q'} < 1.
	\end{eqnarray*}
	Fix $t\in\IZ$. Partitioning of $[0,1]$ into (overlapping) closed intervals $I_1,...,I_K$ of length at most $\delta$ and applying Theorem \ref{DP_thm_cont} on each of these intervals $I_k$, $k = 1,...,K$ provides the existence of a continuous modification of $(\hat X_t^{(k)}(u))_{u \in I_k}$ of $(\tilde X_t(u))_{u\in I_k}$ on each of these subintervals with $\sup_{u\in I_k}|\hat X_t^{(k)}(u)| \in L^{q}$. For fixed $k,k' \in \{1,...,K\}$ with $I_{k} \cap I_{k'} \not= \emptyset$ the continuous processes $(\hat X_t^{(k)}(u))_{u \in I_k}$, $(\hat X_t^{(k')}(u))_{u \in I_{k'}}$ are a.s. equal on $I_{k} \cap I_{k'}$ which ensures continuity of a process $(\hat X_t(u))_{u \in [0,1]}$ which is assembled from $(\hat X_t^{(k)}(u))_{u \in I_k}$, $k = 1,...,K$ and thus a modification of $(\tilde X_t(u))_{u\in[0,1]}$.
\end{proof}

\begin{proof}[\normalfont \textbf{Proof of Theorem \ref{DP_thm_diff}}] (i) Note that Assumption \ref{DP_ass1}\ref{DP_ass1_l2},\ref{DP_ass1_l3} imply \ref{DP_ass1}\ref{DP_ass1_l1} and \reff{DP_cond_3}. We will only use these conditions for the following proof. Since the process $\tilde X_t(u)$ is already known to exist, we are able to define a new recursion function based on $\tilde X_t(u)$. For $y\in\IR^p$, define the random map $\hat G_{t}(y,u) := \langle \partial_1 G_{\varepsilon_t}(\tilde Y_{t-1}(u), u), y\rangle + \partial_2 G_{\varepsilon_t}(\tilde Y_{t-1}(u),u)$ and  $\hat R_{t,u}(y) := (\hat G_{t}(y,u), y_1,...,y_{p-1})$, and let $D_{t,n,u}(y)$ be the first element of $\hat R_{t,u} \circ \hat R_{t-1,u} \circ ... \circ \hat R_{t-n,u}(y)$ for $n\in\IN$. For $y,y'\in\IR^p$, \reff{DP_cond_1} and Fatou's lemma imply
	\begin{align}
		& \|\hat R_{t,u}(y) - \hat R_{t,u}(y')\|_q = \| \langle \partial_1 G_{\varepsilon_t}(\tilde Y_{t-1}(u),u), y-y'\rangle\|_q\nonumber\\
		&\le \liminf_{h\to 0}\frac{\| G_{\varepsilon_t}(\tilde Y_{t-1}(u) + h(y-y'), u) - G_{\varepsilon_t}(\tilde Y_{t-1}(u),u)\|_q}{h}\nonumber\\
		&\le \liminf_{h\to 0}\frac{\| G_{\varepsilon_t}(\tilde Y_{t-1}(u) + h(y-y'), u) - G_{\varepsilon_t}(\tilde Y_{t-1}(u),u)\|_q}{|h(y-y')|_{\chi,q'}}\cdot |y-y'|_{\chi,q'} \le |y-y'|_{\chi,q'}.\label{DP_rek_3_sol_eq1}
	\end{align}
	Similar to the proof of Proposition \ref{DP_rek_1_sol}, we obtain $C_{\lambda}> 0, \lambda_0\in(0,1)$ with
	\[
		\|D_{t,n,u}(y) - D_{t,n,u}(y')\|_q^{q'} \le C_{\lambda}\cdot \lambda_0^{n+1}|y-y'|^{q'}.
	\]
	Applying this to $y = y_0$ and $y' = \hat R_{t-n-1,u}(y_0)$ we obtain
	\[
		\Big\|\sum_{n=0}^{\infty}|D_{t,n,u}(y_0) - D_{t,n+1,u}(y_0)|\Big\|_q^{q'} \le C_{\lambda}\sum_{n=0}^{\infty}\lambda_0^{n+1} \cdot \| |y_0 - \hat R_{t-n-1,u}(y_0)|_1\|_q^{q'}
	\]
	which is finite by \reff{DP_cond_3} and \reff{DP_rek_3_sol_eq1}. This implies that $D_{0,n,u}(y_0)$ converges a.s. to some limit $D_0(u)$, say. Because $\tilde X_k(u) \in \sF_k$ ($k\in\IZ$), it is obvious that $D_0(u)$ is $\sF_0$-measurable and therefore has a representation $D_0(u) = \hat H(u,\sF_0)$. The rest of the proof is the same as in Proposition \ref{DP_rek_1_sol}(i).

(ii) Because of the continuous differentiability of $G$, the process $(X_{n,u}(y_0))_{u\in[0,1]}$ is a random element of $(C^1[0,1],|\cdot|_{C^1})$, where $\|f\|_{C^1} = |f|_{\infty}+ |f'|_{\infty}$ and $|\cdot|_{\infty}$ denotes the supremum norm on $[0,1]$.
	
	We will only consider the case that Assumption \ref{DP_ass1}\ref{DP_ass1_l3}(a) is fulfilled. In the case of Assumption \ref{DP_ass1}\ref{DP_ass1_l3}(b), one can set $\tilde q = q$ in the following with obvious changes in the proofs.
	
	Define $\tilde q := q/2$ and $\tilde q' := \min\{\tilde q, 1\}$. Let $u \mapsto y_1(u),y_2(u) \in \IR^p$ be two differentiable functions (for brevity, we will omit the argument $u$ in the following). Because of $X_{n,u}(y) = G_{\varepsilon_0,u}(J_{n-1,u}(y))$, we have:
	\begin{equation*}
		\partial_u X_{n,u}(y_1) = \langle \partial_1 G_{\varepsilon_0}(J_{n-1,u}(y_1),u), \partial_u J_{n-1,u}(y_1)\rangle + \partial_2 G_{\varepsilon_0}(J_{n-1,u}(y_1),u).\label{DP_thm_diff_eq3}
	\end{equation*}
	This shows (use similar techniques as in \reff{DP_rek_3_sol_eq1}):
	\begin{align}
		& \big\|\sup_{u\in[0,1]}|\partial_u X_{n,u}(y_1)|\big\|_q^{q'}\nonumber\\
		&\le \sum_{j=1}^{p}\chi_j \big\|\sup_u |\partial_u X_{n-j,u}(y_1)|\big\|_q^{q'} + \big\| \sup_{u}|\partial_2 G_{\varepsilon_0,u}(J_{n-1,u}(y_1))|\big\|_q^{q'}\nonumber\\
		&\le  \sum_{j=1}^{p}\chi_j \big\|\sup_u |\partial_u X_{n-j,u}(y_1)|\big\|_q^{q'} + C_2^{q'} \sum_{j=1}^{p}\big\|\sup_u|X_{n-j,u}(y_1)|\big\|_q^{q'} + \big\|\sup_u |\partial_2 G_{\varepsilon_0}(0,u)|\big\|_q^{q'}.\label{DP_thm_diff_eq100}
	\end{align}
	The third term is finite by assumption (follows from \reff{DP_cond_9} and $\|\sup_{u}|\partial_2 G_{\varepsilon_0}(y_0,u)| \|_{q} < \infty$). In the proof of Theorem \ref{DP_thm_cont} it was shown that there exist $C_{\lambda}' > 0$, $0 \le \lambda_0' < 1$ such that for all $n\in\IN$: $\|\sup_u|X_{n,u}(y_1)|\|_q^{q'} \le D(y_1)^{q'}$ with
	\[
		D(y_1)^{q'} := \frac{C_{\lambda}'\lambda_0'}{1-\lambda_0'}\|\sup_{u\in [0,1]}|y_1 - R_{\varepsilon_0}(y_1,u)|_1 \|_q^{q'} + \sup_{u\in[0,1]} |y_1|_1^{q'},
	\]
	see \reff{DP_tightnessarg_eq1}. Since $|\chi|_1 < 1$, Lemma \ref{DP_standard_rec_argument} and \reff{DP_thm_diff_eq100} imply that there exist $C_{\lambda} > 0$, $0 \le \lambda_0 < 1$ such that for all $n\in\IN$:
	\begin{align}
		& \big\|\sup_{u\in[0,1]}|\partial_u X_{n,u}(y_1)|\big\|_q^{q'}\nonumber\\
		& \le C_{\lambda}\big(\sup_{u}|\partial_u y_1|_1^{q'} + (1-\lambda_0)^{-1}\big(C_2^{q'} p D(y_1)^{q'} +  \big\|\sup_u |\partial_2 G_{\varepsilon_0}(0,u)|\big\|_q^{q'}\big) =: E(y_1)^{q'}.\label{DP_thm_diff_eq5}
	\end{align}
	Using the triangle inequality, we obtain
	\begin{eqnarray*}
		&& \big\|\sup_{u\in[0,1]}\big|\partial_u X_{n,u}(y_1)- \partial_u X_{n,u}(y_2)\big|\big\|_{\tilde q}^{\tilde q'}\\
		&\le& \big\| \sup_{u \in [0,1]}\big|\langle \partial_1 G_{\varepsilon_0,u}(J_{n-1,u}(y_1),u) - \partial_1 G_{\varepsilon_0,u}(J_{n-1,u}(y_2),u), \partial_u J_{n-1,u}(y_1)\rangle\big|\big\|_{\tilde q}^{\tilde q'}\\
		&&\quad + \big\| \sup_{u\in[0,1]}\big|\langle \partial_1 G_{\varepsilon_0,u}(J_{n-1,u}(y_2),u), \partial_u J_{n-1,u}(y_1) - \partial_u J_{n-1,u}(y_2)\rangle\big|\big\|_{\tilde q}^{\tilde q'}\\
		&&\quad + \big\| \sup_{u\in[0,1]}\big| \partial_2 G_{\varepsilon_0,u}(J_{n-1,u}(y_1),u) - \partial_2 G_{\varepsilon_0,u}(J_{n-1,u}(y_2),u)\big|\big\|_{\tilde q}^{\tilde q'} =: A_{1} + A_2 + A_3.
	\end{eqnarray*}
	Condition \reff{DP_cond_9} and the result \reff{DP_tightnessarg2} from the proof of Theorem \ref{DP_thm_cont} (use $\tilde C_{\lambda}$, $\tilde \lambda_0$ for the result therein) implies
	\begin{eqnarray*}
		A_3 &\le& C_2^{\tilde q'} \cdot \big\| \sup_{u\in[0,1]}|J_{n-1,u}(y_1) - J_{n-1,u}(y_2)|_1\big\|_{\tilde q}^{\tilde q'}\\
		&\le& C_2^{\tilde q'} \cdot \Big(\sum_{j=1}^{p}\big\|\sup_{u\in[0,1]}|X_{n-j,u}(y_1) - X_{n-j,u}(y_2)|\big\|_{q}^{q'}\Big)^{\tilde q'/q'}\\
		&\le& C_2^{\tilde q'} \big(\tilde C_{\lambda} p\tilde\lambda_0^{n-p}\big)^{\tilde q'/q'}\sup_u |y_1-y_2|_1^{\tilde q'}.
	\end{eqnarray*}
	Using Assumption \ref{DP_ass1}\ref{DP_ass1_l3}(a), a similar technique as in \reff{DP_rek_3_sol_eq1} gives
	\[
		A_2 \le \sum_{j=1}^{p}\chi_j \big\|\sup_{u\in[0,1]}|\partial_u X_{n-j,u}(y_1) - \partial_u X_{n-j,u}(y_2)|\big\|_{\tilde q}^{\tilde q'}.
	\]
	By the Cauchy-Schwarz inequality, we have
	\begin{eqnarray*}
		A_{1} &\le& \sum_{j=1}^{p}\big\|\sup_{u\in[0,1]}\big|\big( \partial_1 G_{\varepsilon_0}(J_{n-1,u}(y_1),u) - \partial_1 G_{\varepsilon_0}(J_{n-1,u}(y_2),u)\big)_j\big|\big\|_q^{\tilde q'}\\
		&&\quad\quad\times \big\| \big|\partial_u J_{n-1,u}(y_1)_j\big|\big\|_{q}^{\tilde q'}\\
		&\le&  C_1^{\tilde q'} \sum_{j=1}^{p}\big\|\sup_{u\in[0,1]}\big| J_{n-1,u}(y_1) - J_{n-1,u}(y_2)\big|_1\big\|_q^{\tilde q'} \cdot \big\| \big|\partial_u J_{n-1,u}(y_1)_j\big|\big\|_{q}^{\tilde q'}\\
		&\le& C_1^{\tilde q'} \sum_{j=1}^{p}\Big(\sum_{i=1}^{p}\big\|\sup_{u\in[0,1]}\big|X_{n-i,u}(y_1) - X_{n-i,u}(y_2)\big|\big\|_{q}^{q'}\Big)^{\tilde q'/q'}\\
		&&\quad\quad \times \big\|\sup_{u\in[0,1]}\big|\partial_u X_{n-j,u}(y_1)\big|\big\|_{q}^{\tilde q'}\\
		&\le& C_1^{\tilde q'} p E(y_1)^{\tilde q'}\big(\tilde C_{\lambda} p \lambda_0^{n-p}\big)^{\tilde q'/q'}\sup_u|y_1-y_2|_1^{\tilde q'}
	\end{eqnarray*}
	Finally we have shown that exists a constant $C(y_1) > 0$ such that
	\begin{eqnarray*}
		&& \Big\|\sup_{u\in[0,1]}\big|\partial_u X_{n,u}(y_2)- \partial_u X_{n,u}(y_1)\big|\Big\|_{\tilde q}^{\tilde q'}\\
		&\le& \sum_{j=1}^{p}\chi_j \Big\|\sup_{u\in[0,1]}\big|\partial_u X_{n-j,u}(y_2)- \partial_u X_{n-j,u}(y_1)\big|\Big\|_{\tilde q}^{\tilde q'} + C(y_1) \big(\tilde\lambda_0^{\tilde q'/q'}\big){n} \sup_u|y_1-y_2|_1^{\tilde q'}.
	\end{eqnarray*}
	Lemma \ref{DP_standard_rec_argument} implies that there exist constants $C_{\lambda} > 0$, $\lambda_0 \in (0,1)$ such that for $n\in\IN$:
	\begin{eqnarray*}
		 && \Big\|\sup_{u\in[0,1]}\big|\partial_u X_{n,u}(y_1)- \partial_u X_{n,u}(y_2)\big|\Big\|_{\tilde q}^{\tilde q'}\\
		 &\le& C_{\lambda}\big(\lambda_0^{n+1}\sup_u|\partial_u y_1 - \partial_u y_2|_1^{\tilde q'} + C(y_1)\sum_{i=0}^{n}\lambda_0^{i} \big(\tilde \lambda_0^{\tilde q'/q'}\big)^{n-i}\big)\sup_u|y_1-y_2|_1^{\tilde q'}.
	\end{eqnarray*}
	Taking $y_1(u) \equiv y_0$, $y_2(u) = R_{\varepsilon_0}(y_0,u)$ and using the inequalities
	\[
		\|\sup_u|\partial_u y_1 - \partial_u y_2|_1\|_{\tilde q} \le \|\sup_u|\partial_2 G_{\varepsilon_0}(y_0,u)|\|_{q} < \infty
	\]
	and $\|\sup_u|y_1 - y_2|_1\|_{\tilde q} \le \|\sup_u |y_0 - R_{\varepsilon_0}(y_0,u)|_1\|_{q} < \infty$ by assumption, we obtain that for all $n\in\IN$:
	\begin{equation}
			\big\|\sup_{u\in[0,1]}|\partial_u X_{n+1,u}(y_0)- \partial_u X_{n,u}(y_0)|\big\|_{\tilde q}^{\tilde q} \le \hat C_{\lambda}(y_0) \hat \lambda_0^{n}\label{DP_thm_diff_eq1}
	\end{equation}
	with $0 < \hat \lambda_0 := \max(\lambda_0, \tilde \lambda_0^{\tilde q'/q'}) < 1$ and some constant $\hat C_{\lambda}(y_0) > 0$. Together with the result \reff{DP_tightnessarg1}, Markov's inequality and Borel-Cantelli's lemma, we obtain that the sequence $(X_{n,u}(y_0))_{u\in[0,1]}$, $n\in\IN$ of elements of $C^1[0,1]$ is a Cauchy sequence in $(C^1[0,1],|\cdot|_{C^1})$ almost surely. Since this space is complete, there exists a continuously differentiable limit $\hat X_0 = (\hat X_0(u))_{u\in[0,1]}$. Because $\hat X_0$ is $\sF_0$-measurable, there exists a measurable function $\hat H = (\hat H(u,\cdot))_{u\in [0,1]}:\IR^{\IN} \to C^{1}[0,1]$ such that $u \mapsto \hat H(u,z)$ is continuously differentiable for all $z \in \IR^{\IN}$. For arbitrary $t\in\IZ$, we may define $\partial_u  \hat X_t(u) := \partial_u \hat H(u,\sF_t)$. The process $X_{t,n,u}(y)$ defined similarly as $X_{n,u}(y)$ but with $\varepsilon_{0},...,\varepsilon_{-n}$ replaced by $\varepsilon_t,...,\varepsilon_{t-n}$ has the same distributional properties as $X_{n,u}(y)$ and therefore $X_{t,n,u}(y) \to \hat H(u,\sF_t)$ a.s. and $\partial_u X_{t,n,u}(y) \to \partial_u \hat H(u,\sF_t)$ a.s. By construction it holds that
	\[
		X_{t,n,u}(y) = G_{\varepsilon_t}(X_{t-1,n-1,u}(y),u)
	\]
	and
	\[
		\partial_u X_{t,n,u}(y) = \langle \partial_1 G_{\varepsilon_t}(X_{t-1,n-1,u}(y),u), \partial_u X_{t-1,n-1,u}(y)\rangle + \partial_2 G_{\varepsilon_t}(X_{t-1,n-1,u}(y),u),
	\]
	thus we obtain for $n\to\infty$ that $\hat X_t(u)$ fulfills \reff{DP_rek_2} and $\partial_u \hat X_t(u)$ fulfills \reff{DP_rek_3} a.s. for all $t\in\IZ$. Since \reff{DP_rek_2}, \reff{DP_rek_3} only allow for a.s. unique solutions, we conclude that $(\hat X_t(u))_{u\in[0,1]}$ is a continuously differentiable modification of $(\tilde X_t(u))_{u\in[0,1]}$ and $(\partial_u \hat X_t(u))_{u\in[0,1]}$ is a continuous modification of $(D_t(u))_{u\in[0,1]}$.\\
	The uniform convergence $\sup_{u}|\partial_u X_{n,u}(y_0) - \partial_u \hat X_0(u)| \to 0$ together with Fatou's lemma and \reff{DP_thm_diff_eq5} implies $\sup_u |\partial_u \hat X_0(u)| \in L^{q}$.
\end{proof}

\begin{proof}[\normalfont \textbf{Proof of Lemma \ref{DP_hoelder_derivative}}]
	Define $\tilde q := q/2$ and $\tilde q' := \min\{\tilde q,1\}$. Let $u,u' \in [0,1]$. Because $\partial_u \tilde X_t(u)$ obeys \reff{DP_rek_3}, we have by the Cauchy Schwarz inequality:
	\begin{eqnarray}
		&& \|\partial_u \tilde X_t(u) - \partial_u \tilde X_{t}(u')\|_{\tilde q}^{\tilde q'}\nonumber\\
		&\le& \sum_{j=1}^{p}\big\|\big(\partial_1 G_{\varepsilon_t}(\tilde Y_{t-1}(u),u) - \partial_1 G_{\varepsilon_t}(\tilde Y_{t-1}(u'),u')\big)_j\big\|_{q}^{\tilde q'}\cdot \|\partial_u \tilde X_{t-j}(u)\|_{q}^{\tilde q'}\nonumber\\
		&& + \big\|\langle\partial_{1} G_{\varepsilon_t}(\tilde Y_{t-1}(u'),u'), \partial_{u}\tilde X_{t-1}(u) - \partial_{u}\tilde X_{t-1}(u')\rangle \big\|_{\tilde q}^{\tilde q'}\nonumber\\
		&& +  \|\partial_2 G_{\varepsilon_t}(\tilde Y_{t-1}(u),u) - \partial_2 G_{\varepsilon_t}(\tilde Y_{t-1}(u'),u')\|_{\tilde q}^{\tilde q'}.\label{DP_hoelder_derivative_eq1}
	\end{eqnarray}
	\reff{DP_cond_9} and \reff{DP_cond_20} give
	\[
		\|\partial_2 G_{\varepsilon_t}(\tilde Y_{t-1}(u),u) - \partial_2 G_{\varepsilon_t}(\tilde Y_{t-1}(u'),u')\|_{\tilde q}^{\tilde q'} \le C_2^{\tilde q'}p^{\tilde q'/q'}\cdot \|\tilde X_t(u) - \tilde X_t(u')\|_q^{\tilde q'} + D_2^{\tilde q'}|u-u'|^{\alpha_2 \tilde q'}.	
	\]
	Similar results are obtained for the first term in \reff{DP_hoelder_derivative_eq1}. Note that $\|\sup_{u}|\partial_u \tilde X_t(u)|\|_q \le M$ with some $M > 0$ by Theorem \ref{DP_thm_diff}. The conditions of Lemma \ref{DP_lemma_1} are fulfilled for $\alpha = 1$, alternatively it can be seen directly that
	\[
		\|\tilde X_t(u) - \tilde X_t(u')\|_q = \Big\|\int_0^{1}|\partial_u \tilde X_t(u' + (u-u')s) \dif s\Big\|_q |u-u'| \le \big\| \sup_v |\partial_u \tilde X_t(v)|\big\|_q |u-u'|.
	\]
	A similar technique as in \reff{DP_rek_3_sol_eq1} applied to the second summand in \reff{DP_hoelder_derivative_eq1} implies the inequality $\big\|\langle\partial_{1} G_{\varepsilon_t}(\tilde Y_{t-1}(u'),u'), \partial_{u}\tilde X_{t-1}(u) - \partial_{u}\tilde X_{t-1}(u')\rangle \big\|_{\tilde q}^{\tilde q'} \le  |\chi|_1 \|\partial_u \tilde X_t(u) - \partial_u \tilde X_t(u')\|_{\tilde q}^{\tilde q'}$. We finally obtain 
	\begin{eqnarray*}
		\|\partial_u \tilde X_t(u) - \partial_u \tilde X_t(u')\|_{\tilde q}^{\tilde q'} &\le& |\chi|_1 \|\partial_u \tilde X_t(u) - \partial_u \tilde X_t(u')\|_{\tilde q}^{\tilde q'}\\
		&&\quad + p M^{\tilde q'}\big(C_1^{\tilde q'}p^{\tilde q'/q'}\cdot M^{\tilde q'} |u-u'|^{\tilde q'} + D_1^{\tilde q'}|u-u'|^{\alpha_2 \tilde q'}\big)\\
		&&\quad + \big(C_2^{\tilde q'}p^{\tilde q'/q'}\cdot M^{\tilde q'} |u-u'|^{\tilde q'} + D_2^{\tilde q'}|u-u'|^{\alpha_2 \tilde q'}\big),
	\end{eqnarray*}
	which gives the result since $|\chi|_1 < 1$.
\end{proof}

\subsection{Proofs of Section \ref{DP_section4}}

\begin{proof}[\normalfont \textbf{Proof of Theorem \ref{DP_mle_consistency}, uniform convergence of $\hat \theta_b$}]
Since a sequence converges in probability to some random variable $Z$ if each subsequence has a further subsequence that converges almost surely towards $Z$, we may assume w.l.o.g. that
\begin{equation}
	\sup_{u\in[\frac{b}{2},1-\frac{b}{2}]}\sup_{\theta \in \Theta}|L_{n,b}(u,\theta) - L(u,\theta)|\to 0 \quad a.s.\label{DP_mle_consistency_eq2}
\end{equation}
Since $\theta_0$ is continuous and $\theta_0(u) \in \mbox{int}(\Theta)$ for all $u\in[0,1]$, the whole curve $\theta_0$ has a positive $|\cdot|_1$-distance $c_{min} := \inf_{u\in [0,1]}\mbox{dist}(\theta_0(u),\partial \Theta) > 0$  to the boundary $\partial \Theta$ of $\Theta$. Choose $\varepsilon \in (0, c_{min})$ arbitrarily. For each $u \in D_u = D_u(n) :=  [\frac{b}{2},1-\frac{b}{2}]$, define $\Theta(u, \varepsilon) := \{\theta \in \Theta: |\theta - \theta_0(u)|_1 < \varepsilon\} \not=\emptyset$ (nonempty since $\theta_0(u)$ is in the interior of $\Theta$ by assumption). Define
	\[
		\theta^{*}(u) :\in \argmin_{\theta \in \Theta \cap \Theta(u,\varepsilon)^{c}}L(u,\theta).
	\]
	Here, $\theta^{*}(u)$ does not need to be unique, but we choose one of the possible values. Because $\Theta \cap \Theta(u,\varepsilon)^{c}$ is compact, there has to exist at least one. Because $\theta_0(u)$ is the unique minimum of $\theta \mapsto L(u,\theta)$ over $\Theta$, there exists $\delta(u) > 0$ such that
	\[
		L(u, \theta^{*}(u)) - L(u, \theta_0(u)) = \delta(u).
	\]
	It holds that $\delta := \inf_{u \in [0,1]}\delta(u) > 0$. Otherwise, because of the compactness of $[0,1]$, there would exist a sequence $(u_n) \subset [0,1]$ with $u_n \to u^{*} \in [0,1]$ and $\delta(u_n) \to 0$. By the continuity of $L$, $\theta_0$ and $u \mapsto \inf_{\theta \in \Theta \cap \Theta(u,\varepsilon)^c}L(u,\theta)$ (use Berge's Maximum theorem and the fact that $u \mapsto \Theta \cap \Theta(u, \varepsilon)^c$ is a continuous set function) this would imply
	\[
		0 \leftarrow \delta(u_n) = \inf_{\theta \in \Theta \cap \Theta(u_n, \varepsilon)^c}L(u_n, \theta) - L(u_n, \theta_0(u_n))  \to  \inf_{\theta \in \Theta \cap \Theta(u^{*},\varepsilon)^c}L(u^{*}, \theta) - L(u^{*}, \theta_0(u^{*})),
	\]
	which is a contradiction to the fact that $\theta_0(u^{*}) \in \Theta(u^{*},\delta)$ is the unique minimum of $L(u^{*}, \theta)$. By  \reff{DP_mle_consistency_eq2}, we may choose $N\in\IN$ such that for all $n\ge N$, $\sup_{u \in D_u} \sup_{\theta \in \Theta} |L_{n,b}(u,\theta) - L(u,\theta)| < \frac{\delta}{2}$. Now suppose that for some $n \ge N$, $\sup_{u \in D_u} |\hat \theta_b(u) - \theta_0(u)|_1 \ge \varepsilon$.
	Then we have for some $u \in D_u$ that
	\begin{eqnarray*}
		L_{n,b}(u, \hat \theta_b(u)) &>& L(u, \hat \theta_b(u)) - \frac{\delta}{2} \ge L(u, \theta^{*}(u)) - \frac{\delta}{2}\\
		&=& L(u, \theta_0(u)) + \delta(u) - \frac{\delta}{2} \ge L(u, \theta_0(u)) + \frac{\delta}{2} > L_{n,b}(u, \theta_0(u)),
	\end{eqnarray*}
	which is a contradiction to the extremal property of $\hat \theta_b(u)$.
\end{proof}

\section{Supplement B}
\label{suppB}
This supplement contains another counterexample where Assumption \ref{DP_ass_general2}(M1) is satisfied but not (M2).\\

Let $X_{t,n} = \sum_{k=1}^{\infty}a_{t,n}(k) \varepsilon_{t-k}$ be a linear process and $\tilde X_t(u) = \sum_{k=1}^{\infty}a(u,k) \varepsilon_{t-k}$ the corresponding stationary approximation, where $\varepsilon_i$ are i.i.d., $\alpha > 0$, $a_{t,n}(k) = \frac{1 + \frac{k}{n}}{k^{2+\alpha} + t/n}$ and $a(u,k) = \frac{1}{k^{2+\alpha} + u}$. Then we have
	\begin{eqnarray}
		|a_{t,n}(k) - a(t/n,k)| &\le& n^{-1}\cdot \frac{1}{k^{1+\alpha}},\label{a1}\\
		|a(u,k) - a(u',k)| \le \frac{|u-u'|}{(k^{2+\alpha}+u')(k^{2+\alpha}+u)} &\le& |u-u'|\cdot k^{-4-2\alpha}\label{a2}
	\end{eqnarray}
	which ensures $\|X_{t,n} - \tilde X_t(t/n)\|_q \le n^{-1}\|\varepsilon_0\|_q \sum_{k=1}^{\infty}k^{-2}$ and $\|\tilde X_{t}(u) - \tilde X_t(u')\|_q \le |u-u'| \|\varepsilon_0\|_q \sum_{k=1}^{\infty}k^{-6}$ for $q \ge 1$. However, the processes show different behavior for the dependence measure,
	\begin{eqnarray*}
		&& \delta_{q}^{X_{\cdot,n}}(k) = |a_{t,n}(k)|\cdot \|\varepsilon_0\|_q \sim k^{-2-\alpha} + n^{-1}k^{-1-\alpha},\\
		&\text{while}& \delta_q^{\tilde X(u)}(k) = |a(u,k)| \cdot \|\varepsilon_0\|_q \sim k^{-2-\alpha}.
	\end{eqnarray*}
	If we choose more specifically $\varepsilon_0 \iid N(0,1)$ and $\alpha = 0$, then clearly $\|\tilde X_t(u)\|_2 < \infty$ exists and for $X_{t,n}$ we have
	\[
		X_{t,n} \sim N\Big(0, \sum_{k=1}^{\infty}\Big(\frac{1 + \frac{k}{n}}{k^{2} + t/n}\Big)^2\Big), \quad\quad X_{t,n} - \tilde X_t(t/n) \sim N\Big(0, n^{-2}\sum_{k=1}^{\infty}\Big(\frac{k}{k^{2}+t/n}\Big)^2\Big)
	\]
	which shows $\sup_{t,n}\|X_{t,n}\|_2^2 < \infty$ and $\|X_{t,n} - \tilde X_t(t/n)\|_2 \le n^{-1}\big(\sum_{k=1}^{\infty}k^{-2}\big)^{1/2}$, but
	\begin{eqnarray*}
		&& \delta_{2}^{X_{\cdot,n}}(k) = |a_{t,n}(k)|\cdot \|\varepsilon_0\|_q \sim k^{-2} + n^{-1}k^{-1},\\
		&\text{while}& \delta_q^{\tilde X(u)}(k) = |a(u,k)| \cdot \|\varepsilon_0\|_2 \sim k^{-2},
	\end{eqnarray*}
	i.e. $\sum_{k=1}^{\infty}\delta_{2}^{X_{\cdot,n}}(k) = \infty$ but $\sum_{k=1}^{\infty}\delta_{2}^{\tilde X(u)}(k) < \infty.$

\end{document}